\documentclass[a4paper,10pt,oneside]{article}
\usepackage{graphicx}

\usepackage{hyperref}
\usepackage{xfrac}
\usepackage{color}

\usepackage[multiple]{footmisc}
\usepackage{amsmath, amsthm, slashed}
\usepackage{amsfonts, slashed, cancel}
\usepackage{amssymb}
\usepackage{rotating}\usepackage{epsfig}
\usepackage{verbatim}
\usepackage{rotating}
\usepackage{graphicx}
\usepackage{amsmath,amsthm,amssymb,verbatim,a4}
\usepackage{mathrsfs,stmaryrd}
\usepackage[a4paper]{geometry}
\usepackage{marginnote}
\newtheorem{theorem}{Theorem}

\usepackage{chngcntr}
\newcounter{dummy}
\counterwithin{dummy}{section}
\newtheorem{definition}[dummy]{Definition}

\newtheorem{proposition}[dummy]{Proposition}
\newtheorem{lemma}[dummy]{Lemma}
\newtheorem{remark}[dummy]{Remark}

\usepackage{pifont}
\usepackage{fourier}
\usepackage{color}
\newcounter{mnotecount}[section]

\newcommand{\zz}{\mathfrak z}
 
\newcommand{\ud}{\text{d}}

\newcommand{\R}{\mathbb{R}}

\newcommand{\K}{\mathbb{K}}
\newcommand{\PP}{\mathbb{P}}

\newcommand{\T}{\mathbf{T}}
\newcommand{\al}{\alpha}
\newcommand{\be}{\beta}
\newcommand{\ab}[1]{|#1|}
\newcommand{\Ab}[1]{\|#1\|}

\newcommand{\eq}[1]{\begin{equation}#1\end{equation}}
\newcommand{\alg}[1]{\begin{aligned}#1\end{aligned}}

\newcommand{\p}[1]{\partial_{#1}}
\newcommand{\pv}[1]{\partial_{v^{#1}}}

\date{}

\author{David Fajman\footnote{Gravitational Physics,
Faculty of Physics,
University of Vienna,
Boltzmanngasse 5,
1090 Wien,
Austria.}\,\,,
J\'er\'emie Joudioux\footnote{Gravitational Physics,
Faculty of Physics,
University of Vienna,
Boltzmanngasse 5,
1090 Wien, Austria.}\, and
Jacques Smulevici
\footnote{Laboratoire de Math\'ematiques, Universit\'e Paris-Sud 11, B\^at. 425, 91405 Orsay, France.}}

\title{A vector field method for relativistic transport equations with applications}

\newcommand\blfootnote[1]{%
  \begingroup
  \renewcommand\thefootnote{}\footnote{#1}%
  \addtocounter{footnote}{-1}%
  \endgroup
}

\begin{document}
\maketitle

\begin{abstract}
We adapt the vector field method of Klainerman to the study of relativistic transport equations. First, we prove robust decay estimates for velocity averages of solutions to the relativistic massive and massless transport equations, without any compact support requirements (in $x$ or $v$) for the distribution functions. In the second part of this article, we apply our method to the study of the massive and massless Vlasov-Nordstr\"om systems. In the massive case, we prove global existence and (almost) optimal decay estimates for solutions in dimensions $n \ge 4$ under some smallness assumptions. In the massless case, the system decouples and we prove optimal decay estimates for the solutions in dimensions $n \ge 4$ for arbitrarily large data, and in dimension $3$ under some smallness assumptions, exploiting a certain form of the null condition satisfied by the equations. The $3$-dimensional massive case requires an extension of our method and will be treated in future work.\blfootnote{University of Vienna preprint ID: UWThPh-2015-27}
\end{abstract}

\tableofcontents

\section{Introduction}
The vector field method of Klainerman \cite{sk:udelicw} provides powerful tools which are at the core of many fundamental results in the study of non-linear wave equations, such as the famous proof of the stability of the Minkowski space \cite{ck:nlsms}. In essence, the method takes advantage of the symmetries of a linear evolution equation to derive in a robust way boundedness and decay estimates of solutions. The robustness is crucial, as the final aim is typically to prove the non-linear stability of some stationary solution, so that the method should be stable when perturbed by the non-linearities of the equations.

In this paper, we are interested in the massive and massless relativistic transport equations\footnote{We will be using in the whole article the Einstein summation convention. For instance, $v^i \partial_{x^i}$ in \eqref{eq:msv} stands for $\displaystyle \sum_{i=1}^n v^i \partial_{x^i}$.}
\begin{equation} \label{eq:msv}
\T_m(f)\equiv\left(\sqrt{m^2+\ab{v}^2}\p t+v^i\partial_{x^i}\right)\left(f\right)=0,
\end{equation}
where $m \ge 0$ is the mass\footnote{In the remainder of this article, we will often normalize the mass to be either $1$ or $0$ and thus consider mostly $\T_1$ and $\T_0$. We will however sometimes keep the mass $m>0$ so that the reader can see how some of the estimates would degenerate as $m\rightarrow0$.} of the particles and $f$ is a function of $(t,x,v)$ defined on $\mathbb R_t\times\mathbb R_x^n\times\mathbb R_v^n$ if $m>0$ and $\mathbb R_t\times\mathbb R_x^n\times \left(\mathbb R_v^n \setminus \{ 0 \} \right)$ otherwise, with $n \ge 1$ being the dimension of the physical space.

\subsubsection*{Decay estimates via the method of characteristics for relativistic transport equations}
For transport equations, the standard method to prove decay estimates is the method of characteristics. The origin of these decay estimates goes back in the non-relativistic case to the work of Bardos-Degond \cite{bd:gevp} on the Vlasov-Poisson system. Recall that if $f$ is a regular solution to say $\T_1(f)=0$ then, for all $(t,x,v) \in \mathbb R_t\times\mathbb R_x^n\times\mathbb R_v^n$, $$f(t,x,v)=f\left(0,x-\frac{vt}{\sqrt{1+|v|^2}},v\right),$$ and assuming that $f$ has initially compact support in $v$, one can easily infer the velocity average estimate, for all $t > 0$ and all $x \in \mathbb{R}^n$,
\begin{equation}\label{ineq:bd}
\int_{v \in \mathbb{R}^n} |f|(t,x,v) dv \lesssim \frac{C(V)}{t^n}\left|\left| f(t=0) \right|\right|_{L^1_x L^\infty_v},
\end{equation}
where $V$ is an upper bound on the size of the support in $v$ of $f$ at the initial time and $C(V) \rightarrow +\infty$ as $V \rightarrow +\infty$.

These estimates, while being relatively easy to derive, suffer from two important drawbacks when applied to a non-linear system.
\begin{enumerate}
\item They require compact support in $v$ of the solutions. In a non-linear setting, one therefore needs to bound an extra quantity, the size of this support at $t$. In particular, we are enlarging the number of variables of the system. Moreover, there are many interesting models where the correct assumptions from the point of view of physics\footnote{See for instance the end of the introduction in \cite{cv:ld}.} is to allow arbitrary large $v$. 
\item They require a strong control of the characteristics of the system. 
\end{enumerate}
These two inconveniences are in fact related and the first can be somehow mitigated by an even finer analysis of the characteristics, see for instance \cite{js:sdtcphv}. Concerning the second problem, we note that there are many evolution problems for which the characteristics in a neighbourhood of a stationary solution will eventually diverge from the original ones, introducing extra difficulties in the analysis. A famous example of that is the stability of Minkowski space, where there is a logarithmic divergence, see \cite{ck:nlsms, lr:gsmhg}. Moreover, to prove decay estimates such as \eqref{ineq:bd}, one needs to control the Jacobian associated with the differential of the characteristic flow\footnote{In the context of the Vlasov equation on a curved Lorentzian manifold, this means that one needs estimates on the differential of the exponential map, or at least on its restriction to certain sub-manifolds.} and in order to obtain improved decay estimates for derivatives, one also needs estimates on the derivatives of the Jacobian. See for instance \cite{hwang11} where such a program is carried out for the Vlasov-Poisson system. In other words, one needs strong control on the characteristics to be able to prove sharp decay estimates via this method in a non-linear setting. 

\subsubsection*{Decay estimates for the wave equation}
In the context of the wave equation 
$$\square \phi\equiv\left[-\partial_t^2+ \sum_{i=1}^n \partial_{x^i}^2\right]\phi=0,$$
several methods exist to prove decay estimates of solutions. For instance, one standard way is to use the Fourier representation of the solution together with estimates for oscillatory integrals. In the fundamental article \cite{sk:udelicw}, Klainerman introduced what is now referred to as \emph{the vector field method}\footnote{Let us also mention that, complementary to the method of Klainerman, which uses vector fields as commutators, one can also use vector fields as \emph{multipliers}, in the style of the work of Morawetz, see for instance \cite{cm:lap, cm:tdkg}.}. Instead of relying on an explicit integral representation of the solutions, it uses
\begin{enumerate}
\item a coercive conservation law. In the case of the wave equation, this is simply the conservation of the energy.
\item a set of vector fields which commute with the equations. In the case of the wave equation, these are the Killing and conformal Killing fields of the Minkowski space.
\item weighted Sobolev $L^2-L^\infty$ inequalities: the standard derivatives $\partial_t$, $\partial_{x^i}$ are rewritten in terms of the commutator vector fields before applying the usual Sobolev inequalities. The weights in these decompositions together with those arising from the conservation laws are then translated into decay rates.
\end{enumerate}
This leads to the decay estimate
\begin{equation} \label{ineq:dew}
|\partial \phi(t,x)| \lesssim \frac{\mathcal{E}^{1/2}(\phi)}{(1+|t-|x|\,|)^{1/2} (1+|t+|x|\,|)^{(n-1)/2}}
\end{equation}
 for solutions of the wave equation $\square \phi=0$, where $E(\phi)$ is an energy norm obtained by integrating $\phi$ and derivatives of $\phi$ (with weights) at time $t=0$. 

These types of estimates, being based on conservation laws and commutators, are quite robust, and as a consequence, are applicable in strongly non-linear settings, such as the Einstein equations or the Euler equations (see for instance \cite{dc:fs3d} for such an application). 

\subsubsection*{A vector field method for transport equations}
In our opinion, the decay estimate \eqref{ineq:bd}, being based on an explicit representation of the solutions, that given by the method of characteristics, should be compared to the decay estimates for the wave equation obtained via the Fourier or other integral representations. In this paper, we derive an analogue of the vector field method for the massive and massless relativistic transport equations \eqref{eq:msv}. The coercive conservation law is given by the conservation of the $L^1$ norm of the solution, while the vector fields commuting with the operators are essentially obtained by taking the \emph{complete lifts} of the Killing and conformal Killing fields, a classical operation in differential geometry which takes a vector field on a manifold $M$ to a vector field on the tangent bundle $TM$. The weighted Sobolev inequalities are slightly more technical. One of the main ingredients is that averages in $v$ possess good commutation properties with the Killing vector fields and their complete lifts. Our decay estimates can then be stated as
\begin{theorem}[Decay estimates for velocity averages of massless distribution functions]\label{th:demsl}
For any regular distribution function $f$, solution to $\T_0(f)=0$ and any $(t,x) \in \mathbb{R}_t^+\times \mathbb{R}_x^{n}$, we have
\begin{equation}\label{es:msl}
\int_{v\in\mathbb{R}^n_v \setminus \{0 \}} |f|(t,x,v) \frac{dv}{|v|} \lesssim \frac{1}{(1+|t-|x|\,|) (1+|t+|x|\,|)^{n-1}} \sum_{\substack{ |\alpha| \le n,\\ \widehat{Z}^\alpha \in \widehat{\K}^{|\alpha|}}} \left|\left| |v|^{-1}\widehat{Z}^\alpha(f)(t=0) \right|\right|_{L^1\left( \mathbb{R}^n_x \times \left(\mathbb{R}^n_v \setminus \{0 \}\right)\right)},
\end{equation}
where the $\alpha$ are multi-indices of length $|\alpha|$ and the $\widehat{Z}^\alpha$ are differential operators of order $|\alpha|$ obtained as a composition of $|\alpha|$ vector fields of the algebra $\widehat{\K}$. 
\end{theorem}
The detailed list of the vector fields and their complete lifts used here is given in Section \ref{sec:completelift}. For the massive transport equation, we prove
\begin{theorem}[Decay estimates for velocity averages of massive distribution functions]\label{th:demsv}
For any regular distribution function $f$, solution to $\T_1(f)=0$, any $x \in \mathbb{R}^{n}$ and any $t \ge \sqrt{1+|x|^2}$, we have
\begin{equation}\label{es:msv}
\int_{v\in\mathbb{R}^n_v} |f|(t,x,v) \frac{dv}{\sqrt{1+|v|^2}} \lesssim \frac{1}{\left(1+t \right)^{n}} \sum_{\substack{ |\alpha| \le n,\\ \widehat{Z}^\alpha \in \widehat{\PP}^{|\alpha|}}} \left|\left| \widehat{Z}^\alpha(f)_{| H_1 \times \mathbb{R}^n_v}v_\alpha \nu_1^\alpha \right|\right|_{L^1\left( H_1 \times \mathbb{R}^n_v \right)},
\end{equation}
where $H_1$ denotes the unit hyperboloid $H_1 \equiv \left\{(t,x) \in \mathbb{R}_t \times \mathbb{R}^n_x\, \big |\,\, 1=t^2-x^2 \right\}$, $\widehat{Z}^\alpha(f)_{| H_1 \times \mathbb{R}^n_v}$ is the restriction to $H_1 \times \mathbb{R}^n_v$ of $\widehat{Z}^\alpha(f)$, $v_\alpha\nu_1^\alpha$ is the contraction of the $4$-velocity $(\sqrt{1+|v|^2},v^i)$ with the unit normal $\nu_1$ to $H_1$ and where the $\widehat{Z}^\alpha$ are differential operators obtained as a composition of $|\alpha|$ vector fields of the algebra $\widehat{\PP}$. 
\end{theorem}

\begin{remark}
No compact support assumptions in $x$ or in $v$ are required for the statements of Theorems \ref{th:demsl} or \ref{th:demsv}. Of course, for the norms on the right-hand sides of \eqref{es:msl} or \eqref{es:msv} to be finite, some amount of decay in $x$ and $v$ is needed. Note that from the point of view of non-linear applications, it is sufficient to propagate bounds for the norms appearing on the right-hand sides of \eqref{es:msl} or \eqref{es:msv}, without any need to control pointwise the decay in $x$ or $v$ of the solutions, to get the desired decay estimates for the velocity averages.
\end{remark}

\begin{remark}
In \eqref{es:msl}, the decay is worse near the light-cone $t=|x|$. This of course is an analogue of the decay estimate \eqref{ineq:dew} as traditionally obtained for the wave equation by the vector field method.
\end{remark}
\begin{remark}
\eqref{es:msv} is the analogue of the decay estimate for Klein-Gordon fields $\phi$, solutions to $\square \phi=\phi$, for which, for all $t \ge \sqrt{1+|x|^2}$,
$$|\partial \phi(t,x)| \lesssim \frac{\mathcal{E}^{1/2}[\phi]}{ (1+t)^{n/2}},$$ 
where $E[\phi]$ is an energy norm obtained by integrating $\phi$ and derivatives of $\phi$ (with weights) on an initial hyperboloid, as obtained by Klainerman in \cite{MR1199196}. 
\end{remark}
\begin{remark}
As in the case of the Klein-Gordon equation, one can easily prove that for regular solutions $f$ to $\T_1(f)=0$ with data given at $t=0$ and decaying sufficiently fast as $|x| \rightarrow +\infty$ (in particular, solutions arising from data with compact support in $x$) the norm on the right-hand side of \eqref{es:msv} is finite, so that the decay estimate applies\footnote{See also \cite{vg:dekg}, where decay estimates for the Klein-Gordon operator were obtained starting from non-compactly supported data at $t=0$ using (mostly) vector field type methods.}. Thus, the use of hyperboloids is merely a technical issue. The restriction ``$t \ge \sqrt{1+|x|^2}$'' simply means in the future of the unit hyperboloid. We provide a classical construction in Appendix \ref{se:csmsv}, which explains how Theorem \ref{th:demsv} can be applied to solutions arising from initial data with compact $x$-support given at $t=0$ to obtain a $1/t^n$ decay of velocity averages in the whole future of the $t=0$ hypersurface.
\end{remark}

\begin{remark}The reader might wonder whether the same type of techniques could be applied for the classical transport operator $\T_{cl}=\partial_t+v^i \partial_{x^i}$. This question was addressed in \cite{js:sdsvpsvfm} where decay estimates for velocity averages of solutions to the classical transport operator were obtained. As an application, \cite{js:sdsvpsvfm} considered the study of small data solutions of the Vlasov-Poisson system and provided an alternative proof (with some additional information on the asymptotic behaviour of the solutions, concerning in particular the decay in $|x|$ and uniform bounds on some global norms) of the optimal time decay for derivatives of velocity averages obtained first in \cite{hwang11}. One of the nice features of the vector field method is that improved decay estimates for derivatives follow typically easily from the main estimates, and \cite{js:sdsvpsvfm} was no exception. In the relativistic case, our vector field method also provides improved decay for derivatives. See Propositions \ref{pro:idmsl} and \ref{pro:idmsv} in Sections \ref{se:ksmsl} and \ref{se:ksmsv}, respectively. 
\end{remark}

\subsubsection*{Applications to the massive and massless Vlasov-Nordstr\"om systems}
In the second part of this paper, we will apply our vector field method to the massive and massless Vlasov-Nordstr\"om systems

\begin{eqnarray} \label{eq:vnp}
\square \phi&=& m^2 \int_v f \frac{dv}{\sqrt{m^2+|v|^2}}, \\
\T_m(f)-\left(\T_m(\phi) v^i +m^2 \nabla^i \phi \right)\frac{\partial f}{\partial v^i}&=&(n+1) f\, \T_m (\phi), \label{eq:vnf}
\end{eqnarray}
where $m=0$ in the massless case and $m>0$ in the massive case, $\square \equiv-\partial_t^2+\sum_{i=1}^n \partial_{x^i}^2$ is the standard wave operator of Minkowski space, $\phi$ is a scalar function of $(t,x)$ and $f$ is as before a function of $(t,x,v^i)$ with $x \in \mathbb{R}^n$, $v \in \mathbb{R}^n$ if $m> 0$, $v \in \mathbb{R}^n\setminus \{0 \}$ if $m=0$. A good introduction to this system can be found in \cite{MR1981446}. See also the classical works \cite{MR2238881,MR2194582}.

Roughly speaking, the Vlasov-Nordstr\"om system can be derived from the Einstein-Vlasov system by considering only a special class of solutions (that of metrics conformal to the Minkowski metric) and by neglecting some of the non-linear interactions in the (Einstein part of) the equations. Since most of the simplifications concern difficulties which we already know how to handle (in the style of \cite{ck:nlsms} or \cite{lr:gsmhg}) and since the method that we are using here is of the same type as the one used to study the Einstein vacuum equations, we believe it is a good model problem before addressing the full Einstein-Vlasov system via vector field methods. 

Before presenting our main results for the massive and massless Vlasov-Nordstr\"om systems, let us explain the main differences between the $m=0$ and $m>0$ cases. First, as easily seen from \eqref{eq:vnp}-\eqref{eq:vnf}, when $m=0$, the system degenerates to a partially decoupled system\footnote{In fact, using $e^{-(n+1) \phi} f$ as an unknown, we can obtain an even simpler form of the equations where the right-hand side of \eqref{eq:vnf0} is put to $0$. See \eqref{eq:tfmsl}.}
\begin{eqnarray}
\square \phi&=&0, \label{eq:vnp0}\\
\T_0(f)-\left(\T_0(\phi) v^i \right)\frac{\partial f}{\partial v^i}&=&(n+1) f\, \T_0 (\phi). \label{eq:vnf0}
\end{eqnarray}
Because of the decoupling, the first equation is simply the wave equation on Minkowski space and the second can be viewed as a linear transport equation, where the transport operator is the massless relativistic transport operator plus some perturbations. In particular, all solutions are necessarily global as long as the initial data is sufficiently regular so that the linear equations can be solved. Thus, our objective is solely to derive sharp asymptotics for the solutions of the transport equation. Moreover, since we have in mind future applications to more non-linear problems, the only estimates that we will use on $\phi$ will be those compatible with what can be derived via a standard application of the vector field method.

Apart from the decoupling just explained, let us mention also two important pieces of structure present in the above equations. First, another great simplification comes from the existence of an \emph{extra scaling symmetry} present only in the massless case: the vector field $v^i \partial_{v^i}$ commutes with the massless transport operator $\T_0$ and it is precisely this combination of derivatives in $v$ which appears in the equation \eqref{eq:vnf0}. This fact will make all the error terms obtained after commutations much better than if a random set of derivatives in $v$ was present in \eqref{eq:vnf0}. Another property of \eqref{eq:vnp0}-\eqref{eq:vnf0} is the existence of a \emph{null structure}, similar to the null structure of Klainerman for wave equations. More precisely, we show that $\T_0(\phi)$ has roughly the structure 
$$
\T_0(\phi)\simeq |v|\bar{\partial} \phi+ \frac{1}{t} \partial \phi\cdot z(t,x,v),
$$ 
where $\bar{\partial} \phi$ denotes derivatives tangential to the outgoing cone, $\partial \phi$ denotes arbitrary derivatives of $\phi$ and $z(t,x,v)$ are weights which are bounded along the characteristics of the linear massless transport operator. Since $\bar{\partial} \phi$ has better decay properties than a random derivative $\partial \phi$, we see that products of the form $\T_0(\phi)g$, where $g$ is a solution to $\T_0(g)=0$ have better decay properties than expected\footnote{It is interesting to compare this form of the null condition to the one uncovered in \cite{md:ncsdsgm} for the massless Einstein-Vlasov system in spherical symmetry. In fact, two null conditions were used there. The obvious one consists essentially in understanding why null components of the energy momentum tensor of $f$ decay better than expected. A more secret null condition is used in the analysis of the differential equation satisfied by the part of the velocity vector tangent to the outgoing cone. Our null condition is closely related to this one, even though we exploit it in a different manner since we are not using directly the characteristic system of ordinary differential equations associated with the transport equations.}. Similar to the study of $3$ dimensional wave equations with non-linearities satisfying the null condition, the extra decay obtained means that in dimension $3$ (or greater), all the error terms in the (approximate) conservation laws are now integrable.

We now state our main results for the massless Vlasov-Nordstr\"om system.
\begin{theorem}[Asymptotics in the massless case for dimension $n \ge 4$] \label{th:asmsl4}
Let $n \ge 4$ and $N \ge \frac{3n}{2}+1$. Let $\phi$ be a solution of \eqref{eq:vnp0} satisfying $\phi(t=0)=\phi_0, \partial_t \phi(t=0)=\phi_1$ for some sufficiently regular functions $(\phi_0,\phi_1)$. Then, if $\mathcal{E}_{N}[\phi_0, \phi_1] < +\infty$, where $\mathcal{E}_N[\phi_0, \phi_1]$ is an energy norm containing up to $N$ derivatives of $(\partial\phi_0,\phi_1)$ and if $E_{N}[f_0]<+\infty $, where $E_{N}[f_0]$ is a norm containing up to $N$ derivatives of $f_0$, then the unique classical solution $f$ to \eqref{eq:vnf0} satisfying $f(t=0)=f_0$ verifies
\begin{enumerate}
\item Global bounds: for all $t \ge 0$, 
$$
E_N[f](t) \le e^{C\mathcal{E}_N^{1/2}[\phi_0, \phi_1]}E_N[f_0],
$$
where $C>0$ is a constant depending only on $N,n$.
\item Pointwise estimates for velocity averages: for all $(t,x) \in [0,+\infty) \times \mathbb{R}^n_x$ and all multi-indices $\alpha$ satisfying $|\alpha| \le N-n$, 
$$
\int_{v \in \mathbb{R}^n_v \setminus \{ 0 \}}\left| \widehat{Z}^\alpha f \right| (t,x,v) |v| dv  \lesssim \frac{e^{C\mathcal{E}_N^{1/2}[\phi_0, \phi_1]}E_N[f_0]} {(1+|t-|x|\,|) (1+|t+|x|\,|)^{n-1}}.
$$
\end{enumerate}
\end{theorem}

In dimension $n \ge 3$, similar results can be obtained provided an extra smallness assumption on the initial data for the wave function as well as stronger decay for the initial data of $f$ hold.
\begin{theorem}[Asymptotics in the massless case for dimension $n=3$] \label{th:asmsl}
Let $n \ge 3$, $N  \ge 7$ and $q \ge 1$. Let $\phi$ be a solution of \eqref{eq:vnp0} satisfying $\phi(t=0)=\phi_0, \partial_t \phi(t=0)=\phi_1$ for some sufficiently regular functions $(\phi_0,\phi_1)$. Then, if $\mathcal{E}_{N}[\phi_0, \phi_1] \le \varepsilon$, where $\mathcal{E}_N[\phi_0, \phi_1]$ is an energy norm containing up to $N$ derivatives of $(\partial\phi_0,\phi_1)$ and if $E_{N,q}[f_0]<+\infty $, where $E_{N,q}[f_0]$ is a norm\footnote{The index $q$ refers to powers of certain weights. See \eqref{def:n3d} for a precise definition of the norms.}  containing up to $N$ derivatives of $f_0$, then the unique classical solution $f$ to \eqref{eq:vnf0} satisfying $f(t=0)=f_0$ verifies 
\begin{enumerate}
\item Global bounds with loss: for all $t\ge 0$,
$$
E_{N,q}[f](t) \le (1+t)^{C \varepsilon^{1/2}}E_{N,q}[f_0],
$$
where $C>0$ depends only on $N,n$.
\item Improved global bounds for lower orders: for any $M \le N-(n+2)/2$ and any $t \ge 0$,  
$$
E_{M,q-1}[f](t) \le e^{C \varepsilon^{1/2}}E_{N,q}[f_0].
$$
\item Pointwise estimates for velocity averages: for all $(t,x) \in [0,+\infty) \times \mathbb{R}^n_x$ and all multi-indices $|\alpha| \leq  N-(3n+2)/2$, 
$$
\int_{v\in\mathbb{R}^n_v \setminus \{0 \}} |\widehat{Z}^\alpha f|(t,x,v) |v| dv \lesssim \frac{E_{N,q}[f_0]}{(1+|t-|x|\,|) (1+|t+|x|\,|)^{n-1}}.
$$
\end{enumerate}
\end{theorem}

Perhaps counter-intuitively, the massive case turned out to be harder to treat. While it is true that in the massive case, the pointwise decay of velocity averages is not weaker along the null cone, there are two important extra difficulties, namely
\begin{itemize}
\item The equations are now fully coupled. In particular, one cannot close an energy estimate for \eqref{eq:vnp} unless we have some decay for the right-hand side. On the other hand, our decay estimates, being based on commutators, necessarily lose some derivatives. In turns, this would imply commuting \eqref{eq:vnf0} more, but we would then fail to close the estimates at the top order. We resolve this issue by another decay estimate for inhomogeneous transport equations with rough source terms satisfying certain product structures. This other type of decay estimates only provides $L^2_x$ time decay of the velocity averages, which is precisely what is required to close the energy estimate for \eqref{eq:vnp}. The proof of this $L^2_x$ decay estimate itself can be reduced to our $L^\infty$ estimates, so that it can also be obtained using purely vector field type methods.
\item The vector field $v^i \partial_{v^i}$ does not commute with the massive transport operator. This implies that commuting with (some of) the standard vector fields will lose a power of $t$ of decay compared to the massless case. 
\end{itemize}
Because of the last issue, the results that we will present here are restricted to dimension $n \ge 4$. One way to treat the $3$-dimensional case would be to improve upon the commutation formulae to eliminate the most dangerous terms. For instance, one could try to use \emph{modified vector fields}  in the spirit of \cite{js:sdsvpsvfm}. We plan to address the $3$ dimensional case in future work.

A slightly more technical consequence of this last issue is that it introduces $t$ weights in the estimates, which are not constant on the leaves of the hyperboloidal foliation that we wish to use. Together with the fact that the energy estimates are weaker on hyperboloids, this implies that the error terms arising in the top-order approximate conservation laws can be shown to be space-time integrable only in dimension $n \ge 5$. 
To address the $n=4$ dimension, instead of estimating directly $\widehat{Z}^\alpha(f)$, where $f$ is the unknown distribution functions and $\widehat{Z}^\alpha$ is a combination of $\alpha$ vector fields, we estimate instead a renormalized quantity of the form $\widehat{Z}^\alpha(f)+g^\alpha$ where the $g^\alpha$ is a (small) non-linear term constructed from the solution. The extra terms appearing in the equation when the transport operator hits $g^\alpha$ will then cancel some of the worst terms in the equations.

Our main result in the massive case can then be stated as follows. 

\begin{theorem} \label{th:vnm4d}
Let $n \ge 4$ and $m>0$. Let $N \in \mathbb{N}$ be sufficiently large depending only on $n$. For any $\rho \ge 1$, denote by $H_\rho$ the hyperboloid
$$
H_\rho=\left \{(t,x)\in \mathbb{R}_t \times \mathbb{R}_x^n\, \big | \, \rho^2=t^2-x^2 \right \}.
$$

For any sufficiently regular function $\psi$ defined on $\mathbb{R}_t \times \mathbb{R}_x^n$, denote by $\psi_{|H_\rho}$ its restriction to $H_\rho$.
 Similarly, for any sufficiently regular function $g$ defined on $\bigcup_{1 \le \rho \le +\infty}H_\rho \times \mathbb{R}^n_v$, denote by $g_{|H_\rho \times \mathbb{R}^n_v}$ its restriction to $H_\rho \times \mathbb{R}^n_v$.
Then, there exists an $\varepsilon_0 > 0$ such that for any $0 < \varepsilon \le \varepsilon_0$, if $E_{N+n}[f_0]+\mathcal{E}_{N}[\phi_0, \phi_1] \le \varepsilon$, where $E_{N+n}[f_0]$ and $\mathcal{E}_{N}[\phi_0, \phi_1]$ are norms depending on respectively $N+n$ derivatives of $f_0$ and $N$ derivatives of $(\partial\phi_0, \phi_1)$, then there exists a unique classical solution $(f,\phi)$ to \eqref{eq:vnp}-\eqref{eq:vnf} satisfying the initial conditions
\begin{eqnarray*}
\phi_{|H_1}&=&\phi_0, \quad \partial_t \phi_{|H_1}=\phi_1, \\
f_{|H_1 \times \mathbb{R}^n_v}&=&f_0,
\end{eqnarray*}
such that $(f,\phi)$ exists globally\footnote{Here, globally means at every point lying in the future of the initial hyperboloid $H_1$. In $3d$, this, of course, would already follow from the work \cite{MR2238881} for regular initial data of compact support given on a constant time slice.} and satisfies the estimates
\begin{enumerate}
\item Global bounds, for all $\rho \ge 1$,
$$
\mathcal{E}_N[\phi](\rho)\lesssim \varepsilon \text { and } E_N[f][\rho] \lesssim \varepsilon \rho^{C\varepsilon^{1/4}},
$$
where $C=1$ when $n=4$ and $C=0$ when $n>4$.
\item Pointwise decay for $ \partial Z^\alpha\phi$: for all multi-indices $|\alpha|$ such that $|\alpha | \le N-\tfrac{n+2}{2}$ and all $(t,x)$ with $t \ge \sqrt{1+|x|^2}$, we have 
$$
\left| \partial Z^\alpha\phi \right| \lesssim \frac{\varepsilon}{(1+t)^{(n-1)/2}(1+(t-|x|)^{1/2}}.
$$
\item Pointwise decay for $ \rho\left( \left|\partial Z^\alpha f \right| \right)$: for all multi-indices $\alpha$ and $\beta$ such that $|\alpha | \le N-n $, $|\beta|\leq \lfloor N/2\rfloor-n $ and all $(t,x)$ with $t \ge \sqrt{1+|x|^2}$, we have 
$$
\int_v \left|\widehat{ Z}^\alpha f \right| \frac{dv}{v^0} \lesssim \frac{\varepsilon}{(1+t)^{n-C\varepsilon^{1/4}}} \text{ and } \int_v v^0\big|\widehat{ Z}^\beta f \big| dv \lesssim \frac{\varepsilon}{(1+t)^{n-C\varepsilon^{1/4}}},
$$
where $C=1$ when $n=4$ and $C=0$ when $n>4$.
\item Finally, the following $L^2$-estimates on $f$ hold: for all multi-indices $\alpha$ with $\lfloor N/2\rfloor-n +1 \le |\al| \le N$, and all $(t,x)$ with $t \ge \sqrt{1+|x|^2}$, we have 
$$
\int_{H_\rho} \frac{t}{\rho}\left(\int_{v}|\widehat{Z}^{\al} f| \frac{dv}{v^0}\right)^{2} d\mu_{H_\rho} \lesssim \varepsilon^2 \rho^{C\varepsilon^{1/4}-n},
$$
where $C=2$ when $n=4$ and $C=0$ when $n>4$.
\end{enumerate}
\end{theorem}

\begin{remark}As for the linear decay estimates of Theorem \ref{th:demsv}, it is not essential to start on an initial hyperboloid for the conclusions of Theorem \ref{th:vnm4d} to hold. In particular, an easy argument based on finite speed of propagation, similar to that given in Appendix \ref{se:csmsv}, shows that our method and results apply to the case of sufficiently small initial data with compact $x$-support given at $t=0$.
\end{remark}

\begin{remark}In \cite{sf:gssvns}, solutions of the massive Vlasov-Nordstr\"om system in dimension $3$ arising from small, regular, compactly supported (in $x$ and $v$) data given at $t=0$ were studied and the asymptotics of velocity averages of the Vlasov field and up to two derivatives of the wave function were obtained. However, no estimates were obtained for derivatives of the Vlasov field or for higher derivatives of the wave function. Thus, \cite{sf:gssvns} is the analogue of \cite{bd:gevp} for the Vlasov-Nordstr\"om system while we obtained here (in dimension $4$ and greater) results more in the spirit of \cite{hwang11, js:sdsvpsvfm}.
\end{remark}

\begin{remark}A posteriori, it is straightforward to propagate higher moments of the solutions in any of the situations of Theorems \ref{th:asmsl4}, \ref{th:asmsl}, \ref{th:vnm4d}, provided that these moments are finite initially. Moreover, we recall that improved decay for derivatives of $f$ and $\phi$ follows from the statements of Theorems \ref{th:asmsl4}, \ref{th:asmsl} and \ref{th:vnm4d}. See for instance Propositions \ref{pro:idmsl} and \ref{pro:idmsv} below.
\end{remark}

\subsubsection*{Aside: the Einstein-Vlasov system}
As explained above, the Vlasov-Nordstr\"om system is a model problem for the more physically relevant Einstein-Vlasov system. We refer to the recent book\footnote{Apart from a general introduction to the Einstein-Vlasov system, the main purpose of this book is to present a proof of stability of exponentially expanding spacetimes for the Einstein-Vlasov system, see \cite{MR3186493}.} \cite{MR3186493} for a thorough introduction to this system. The small data theory around the Minkowski space is still incomplete for the Einstein-Vlasov system. The spherically symmetric cases in dimension $(3+1)$ have been treated in \cite{rr:gesssvsssd, rr:err} for the massive case and in \cite{md:ncsdsgm} for the massless case with compactly supported initial data.
A proof of stability for the massless case without spherical symmetry but with compact support in both $x$ and $v$ has recently been announced in \cite{mt:smsmevs}. As in \cite{md:ncsdsgm}, the compact support assumptions and the fact that the particles are massless are important as they allow to reduce the proof to that of the vacuum case outside from a strip going to null infinity. Interestingly, the argument in \cite{mt:smsmevs} is quite geometric, relying for instance on the double null foliation, in the spirit of \cite{kn:epgr}, as well as several structures associated with the tangent bundle of the tangent bundle of the base manifold.

We hope to address the stability of the Minkowski space for the Einstein-Vlasov system in the massive and massless case (without the compact support assumptions) using the method developed in this paper in future works.

\subsubsection*{Structure of the article}
Section \ref{se:pre} contains preliminary materials, such as basic properties of the transport operators, the definition and properties of the foliation by hyperboloids used for the analysis of the massive distribution function, the commutation vector fields and elementary properties of these vector fields. In Section \ref{se:vfvf}, we introduce the vector field method for relativistic Vlasov fields and prove Theorems \ref{th:demsl} and \ref{th:demsv}. In Section \ref{se:avn}, we apply our method first to the massless case in dimension $n \ge 4$ (Section \ref{se:mcd4}) and $n=3$ (Section \ref{se:mcd3}) and then to the massive case in dimension $ n \ge 4$ (Section \ref{se:msvcd4}). In Appendix \ref{se:csmsv}, we provide a classical construction which explains how our decay estimates in the massive case can be applied to data of compact support in $x$ given at $t=0$. Some integral estimates useful in the course of the paper are proven in Appendix \ref{app:ie}. Finally, Appendix \ref{se:gvf} contains a general geometrical framework for the analysis of the Vlasov equation on a Lorentzian manifold. 

\subsubsection*{Acknowledgements} We would like to thank Martin Taylor for several interesting discussions on his work on the massless Einstein-Vlasov system. We would also like to thank Olivier Sarbach for his geometric introduction to Vlasov fields. The second and third authors are partially funded by ANR-12-BS01-012-01 (AARG). The third author is also partially funded by ANR SIMI-1-003-01. The first author gratefully acknowledges the travel support by ANR-12-BS01-012-01 (AARG).

\section{Preliminaries} \label{se:pre}
\subsection{Basic notations}

Throughout this paper we work on the $n+1$-dimensional Minkowski space $\left(\mathbb{R}^{n+1},\eta\right)$, where the standard Minkowski metric $\eta$ is globally  defined in Cartesian coordinates $(t,x^i)$ by $\eta=\mathrm{diag}\{-1,1,\hdots,1\}$. 
We denote spacetime indices by Greek letters $\al,\beta,\hdots\in\{0,\hdots n\}$ and spatial indices by Latin letters $i,j,\hdots\in \{1,\hdots,n\}$. We will sometimes use $\partial_{x^\alpha}, \partial_t, \partial_{x^i}, \partial_{v^i} ,\hdots$  to denote the partial derivatives $\frac{\partial}{\partial x^\alpha}$, $\frac{\partial}{\partial t},\hdots$

Since we will be interested in either massive particles with $m=1$ or massless particles $m=0$, the velocity vector $(v^\beta)_{\beta=0,..,n}$ will be parametrized by $(v^i)_{i=1,..,n}$ and $v^0=|v|$ in the massless case, $v^0=\sqrt{1+|v|^2}$ in the massive case.

The indices $0$ and $m>0$ will be used to denote objects corresponding to the massless or massive case, such as the massless transport operator $\T_0$ and the massive one $\T_m$ and should not be confused with spatial or spacetime indices for tensor components (we use bold letters on the transport operators to avoid this confusion).

The notation $A \lesssim B$ will be used to denote an inequality of the form $A \le CB$, for some constant $C > 0$ independent of the solutions (typically $C$ will depend on the number of dimensions, the maximal order of commutations $N$, the value of the mass $m$).

\subsection{The relativistic transport operators}

For any $m > 0$ and any $v \in \mathbb{R}^n$, let us define the \emph{massive relativistic transport operator} $\T_{m}$ by 

\eq{\alg{ \label{eq:msvto}
\T_{m}=v^0\p t+v^i\partial_{x^i},\mbox{ with } v^0=\sqrt{m^2+\ab{v}^2}. 
}} 

Similarly, we define for any $v \in \mathbb{R}^3 \setminus \{ 0 \}$, the \emph{massless transport operator} $\T_0$ by
\eq{\alg{ \label{eq:mslto}
\T_{0}=v^0 \p t+v^i\partial_{x^i},\mbox{ with } v^0=|v|.
}}

For the sake of comparison, let us recall that the \emph{classical} transport operator is given by 
$$
\T_\mathrm{cl}=\p t+v^i\partial_{x^i}.
$$

In the remainder of this work, we will normalize the mass to be either $1$ or $0$, so that the massive transport operators we will study are
$$
\T_1=\sqrt{1+|v|^2}\p t+v^i\partial_{x^i},
$$
and
$$
\T_0=|v|\p t+v^i\partial_{x^i}.
$$

\subsection{The foliations}\label{se:tf}
We will consider two distinct foliations of (some subsets of) the Minkowski space.

Let us fix global Cartesian coordinates $(t,x^i)$, $1 \le i \le n$ on $\mathbb{R}^{n+1}$ and denote by $\Sigma_t$ the hypersurface of constant $t$. The hypersurfaces $\Sigma_t$, $t\in \mathbb{R}$ then give a complete foliation of $\mathbb{R}^{n+1}$.
The second foliation is defined as follows. For any $\rho > 0$, define $H_\rho$ by 
$$
H_{\rho}=\left\{ (t,x)\,\, \big |\,\, t \ge |x|\,\, \mathrm{and}\,\, t^2-|x|^2= \rho^2 \right\}.
$$
For any $\rho > 0$, $H_\rho$ is thus only one sheet of a two sheeted hyperboloid.
\begin{figure}[!h]
\center \includegraphics[width=8cm]{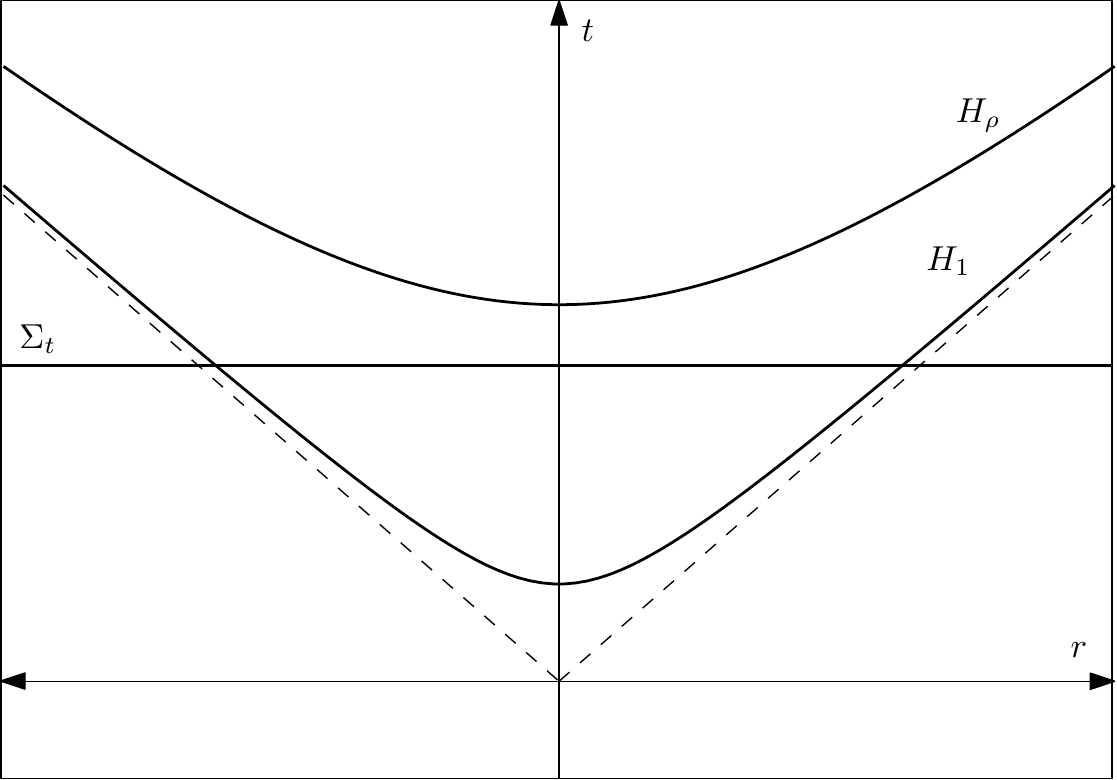}
\caption{The $H_\rho$ foliations in the $(t,r)$ plane, $\rho > 1$}
\end{figure}
\begin{figure}[!h]
\center \includegraphics[width=8cm]{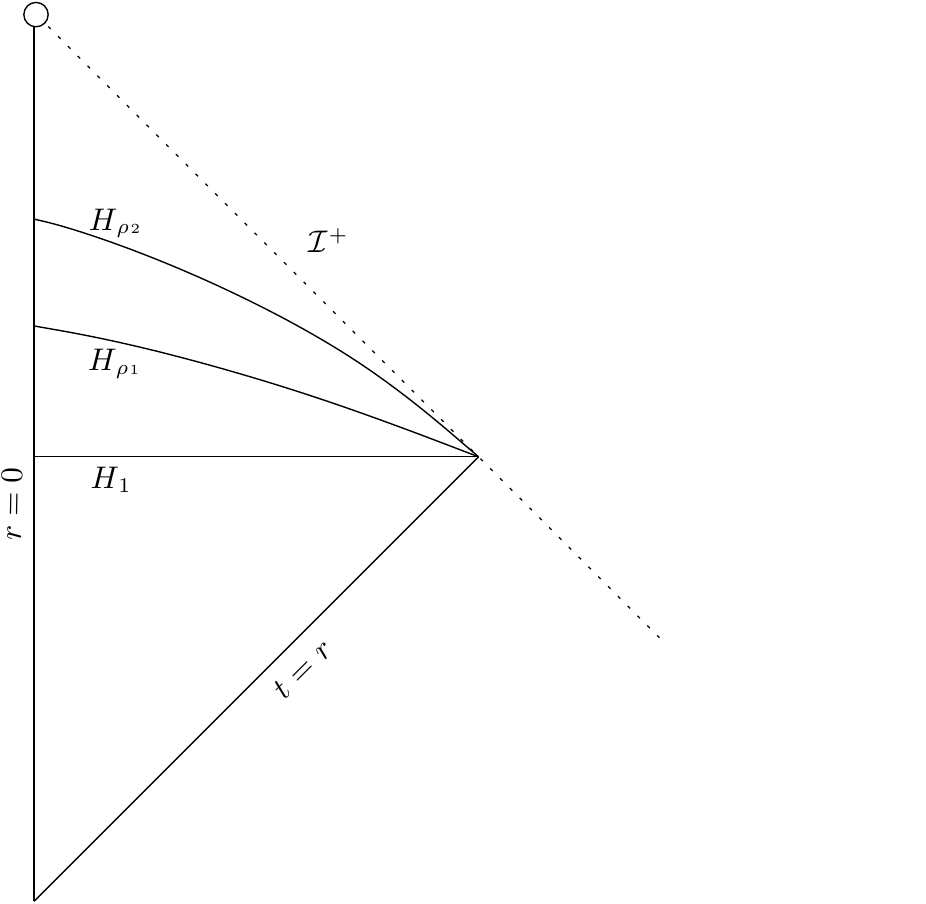}
\caption{The $H_\rho$ foliations in a Penrose diagram of Minkowski space, $\rho_2 > \rho_1 > 1$ }
\end{figure}

Note that 
$$\displaystyle \bigcup_{\rho \ge 1} H_\rho=\left\{ (t,x)\in \mathbb{R}^{n+1}\, \big|\, t \ge (1+|x|^2)^{1/2} \right\}.$$
The above subset of $\mathbb{R}^{n+1}$ will be referred to as \emph{the future of the unit hyperboloid}. On this set, we will use as an alternative to the Cartesian coordinates $(t,x)$ the following two other sets of coordinates.\\

\textbf{Spherical coordinates}\\
We first consider spherical coordinates $(r,\omega)$ on $\mathbb{R}^n_x$, where $\omega$ denotes spherical coordinates on the $n-1$ dimensional spheres and $r=|x|$. $(\rho,r,\omega)$ then defines a coordinate system on the future of the unit hyperboloid. These new coordinates are defined globally on the future of the unit hyperboloid apart from the usual degeneration of spherical coordinates and at $r=0$.\\

\textbf{Pseudo-Cartesian coordinates}\\
These are the coordinates $(y^0,y^j)\equiv(\rho,x^j)$. These new coordinates are also defined globally on the future of the unit hyperboloid.\\

For any function defined on (some part of) the future of the unit hyperboloid, we will move freely between these three sets of coordinates.

\subsection{Geometry of the hyperboloids} \label{se:gh}
The Minkowski metric $\eta$ is given in $(\rho,r,\omega)$ coordinates by
$$
\eta= -\frac{\rho^2}{t^2} \left( d\rho^2-dr^2 \right) -\frac{2 \rho r }{t^2} d\rho dr+ r^2 \sigma_{\mathbb{S}^{n-1}},
$$
where $\sigma_{\mathbb{S}^{n-1}}$ is the standard round metric on the $n-1$ dimensional unit sphere, so that for instance
$$
\sigma_{\mathbb{S}^2}=\sin \theta^2 d\theta^2+d\phi^2,
$$
in standard $(\theta, \phi)$ spherical coordinates for the $2$-sphere.
The $4$ dimensional volume form is thus given by $$\frac{\rho}{t}r^{n-1} d\rho dr d\sigma_{\mathbb{S}^{n-1}},$$ where $d\sigma_{\mathbb{S}^{n-1}}$ is the standard volume form of the $n-1$ dimensional unit sphere.

The Minkowski metric induces on each of the $H_\rho$ a Riemannian metric given by
$$
ds_{H_\rho}^2=\frac{\rho^2}{t^2}dr^2 +r^2 \sigma_{\mathbb{S}^{n-1}}.
$$

A normal differential form to $H_\rho$ is given by $t dt-r dr$ while $t \partial_t +r \partial_r$ is a normal vector field. Since 
$$
\eta\left( t \partial_t +r \partial_r,  t \partial_t +r \partial_r\right)=-\rho^2,
$$
the future unit normal vector field to $H_\rho$ is given by the vector field 
\eq{
\nu_\rho \equiv\frac{1}{\rho}\left(t \partial_t +r \partial_r \right).
}

Finally, the induced volume form on $H_\rho$, denoted $d\mu_{H_\rho}$, is given by
$$
d\mu_{H_\rho}=\frac{\rho}{t}r^{n-1}drd\sigma_{\mathbb{S}^{n-1}}.
$$

\subsection{Regular distribution functions} \label{se:rdf}
For the massive transport operator, we will consider distribution functions $f$ as functions of $(t,x,v)$ or $(\rho,r,\omega,v)$ defined on $$\displaystyle \bigcup_{1 \le \rho < P}H_\rho \times \mathbb{R}^n_v, \quad P \in [1,+\infty],$$
i.e.~we are looking at the future of the unit hyperboloid, or a subset of it, times $\mathbb{R}^n_v$. 

For the massless transport operator, we need to exclude $|v|=0$ and we will only use the $\Sigma_t$ foliation so that we will consider distribution functions $f$ as functions of $(t,x,v)$ defined on $[0,T) \times \mathbb{R}^n_x \times \left( \mathbb{R}^n_v\setminus \{ 0 \} \right)$, $T \in [0,+\infty]$. 

In the remainder of this article, we will denote by \emph{regular} distribution function any such function $f$ that is sufficiently regular so that all the norms appearing on the right-hand sides of the estimates are finite. For simplicity, the reader might assume that $f$ is smooth and decays fast enough in $x$ and $v$ at infinity and in the massless case, that $f$ is integrable near $v=0$ and similarly for the distribution functions obtained after commutations. 

In physics, distribution functions represent the number of particles and are therefore required to be non-negative. This will play no role in the present article so we simply assume that distribution functions are real valued.

\subsection{The linear equations}
In the first part of this paper, we will study, for any $\T=\T_0,\T_1$, the homogeneous transport equation
\eq{ \label{eq:he}
\T f=0,
} 
as well as the inhomogeneous transport equation 
\begin{equation}\label{eq-model}
\mathbf T f = v^0 h,
\end{equation}
where $v^0=\sqrt{1+|v|^2}$ in the massive case and $v^0=|v|$ in the massless case and where the source term $h$ is a regular distribution function, as explained in Section \ref{se:rdf}.

In the massless case, we will study the solution $f$ to \eqref{eq:he} or \eqref{eq-model} with the initial data condition $f(t=0,.)=f_0$, where $f_0$ is a function defined on $\mathbb{R}^n_x \times \left( \mathbb{R}^n_v \setminus \{0\} \right)$. % for the massless case.

In the massive case, we will study the solution $f$ to \eqref{eq:he} or \eqref{eq-model} in the future of the unit hyperboloid with the initial data condition $f_{|H_1 \times \mathbb{R}^n_v}=f_0$, where $f_0$ is a function defined on $H_1 \times  \mathbb{R}^n_v$.

Equations \eqref{eq:he} and \eqref{eq-model} are transport equations and can therefore be solved explicitly (at least for $C^1$ initial data) via the method of characteristics. If $f$ solves \eqref{eq:he}, then 
$$
f(t,x,v)=f\left(0, x-\frac{v}{v^0}t,v\right)
$$
where $v^0=\sqrt{1+|v|^2}$ for the massive case and $v^0=|v|$ in the massless case.
In the inhomogeneous case, we obtain via the Duhamel formula that if $f$ solves \eqref{eq-model} with $0$ initial data, then 
$$
f(t,x,v) = \int_{0}^t h\left(s,x-(t-s)\dfrac{v}{v^0},v\right) ds.
$$

\subsection{The commutation vector fields}

\subsubsection{Complete lifts of isometries and conformal isometries}\label{sec:completelift}
Let us recall that the set of generators of isometries of the Minkowski space, that is to say the set of Killing fields, denoted by $\mathbb P$, consists of the translations, the rotations and the hyperbolic rotations, i.e. 

\begin{equation*}
\mathbb P=\Big\{\partial_t,\partial_{x^1},\hdots,\partial_{x^n}\Big\}
\cup\Big\{\Omega_{ij}=x^i\partial_{x^j}-x^j\partial_{x^i},\,\, 1\leq i,j\leq n\Big\}\cup \Big\{\Omega_{0i}=t\partial_{x^i}+x^i\partial_{t}, \,\, 1\leq i\leq n\Big\}.
\end{equation*}
Mostly in the case of the massless transport operator, it will be useful, as in the study of the wave equation, to add the scaling vector field $S=t\partial_t+x^i\partial_i$ to our set of commutator vector fields. Let us thus define the set 
\begin{equation*}
\mathbb K= \mathbb P\cup \big\{S\big\}.
\end{equation*}

The vector fields in $\K$ or $\PP$ lie in the tangent bundle of the Minkowski space. To any vector field on a manifold, we can associate a \emph{complete lift}, which is a vector field lying on the tangent bundle to the tangent bundle of the manifold. In Appendix \ref{se:gvf}, we recall the general construction on a Lorentzian manifold. For the sake of simplicity, let us here give a working definition of the complete lifts only in coordinates.

\begin{definition}
Let $W$ be a vector field of the form 
$W=W^\al\partial_{x^\al}$, then let
\eq{
\widehat{W}=W^\al{\partial}_{x^\al}+v^\be \frac{\partial W^i }{\partial x^\be } {\partial}_{v^i},
}
where $\left(v^\beta\right)_{\beta=0,..,n}=(v^0,v^1,.., v^n)$ with $v^0=|v|$ in the massless case, $v^0=\sqrt{1+|v|^2}$ in the massive case, be called the \emph{complete lift}\footnote{This is in fact a small abuse of notation, as, with the above definition, $\widehat{W}$ actually corresponds to the restriction of the complete lift of $W$ to the submanifold corresponding to $v^0=\sqrt{1+|v|^2}$ in the massive case and $v^0=|v|$ in the massless case. See again Appendix \ref{se:gvf} for a more precise definition of $\widehat{W}$.} of $W$. 
\end{definition}

We will denote by
$$
\widehat{\K}\equiv\{\widehat Z\,|\, Z\in\K\}
$$
and
$$
\widehat{\PP}\equiv\{\widehat Z\,|\, Z\in \PP\},
$$
the sets of the complete lifts of $\K$ and $\PP$.

Finally, let us also define $\widehat{\PP}_0$ and $\widehat{\K }_{\mathrm{0}}$ as the sets composed respectively of $\widehat{\PP}$ and $\widehat{\K}$ and a scaling vector field\footnote{Here, by a small abuse of notation, we denote with the same letter $S$, the vector field $t \partial_t +x^i \partial_{x^i}$ irrespectively of whether we consider it as a vector field on $\mathbb{R}^{n+1}$ or a vector field on (some subsets of) $\mathbb{R}^{n+1} \times \mathbb{R}^n_v$.} in $(t,x)$ only:

\begin{eqnarray}
\widehat{\PP }_{\mathrm{0}}&\equiv&\widehat{\PP} \cup \{ t\partial_t+x^i \partial_{x^i} \}=\widehat{\PP} \cup \{ S \}, \label{def:p0h}\\
\widehat{\K }_{\mathrm{0}}&\equiv&\widehat{\K} \cup \{ t\partial_t+x^i \partial_{x^i}  \}=\widehat{\K} \cup \{ S  \}.  \label{def:k0h}
\end{eqnarray}

\begin{lemma}
In Cartesian coordinates, the complete lifts of the elements of\, $\PP$ and $\K$ are given by the following formulae:
\eq{\alg{
\widehat{\partial_t}&=\p t,\\
\widehat{\partial_{x^i}}&=\p {x^i},\\
\widehat{\Omega_{ij}}&=x^i\p {x^j}-x^j\partial_{x^i}+v^i\pv j-v^j\pv i,\\
\widehat{\Omega_{0i}}&=t\partial_{x^i}+x^i\p t+v^0\pv i,\\
\widehat{S}&=t\partial_t+x^i\partial_{x^i}+v^i\pv i .\nonumber
}}
\end{lemma}

\subsubsection{Commutation properties of the complete lifts}\label{sec:commutators}
As for the wave equation, the symmetries of the Minkowski space are reflected in the transport operators \eqref{eq:msvto} and \eqref{eq:mslto} through the existence of commutation vector fields. More precisely,
\begin{lemma} \label{lem:cpcl}

\begin{itemize}

\item Commutation rules for the massive transport operator
\begin{align}
[\T_1, \widehat Z]&=0, \qquad \forall \widehat Z\in \widehat{\PP}, \\
[\T_1, S]&=\T_1, \label{eq:csvm}
\end{align}
where $S=t \partial_t +x^i \partial_{x^i}$ is the usual scaling vector field.

\item Commutation rules for the massless transport operator
\begin{align}
[\T_0, \widehat{Z}]&=0, \qquad \forall \widehat Z \in \widehat{\K}, \\
[{\T}_0, S ] &=\T_0.
\end{align}
\end{itemize}

\end{lemma}
\begin{proof}
The identities can be verified directly using the explicit expressions for the elements in $\widehat{\PP}$ and $\widehat{\K}$, but also follow from the general formula given in Appendix \ref{se:gvf} (cf. Lemma \ref{lem:commutationgeodesicspray}). 
\end{proof}
\begin{remark}
Note that from the expression of $\widehat{S}$ and the two commutation rules for $\T_0$ and $\widehat{S}$ and for $\T_0$ and $S$, it follows that
$$
[\T_0, v^i \partial_{v^i}]=-\T_0.
$$
Thus, we have in a certain sense two scaling symmetries, one in $x$ and one in $v$.
\end{remark}
\begin{remark}
It is interesting to note that while the Klein-Gordon operator $\square-m^2$ ($m>0$) does not commute with the scaling vector field, the massive transport equation does commute in the form of equation \eqref{eq:csvm}. What does not commute is the second scaling vector field $v^i \partial_{v^i}$. 
\end{remark}

We also have the following commutation relation within $\widehat{\PP}_0$ and $\widehat{\K}_0$.

\begin{lemma}For any $Z,Z' \in \widehat{\PP}_0$, there exist constant coefficients $C_{ZZ'W}$ such that 
$$
[Z,Z']=\sum_{W \in \widehat{\PP}} C_{ZZ'W} W.
$$
Similarly, for any $Z,Z' \in \widehat{\K}_0$, there exist constant coefficients $D_{ZZ'W}$ such that 
$$
[Z,Z']=\sum_{W \in \widehat{\K}} D_{ZZ'W} W.
$$
\end{lemma}

\subsection{Weights preserved by the flow} \label{se:wpf}
Recall that in a general Lorentzian manifold with metric $g$, if $\gamma$ is a geodesic with tangent vector $\dot \gamma$ and $K$ denotes a Killing field, then $g(\dot \gamma, K)$ is preserved along $\gamma$. In this section, we explain how to transpose this fact to the transport operators on Minkowski space. 

We define the set of weights

\eq{
	\mathbb{k}_m\equiv\{v^\al x^\be - x^\al v^\be, v^\al \}
}
and
\eq{
	\mathbb{k}_0\equiv\{x^\al v_\al,v^\al x^\be - x^\al v^\be, v^\al \}. 
}

The following lemma can be easily checked. 
\begin{lemma} 
\begin{enumerate}
\item For all $\zz \in \mathbb{k}_0$, \,$[\T_0, z]=0$.
\item For all $\zz \in \mathbb{k}_m$, \,$[\T_m, z]=0$.
\end{enumerate}

\end{lemma}

The weights in $\mathbb{k}_m$ or $\mathbb{k}_0$ also have nice commutation properties with the vector fields in $\widehat{\PP}_0$ and $\widehat{\K}_0$.

\begin{lemma} For any $\zz \in \mathbb{k}_m$ and any $\widehat{Z} \in \widehat{\PP}_0$, 
$$
[\widehat{Z}, \zz ]=\sum_{ \zz' \in \mathbb{k}_m} c_{\zz'} \zz'
$$
where the $c_{\zz'}$ are constant coefficients.

Similarly for any $\zz \in \mathbb{k}_0$ and any $\widehat{Z} \in \widehat{\K}_0$, 
$$
[\widehat{Z}, \zz ]=\sum_{ \zz' \in \mathbb{k}_m} d_{\zz'} \zz',
$$
for some constant coefficients $d_{\zz'}$.
\end{lemma}

\begin{proof}
This follows from straightforward computations.
\end{proof}

\subsection{Multi-index notations}\label{sec:multindex}
Recall that a multi-index $\alpha$ of length $|\alpha|$ is an element of $\mathbb{N}^r$ for some $r \in \mathbb{N} \setminus \{0\}$ such that $\sum_{i=1}^r \alpha_i=|\alpha|$.

Let $Z^i,i=1,..,2n+2+n(n-1)/2$ be an ordering of $\K$. For any multi-index $\alpha$, we will denote by $Z^\alpha$ the differential operator of order $|\alpha|$ given by the composition $Z^{\alpha_1} Z^{\alpha_2}..$. 

In view of the above discussion, the complete lift operation defines a bijection between $\K$ and $\widehat{\K}$. Thus, to any ordering of $\K$, we can associate an ordering of $\widehat{\K}$. One extends this ordering to $\widehat{\K}_0$ by setting\footnote{Note that this is a small abuse of notation, since $S$ is not obtained via the complete lift construction.} $\widehat{Z}^{2n+3+n(n-1)/2}=S$.
We will again write $\widehat{Z}^\alpha$ to denote the differential operator of order $|\alpha|$ obtained by the composition $\widehat{Z}^{\alpha_1} \widehat{Z}^{\alpha_2}..$.

Similarly, we consider an ordering of $\PP$, which gives us an ordering of $\widehat{\PP}$ which can be extended to give an ordering of $\widehat{\PP}_0$, and we write $Z^\alpha$ and $\widehat{Z}^\alpha$ for a composition of $|\alpha|$ vector fields in $\PP$, $\widehat{\PP}$ or $\widehat{\PP}_0$.

The notation $\K^{N}$ will be used to denote the set of all the differential operators of the form $Z^\alpha$, with $|\alpha|=N$. Similarly, we will use the notations $\PP^{|\alpha|}$, $\widehat{\K}_0^{|\alpha|}$ and $\widehat{\PP}_0^{|\alpha|}$.

We will also write $\partial_{t,x}^\alpha$ to denote a differential operator of order $|\alpha|$ obtained as a composition of $|\alpha|$ translations among the $\partial_t$, $\partial_{x^i}$ vector fields. 

As for the sets of vector fields, we will also consider orderings of the sets of weights $\mathbb{k}_m$ and $\mathbb{k}_0$ and we will write $\zz^{\alpha} \in \mathbb{k}_m^{|\alpha|}$ or $\zz^{\alpha} \in \mathbb{k}_0^{|\alpha|}$   to denote a product of $|\alpha|$ weights in $\mathbb{k}_m$ or $\mathbb{k}_0$. 

\subsection{Vector field identities}
The following classical vector field identities will be used later in the paper.

\begin{lemma}\label{lem:vfi}
The following identities hold:
\begin{eqnarray*}
(t^2-r^2)\partial_t &=& tS -x^i \Omega_{0i},\\
(t^2-r^2)\partial_i&=&-x^j\Omega_{ij}+t \Omega_{0i}-x^i S, \\
(t^2-r^2)\partial_r&=& t\frac{x^i}{r} \Omega_{0i}-rS.
\end{eqnarray*}
Furthermore,
\begin{equation} \label{id:vfi2}
\partial_s \equiv\dfrac{1}{2}\left(\partial_t+\partial_r\right) = \dfrac{S+\omega^i\Omega_{0i}}{2(t+r)},\quad \overline{\partial}_i \equiv \partial_i - \omega_i \partial_r = \dfrac{\omega^j\Omega_{ij}}{r} = \dfrac{-\omega_i\omega^i\Omega_{0j}+\Omega_{0i}}{t}.
\end{equation}
\end{lemma}

\subsection{The particle vector field and the stress energy tensor of Vlasov fields} \label{se:pvfset}

Recall that in the Vlasov-Poisson or Einstein-Vlasov systems, the transport equation for $f$ is coupled to an elliptic equation or a set of evolution equations, via integrals in $v$ of $f$, often referred to as velocity averages in the classical case. 
In the relativistic cases, the volume forms\footnote{For the interested reader, they can be interpreted geometrically as the natural volume forms associated with an induced metric on the manifold on which the averages are computed, together with a choice of normal in the massless case.} in these integrals are defined as 
\begin{equation}\label{def:vf}
d\mu_{m}\equiv\frac{dv^1 \wedge \hdots \wedge dv^n}{v^0}=\frac{dv}{\sqrt{m^2+\ab v^2}},
\end{equation}
where as usual $m=0$ in the massless case. 

\begin{remark} In the massless case, the volume form $\frac{dv}{|v|}$ is singular near $v=0$. In the remainder of this article, we will however study mostly energy densities, which introduce an additional factor of $|v|^2$ in the relevant integrals and thus remove this singular behaviour near $v=0$.
\end{remark}

We now define the particle vector field in the case of massive particles as 

$$
N^{\mu }_m \equiv \int_{\mathbb R^n}f v^{\mu}d\mu_m,
$$
and in the case of massless particles as
$$
N^{\mu }_0 \equiv \int_{\mathbb R^n\setminus\{0\}}f v^{\mu}d\mu_0,
$$
as well as the energy momentum tensors 
$$
T^{\mu \nu}_m \equiv \int_{\mathbb R^n}f v^{\mu}v^{\nu}d\mu_m,
$$
and 
$$
T^{\mu \nu}_0 \equiv \int_{\mathbb R^n\setminus\{0\}}f v^{\mu}v^{\nu}d\mu_0,
$$
where $d\mu_m$ and $d\mu_0$ are the volume forms defined in \eqref{def:vf}. 
More generally, we can define the higher moments
$$
M^{\al_1\hdots\al_p}_m\equiv \int_{\mathbb R^n}f v^{\al_1}\hdots v^{\al_p}d\mu_m,
$$
 and similarly for the massless system. 

The interest in any of the above quantities is that if $f$ is a solution to the associated massless or massive transport equations, then these quantities are divergence free. Indeed, we have 
\begin{eqnarray} \label{eq:divt0}
\partial_\mu T_0^{\mu \nu}&=& \int_{\mathbb R^n\setminus\{0\}}\T_0(f)v^{\nu}d\mu_{\mathrm{0}},\\
\partial_\mu T_m^{\mu \nu}&=& \int_{\mathbb R^n}\T_m(f)v^{\nu}d\mu_{m}. \label{eq:divtm}
\end{eqnarray}

We will be interested in particular in the energy densities
\eq{
\rho_{\mathrm{0}}(f)\equiv T_0( \partial_t, \partial_t)=\int_{\mathbb R^n \setminus\{0\}}f \ab{v} dv,
}
for the massless case, while for the massive case we define
\eq{
\rho_{\mathrm{m}}(f)\equiv T_m( \partial_t, \partial_t)=\int_{\mathbb R^n }f v^0 dv.
}

In the following, we will denote by $\rho(f)$ any of the quantities $\rho_{\mathrm{m}}(f)$ or $\rho_{\mathrm{0}}(f)$ depending on whether we are looking at the massive or the massless relativistic operator.

In the massive case, we will also make use of the following energy density
\eq{
\chi_{m}(f)\equiv T_m( \partial_t, \nu_\rho), 
}
where $\nu_\rho$ is the future unit normal to $H_\rho$ introduced in Section \ref{se:gh}. We compute
\begin{eqnarray*}
\chi_{\mathrm{m}}(f)&=&\int_{v \in \mathbb{R}^n} fv_0\left(\frac{t}{\rho}v_0+\frac{r}{\rho}v^r \right) d\mu_m, \\
&=&\int_{v \in \mathbb{R}^n } f\left(\frac{t}{\rho}v^0-\frac{x^i}{\rho}v_i \right) dv.
\end{eqnarray*}

The following lemma will be used later.
\begin{lemma}[Coercivity of the energy density normal to the hyperboloids]\label{lem:coercivity}
Assuming that $t \ge r$, we have
\begin{eqnarray} \label{ineq:coeednh}
\chi_m(f) &\ge& \frac{t}{2\rho} \int_{v \in \mathbb{R}^n}f\left[ \left(1-\frac{r}{t}\right)\left((v^0)^2+v_r^2\right)+r^2 \sigma_{AB} v^A v^B+m^2\right] \frac{dv}{v^0}.
\end{eqnarray}
\end{lemma}
\begin{proof}
Using that 
$$(v^0)^2=v_r^2+r^2 \sigma_{AB} v^A v^B+ m^2 $$
where $\sigma_{AB}$ denotes the components of the metric $\sigma_{\mathbb{S}^n}$ and $v^A,v^B$ are the angular velocities, we have
$$
(v^0)^2=\frac{(v^0)^2}{2}+\frac{1}{2}\left(v_r^2+r^2 \sigma_{AB} v^A v^B+ m^2 \right)
$$
and thus
\begin{eqnarray*}
v^0\left(\frac{t}{\rho}v^0-\frac{x^i}{\rho}v_i \right)=\frac{t}{2\rho}\left( (v^0)^2+v_r^2+r^2 \sigma_{AB} v^A v^B+ m^2-2\frac{x^i}{t}v_iv^0 \right).
\end{eqnarray*}

The lemma now follows from
$$(v^0)^2+v_r^2-2\frac{x^i}{t}v_iv^0 \ge \left(1-\frac{r}{t}\right)\left((v^0)^2+v_r^2\right),
$$
assuming $t \ge r$. %, which, since $(v^0)^2 \ge v_r^2$, we simplify in 
\end{proof}
\begin{remark}\label{re:cchi} \begin{itemize}
\item Since $(v^0)^2 \ge v_r^2$, we will use \eqref{ineq:coeednh} in the form
$$
\chi_m(f) \ge  \frac{t}{2\rho} \int_{v \in \mathbb{R}^n}f\left[ \left(1-\frac{r}{t}\right)(v^0)^2+r^2 \sigma_{AB} v^A v^B+m^2\right]d\mu_m.
$$
\item We also remark that
$$
\chi_m(\vert f\vert ) \geq \dfrac{m^2}{2} \int_v |f| \dfrac{dv}{v^0} = \dfrac{m^2}{2} \rho_m\left(\frac{|f|}{(v^0)^2}\right),
$$
since $\tfrac{t}{2\rho}\geq \tfrac12$, and, furthermore,
$$
\chi_m(\vert f\vert ) \geq \frac{t-r}{2\rho} \rho_{m}(\vert f \vert) =  \frac{\rho}{2(t+r)} \rho_{m}(\vert f \vert).
$$
\item Finally, independently of Lemma \ref{lem:coercivity}, since by the Cauchy-Schwarz inequality for Lorentzian metrics, as the vectors $\nu_\rho$, and $v$ are both timelike future directed,
$$
\left|\dfrac{tv^0 - x^iv_i}{\rho}\right| = |\langle v,\nu_\rho\rangle | \geq  |v| |\nu_\rho| = m, \text{ where }  |v| = |g(v,v)|^{\tfrac{1}{2}},
$$
we get immediately
$$
\int_v |f| dv \leq \int_v \dfrac{1}{m}  \left|\dfrac{tv^0 - x^iv_i}{\rho}\right| |f| dv = \dfrac{1}{m} \chi_m(|f|).
$$
\end{itemize}
\end{remark}

\subsection{Commutation vector fields and energy densities}
Vector fields and the operator of averaging in $v$ essentially commute in the following sense.
\begin{lemma} \label{lem:cvfavmsl}
Let $f$ be a regular distribution function for the massless case. Then, 
\begin{itemize}
\item for any translation $\partial_{x^\alpha}$, we have
$$
\partial_{x^\alpha}\left[\rho_0( f ) \right]= \rho_0 \left( \partial_{x^\alpha}(f) \right)=\rho_0 \left( \widehat{\partial_{x^\alpha}}(f) \right).
$$
\item for any rotation $\Omega_{ij}$, $1\le i,j, \le n$, we have
$$
\Omega_{ij}\left[\rho_0( f ) \right]=\rho_0 \left( \widehat{\Omega_{ij}}(f) \right),
$$
where $\widehat{ \Omega_{ij} } $ is the complete lift of the vector field $\Omega_{ij}$.
\item for any Lorentz boost $\Omega_{0i}$, $1\le i \le n$, we have
$$
\Omega_{0i}\left[\rho_0( f ) \right]=\rho_0 \left( \widehat{\Omega_{0i}}(f) \right)+2\rho_0 \left( \frac{v^i}{|v|}f \right).
$$
\item for the scaling vector field $S$, we have
$$
S\left[\rho_0(f) \right]= \rho_0 \left( \widehat{S}(f) \right)+ (n+1)\rho_0(f).
$$
\item finally, all the above equalities hold (almost everywhere) with $f$ replaced by $|f|$, for instance
$$
S\left[\rho_0(|f|) \right]= \rho_0 \left( \widehat{S}(|f|) \right)+ (n+1)\rho_0(|f|).
$$
\end{itemize}
\end{lemma}
\begin{proof}
Let us consider, for instance, a Lorentz boost  $\Omega_{0i}=t\partial_{x^i}+x^i\partial_t$, then 
\begin{eqnarray}
\Omega_{0i}\left[\rho_0( f ) \right]&=&\int_{v}\left(t\partial_{x^i}+x^i\partial_t \right)(f)|v|dv.
\end{eqnarray}
On the other hand, note that 
\begin{eqnarray*}
\int_{v}\left(t\partial_{x^i}+x^i\partial_t \right)(f)|v|dv&=&\int_{v}\left(t\partial_{x^i}+x^i\partial_t +|v|\partial_{v^i} \right)(f) |v| dv-\int_{v}|v|^2\partial_{v^i} (f) dv \\
&=&\int_{v}\widehat{\Omega_{0i}}(f) |v| dv+2\int_{v}\frac{v^i}{|v|} (f)|v| dv \\
&=&\rho_0 \left( \widehat{\Omega_{0i}}(f) \right)+2\rho_0 \left( \frac{v^i}{|v|}f \right),
\end{eqnarray*}
using an integration by parts in $v^i$. The other cases can all be treated similarly, the translations being trivial since $\widehat{\partial_{x^\alpha}}=\partial_{x^\alpha}$. That $f$ can be replaced by $|f|$ follows from the standard property of differentiation of the absolute value\footnote{Recall that $f \in W^{1,1}$ implies that $|f| \in W^{1,1}$ with $\partial |f|= \frac{f}{|f|} \partial f$ almost everywhere. See for instance \cite{ll:a}, Chap 6.17.}.
\end{proof}

In the massive case, we have the following lemma, whose proof is left to the reader since it is very similar to the above.

\begin{lemma}\label{lem:cvfavmsv}
Let $f$ be a regular distribution function for the massive case. Then, 
\begin{itemize}
\item for any translation $\partial_{x^\alpha}$, we have
$$
\partial_{x^\alpha}\left[\rho_m( f ) \right]= \rho_m \left( \partial_{x^\alpha}(f) \right)=\rho_m \left( \widehat{\partial_{x^\alpha}}(f) \right).
$$
\item for any rotation $\Omega_{ij}$, $1\le i,j, \le n$, we have
$$
\Omega_{ij}\left[\rho_m( f ) \right]=\rho_m \left( \widehat{\Omega_{ij}}(f) \right),
$$
where $\widehat{ \Omega_{ij} } $ is the complete lift of the vector field $\Omega_{ij}$.
\item for any Lorentz boost $\Omega_{0i}$, $1\le i \le n$, we have
$$
\Omega_{0i}\left[\rho_m( f ) \right]=\rho_0 \left( \widehat{\Omega_{0i}}(f) \right)+2\rho_m \left( \frac{v^i}{v^0}f \right).
$$
\item finally, all the above equalities holds with $f$ replaced by $|f|$.
\end{itemize}
\end{lemma}
\begin{remark} \label{rem:cvav}
Although we do not have for all commutation vector fields $Z \rho= \rho \widehat{Z}$, we do have that $|Z\rho(|f|)|\lesssim \rho\left(|\widehat{Z}(f)| \right) + \rho(|f|)$ and this is all we shall need from the above. Note also that if we were looking at other moments, then similar formulae would hold with different coefficients. For instance, we have $\Omega_{0i} \int_v f d\mu_m=\int_{v} \widehat\Omega_{0i}f d\mu_m$ for sufficiently regular $f$. 
\end{remark}
\begin{remark}
In the massless case, we included the scaling vector field, but recall that $\T_0$ actually commutes with $S$ (in the sense that $[\T_0,S]=\T_0$) so that we will not really need to replace $S$ by $\widehat{S}$. Note also that $S$ also enjoys good commutation properties with $\T_m$ and that $S\rho_m=\rho_m S$. % later in the text.
\end{remark}

\subsection{(Approximate) conservation laws for Vlasov fields}
The following lemma is easily verified\footnote{Recall that if $f$ is a regular solution to $\T(f)=v^0 h$, then $|f|$ is a solution, in the sense of distributions, of $\T( |f|)=\frac{f}{|f|} v^0 h$.}

\begin{lemma}\label{lem:acl}[Massless case] Let $h$ be a regular distribution function for the massless case in the sense of Section \ref{se:rdf}. 
%\begin{itemize}
%\item 
Let $f$ be a regular solution to $\T_0(f)=v^0 h$, with $v^0=|v|$, defined on $[0,T] \times \mathbb{R}^n_x \times \left( \mathbb{R}^n_v\setminus \{0 \} \right)$ for some $T>0$. Then, for all $t \in [0,T]$,
\begin{equation} \label{eq:aclr}
\int_{\Sigma_t} \rho_0(f)(t,x)dx\left(\equiv\int_{x \in \mathbb{R}^n} \int_{v \in \mathbb{R}^n \setminus \{ 0 \}}|v| f(t,x,v)dxdv\right)= \int_{\Sigma_0} \rho_0(f)(0,x)dx+\int_{0}^t \int_{\Sigma_s} \rho_0(h)(s,x)dx ds,
\end{equation}
and 
\begin{equation} \label{eq:aclav}
\int_{\Sigma_t}\rho_0(|f|)(t,x)dx \le \int_{\Sigma_0}\rho_0(|f|)(0,x)dx+\int_{0}^t\int_{\Sigma_s} \rho_0(|h|)(s,x)dx ds.
\end{equation}
\end{lemma}
\begin{proof}
The proof of \eqref{eq:aclr} follows from an easy integration by parts (or an application of Stoke's theorem) and \eqref{eq:divt0}. A standard regularization argument of the absolute value allows to derive \eqref{eq:aclav} in a similar manner. 
\end{proof}
A similar identity holds for the massive case, but we shall need the following variant where we replace the $\Sigma_t$ foliation by the $H_\rho$ one.
\begin{lemma}[Massive case]\label{lem:macl} Let $h$ be a regular distribution function for the massive case in the sense of Section \ref{se:rdf}. 
Let $f$ be a regular solution to $\T_m(f)=v^0 h$, with $v^0=\sqrt{m^2+|v|^2}$, $m > 0$, defined on $\bigcup_{\rho \in[1,P]}H_\rho \times \mathbb{R}^n_v$ for some $P>1$. Then, for all $\rho \in [1,P]$,
\begin{equation} \label{eq:aclrm}
\int_{H_\rho} \chi_m(f)(\rho,r,\omega)d\mu_{H_\rho}= \int_{H_1} \chi_m(f)(1,r,\omega)d\mu_{H_1}+\int_{1}^\rho \int_{H_\rho} \rho_m(h)(s,r,\omega)d\mu_{H_s}ds,
\end{equation}
and 
\begin{equation} \label{eq:aclavm}
\int_{H_\rho} \chi_m(|f|)(\rho,r,\omega)d\mu_{H_\rho} \le \int_{H_1} \chi_m(|f|)(1,r,\omega)d\mu_{H_1}+\int_{1}^\rho \int_{H_\rho} \rho_m(|h|)(s,r,\omega)d\mu_{H_s}ds,
\end{equation}
\end{lemma}
\begin{proof}Again, the proof of \eqref{eq:aclrm} just follows from \eqref{eq:divtm} and an integration by parts, while that of \eqref{eq:aclavm} follows similarly after a standard regularization argument.
\end{proof}

\section{The vector field method for Vlasov fields} \label{se:vfvf}

\subsection{The norms} \label{se:tn}
We define in the following norms of distribution functions obtained from the standard conservation laws for the transport equations and the commutation vector fields introduced in the previous section.

\begin{definition}
\begin{itemize}
\item Let $f$ be a regular distribution function for the massless case in the sense of Section \ref{se:rdf} defined on $[0,T] \times \mathbb{R}^n_x \times \left(\mathbb{R}^n_v \setminus \{0\}\right)$. For $k\in \mathbb{N}$, we define, for all $t \in [0,T]$, 
\eq{
\Ab{f}_{\mathbb K,k}(t)\equiv\sum_{|\al|\leq k}\sum_{\widehat{Z}^\alpha \in \widehat{\K}^{|\alpha|}}\int_{\Sigma_t} \rho_0(\ab{\widehat{Z}^\al  f})(t,x)  dx.
}
\item Similarly, let $f$ be a regular distribution function for the massive case in the sense of Section \ref{se:rdf} defined on $\displaystyle \bigcup_{1 \le \rho \le P}H_\rho \times \mathbb{R}^n_v$. For $k\in \mathbb{N}$, we define, for all $\rho \in [0,P],$
\eq{
\Ab{f}_{\mathbb P,k}(\rho)\equiv\sum_{|\al|\leq k}\sum_{\widehat{Z}^\alpha \in \widehat{\PP}^{|\alpha|}}\int_{H_\rho} \chi_m(\ab{\widehat{Z}^\al  f})  d\mu_{H_\rho}.
}
 \end{itemize}
\end{definition}

\subsection{Klainerman-Sobolev inequalities and decay estimates: massless case} \label{se:ksmsl}
We are now ready to prove to following variant of the Klainerman-Sobolev inequalities\footnote{Note that (in more than $1$ spatial dimension) we cannot apply directly the standard Klainerman-Sobolev inequalities, in fact not even the usual Sobolev inequalities, to quantities such as $\rho( |f|)$ because of the lack of regularity of the absolute value. The aim of this section is therefore to explain how to circumvent this technical issue.}\footnote{For a very clear introduction to Klainerman-Sobolev inequalities in the classical case of the wave equation, the interested reader may consult \cite{qw:lnlw}. Some of the arguments below have been adapted from those notes.}.

\begin{theorem}[Klainerman-Sobolev inequalities for velocity averages of massless distribution functions]\label{ineq:ksmsl}
Let $f$ be a regular distribution function for the massless case defined on $[0,T] \times \mathbb{R}_x^n \times \left(\mathbb{R}^n_v \setminus \{0 \}\right)$ for some $T>0$. Then, for all $(t,x) \in [0,T] \times \mathbb{R}_x^n$, 
\begin{equation}\label{ineq:msl}
\rho_0(|f|) (t,x) \lesssim \frac{1}{(1+|t-|x|\,|) (1+|t+|x|\,|)^{n-1}}\Ab{f}_{\mathbb K,n}(t). 
\end{equation}
\end{theorem}
\begin{proof}Let $(t,x) \in [0,T] \times \mathbb{R}_x^n$ and assume first that $|x| \notin [t/2, 3/2t]$ and $t+|x| \ge 1$.
Let $\psi$ be defined as
$$
\psi: y \rightarrow \rho_0 \left( |f|(t,x+(t+|x|)y) \right),
$$
where $y=(y_1,y_2,..,y_n)$. Note that 
$$\partial_{y_i} \psi(y)=\partial_{y^i}\left[\rho_0\left( |f|(t,x+(t+|x|)y\right) \right]=(t+|x|)\partial_{x^i}\left(\rho_0 \left[ |f|\right] \right)(t,x+(t+|x|)y).$$
Assume now that $|y| \le 1/4$. Using the fact that we are away from the light-cone and the condition on $|y|$, it follows that 
$$1/C \le \frac{|t+|x||}{|t-|x+(t+|x|)y|} \le C,$$ for some $C>0$. It then follows from the vector field identities of Lemma \ref{lem:vfi} that
$$|\partial_{y^i}\rho_0( |f|(t,x+(t+|x|)y) | \lesssim \sum_{Z \in \K} \left| Z \left( \rho_0 \left[|f|\right] \right)\right|(t,x+(t+|x|)y).$$
From Lemma \ref{lem:cvfavmsl}, we then obtain that 
\begin{eqnarray*}
\left|\partial_{y^i} \rho_0\left[ |f|(t,x+(t+|x|)y) \right] \right| &\lesssim & \sum_{|\alpha| \le 1, \widehat{Z}^{\alpha} \in \widehat{\K}^{|\alpha|}} \left|\rho_0 \left[\widehat{Z}(|f|)\right]\right|(t,x+(t+|x|)y)+\rho_0(|f|)(t,x+(t+|x|)y), \\
&\lesssim& \sum_{|\alpha| \le 1, \widehat{Z}^{\alpha} \in \widehat{\K}^{|\alpha|}}\left|\rho_0 \left[\widehat{Z}^\alpha (|f|)\right]\right|(t,x+(t+|x|)y), \\
&\lesssim& \sum_{|\alpha| \le 1, \widehat{Z}^{\alpha} \in \widehat{\K}^{|\alpha|}}\rho_0 \left[\left|\widehat{Z}^\alpha (|f|)\right|\right](t,x+(t+|x|)y), \\
&\lesssim& \sum_{|\alpha| \le 1, \widehat{Z}^{\alpha} \in \widehat{\K}^{|\alpha|}}\rho_0 \left[\left|\widehat{Z}^\alpha (f) \right|\right](t,x+(t+|x|)y),
\end{eqnarray*}
where we have used in the last line that for any vector field $\widehat{Z}$, $\left| \widehat{Z}(|f|)\right|= \left| \widehat{Z}(f) \right|$ (almost everywhere and provided $f$ is sufficiently regular), which essentially follows from the fact that $\partial |f|=\frac{f}{|f|} \partial f$ almost everywhere if $f \in W^{1,1}$.
Let now $\delta=\frac{1}{16n}$, so that if $|y_i| \le \delta^{1/2}$ for all $1 \le i \le n$, we then have $|y| \le 1/4$. Applying now a $1$ dimensional Sobolev inequality in the variable $y_1$, we have
\begin{eqnarray*}
| \psi(0)|=\rho_0\left[|f|\right](t,x) &\lesssim& \int_{|y_1| \le \delta^{1/2}} \left( \left| \partial_{y_1} \psi (y_1,0,..,0) \right| + \left|\psi (y_1,0,..,0)\right|\right) dy_1, \\
 &\lesssim& \int_{|y_1| \le \delta^{1/2}} \left( \sum_{|\alpha| \le 1, \widehat{Z}^{\alpha} \in \widehat{\K}^{|\alpha|}}\rho_0 \left[\left|\widehat{Z}^\alpha (f) \right|\right]\left(t,x+(t+|x|)(y_1,0,..,0) \right) \right) dy_1.
\end{eqnarray*}
We can now apply a $1$ dimensional Sobolev inequality in the variable $y_2$ and repeat the previous argument, with $|Z^\alpha(f)|$ replacing $|f|$, to obtain
$$
| \psi(0)| \lesssim \int_{|y_1| \le \delta^{1/2}}\int_{|y_2| \le \delta^{1/2}} \left( \sum_{|\alpha| \le 2, \widehat{Z}^{\alpha} \in \widehat{\K}^{|\alpha|}}\rho_0 \left[\left|\widehat{Z}^\alpha (f) \right|\right]\left(t,x+(t+|x|)(y_1,y_2,..,0) \right) \right) dy_1 dy_2.
$$
Repeating the argument up to exhaustion of all variables, we obtain that 
\begin{equation*}
\begin{aligned}
&\rho_0\left[|f|\right](t,x)\\
&\quad\lesssim \int_{|y_1| \le \delta^{1/2}}\int_{|y_2| \le \delta^{1/2}}..\int_{|y_n| \le \delta^{1/2}} \left( \sum_{|\alpha| \le n, \widehat{Z}^{\alpha} \in \widehat{\K}^{|\alpha|}}\rho_0 \left[\left|\widehat{Z}^\alpha (f) \right|\right]\left(t,x+(t+|x|)(y_1,y_2,..,y_n) \right) \right) dy_1 dy_2..dy_n.
\end{aligned}
\end{equation*}
Applying the change of variable $z=(t+|x|)y$ gives us a $(t+|x|)^n$ factor which completes the proof of the inequality in this particular case. 
The case where $(t+|x|) \le 1$ follows from simpler considerations and is therefore left to the reader. 

Let us thus turn to the case where $x \in [t/2,3/2t]$ and $(t+|x|) \ge 1$ . Note that it then follows that $t > 2/5$ and $|x| > 1/3$. Let us introduce spherical coordinates $(r,\omega) \in [0,+\infty) \times \mathbb{S}^{n-1}$, such that $x=r \omega$ and denote by $q$ the optical function $q\equiv r-t$. Let $v(t,q,\omega)\equiv \rho_0(f)(t,(t+q)\omega)$.

Note that $\partial_q v= \partial_r \rho_0$, $q \partial_q v=(r-t) \partial_r$ and that there exist constants $C_{ij}$ such that $$\partial_\omega v = \partial_\omega \left( \rho_0(f)(t,(q+t) \omega )\right)=\sum_{i<j}C_{ij}\Omega_{ij}\rho_0(f),$$ where the $\Omega_{ij}$ are the rotation vector fields.

Let $q_0=|x|-t$. We need to prove that 
$$
t^{n-1} (1+|q_0|) |v(t,q_0,\omega) | \lesssim \Ab{f}_{\mathbb K,n}(t).
$$
Using a one dimensional Sobolev inequality, we have for any $\omega \in \mathbb{S}^{n-1}$.
\begin{eqnarray*}
|v(t,q_0,\omega)| \lesssim \int_{|q| < t/4} \sum_{|\alpha| \le 1} \left| \left(\partial_q^{\alpha} v\right)(t,q+q_0,\eta)\right|dq.
\end{eqnarray*}
Note now that 
$$
\left(\partial_q v\right)(t,q+q_0,\omega)=\left(\partial_r \rho_0(f)\right)(t,q+q_0, \omega)= \rho_0 ( \partial_r (|f|))
$$
and thus
\begin{eqnarray*}
|\partial_q v(t,q+q_0,\omega)| &\lesssim&\rho_0 ( | \partial_r f|)(t,q+q_0,\omega),
\end{eqnarray*}
where we have used again the properties of the derivatives of the absolute value. Let now $(\omega_1,\omega_2,..,\omega_{n-1})$ be a local coordinate patch in a neighbourhood of the point $\omega \in \mathbb{S}^{n-1}$. Using again a $1-d$ inequality, we have 
\begin{eqnarray*}
\left| \rho_0 ( | \partial_r^{\alpha} f|)(t,q+q_0,\omega) \right| &\lesssim& \int_{\omega_1} \left| \partial_{\omega_1}\rho_0 ( | \partial_r^{\alpha} f|)(t,q+q_0,\omega+(\omega_1,0,..,0)) \right|d\omega_1\\
&&\hbox{}+\int_{\omega_1} \left| \rho_0 ( | \partial_r^{\alpha} f|)(t,q+q_0,\omega+(\omega_1,0,..,0)) \right|d\omega_1.
\end{eqnarray*}

Since $\partial_{\omega_1}$ can be rewritten in terms of the rotation vector fields, it follows from Lemma \ref{lem:cvfavmsl} that 
$$
\left| \partial_{\omega_1}\rho_0 ( | \partial_r^\alpha f|) \right| \lesssim \sum_{|\beta| \le 1} \rho_0\left( \left| \widehat{Z}^\beta \partial_r^\alpha f \right| \right).
$$
Repeating until exhaustion of the number of variables on $\mathbb{S}^{n-1}$ and using that $\partial_r=\frac{x^k}{|x|} \partial_{x^k}$ and the commutation properties between $\widehat{Z}^\alpha$ and $\partial_r$, we obtain that 
$$
|\rho_0(|f|)(t,q+q_0,\omega)| \lesssim \int_{|q| \le t/4} \int_{\eta \in \mathbb{S}^{n-1}}\sum_{ |\alpha| \le n} \rho_0\left( \left| \widehat{Z}^{\alpha}(f)\right| \right)(t,q+q_0, \eta) dq d\sigma_{\mathbb{S}^{n-1}}.
$$
Now since in the domain of integration, $r=t+q+q_0=q+|x| \sim t$, we have 

\begin{eqnarray*}
t^{n-1} |\rho_0(|f|)(t,q+q_0,\omega)| &\lesssim&\sum_{ |\alpha| \le n} \int_{|q| \le t/4} \int_{\eta \in \mathbb{S}^{n-1}} \rho_0\left( \left| \widehat{Z}^{\alpha}(f)\right| \right)(t,q+q_0, \eta)r^{n-1} dq d\sigma_{\mathbb{S}^{n-1}}, \\
&\lesssim& \sum_{ |\alpha| \le n} \int_{t/4 \le r \le 7t/4}  \int_{\eta \in \mathbb{S}^{n-1}} \rho_0\left( \left| \widehat{Z}^{\alpha}(f)\right| \right)(t,r,\eta)r^{n-1} dr d\sigma_{\mathbb{S}^{n-1}}\\
&\lesssim&  \sum_{ |\alpha| \le n}\int_{t/4 \le |y| \le 7t/4}  \rho_0\left( \left| \widehat{Z}^\alpha (f)\right| \right)(t,y)dy,
\end{eqnarray*}
which concludes the proof when $|q_0| \le 1$.

Assume now that $|q_0| > 1$. Let $\chi \in C^\infty_0(-1/2,1/2)$ be a smooth cut-off function such that $\chi(0)=1$ and define $V_{q_0}(t,q,\omega) \equiv \chi ((q-q_0)/q_0) v(t,q,\omega)$. To get the extra factor of $|q_0|$, we apply the method used above replacing the function $v$ by the function $(s,\eta) \rightarrow V_{q_0}(t,q_0+q_0s,\eta)$ and applying first a $1-d$ Sobolev inequality in $s$ on $|s| < 1/2$. The extra power of $q_0$ appearing are then absorbed since $|q_0+q_0s| \sim |q_0|$ in the region of integration and since $(r-t) \partial_r$ can be expressed as a linear combination of commutation vector fields from Lemma \ref{lem:vfi} (with coefficients homogeneous of degree $0$). The rest of the proof is similar to the one just given when $|q_0| \le 1$ and therefore omitted.
\end{proof}

Since the norm on the right-hand side is conserved for solutions of the homogeneous massless transport equations, we obtain in particular, 

\begin{theorem}[Decay estimates for velocity averages of massless distribution functions]\label{de:msl}
Let $f$ be a regular distribution function for the massless case, a solution to $\T_0(f)=0$ on $\mathbb{R}_t  \times \mathbb{R}_x^n \times \left(\mathbb{R}^n_v \setminus \{0 \}\right)$. Then, for all $(t,x) \in \mathbb{R}_t \times \mathbb{R}_x^n$, 
\begin{equation}
\rho_0(|f|) (t,x) \lesssim \frac{1}{(1+|t-|x|\,|) (1+|t+|x|\,|)^{n-1}}\Ab{f}_{\mathbb K,n}(0). 
\end{equation}
\end{theorem}

Finally, as for the wave equation, we have improved decay for derivatives of the solutions. More precisely, denote $\partial=\partial_t, \partial_{x^i}$ any translation, and $\bar{\partial}$ a derivative tangential to the cone $t=|x|$ such as $\partial_t+\partial_r$ or the projection on the angular derivatives of $\partial_{x^i}$, $\bar{\partial}_{x^i}=\partial_{x^i}-\frac{x^i}{r}\partial_{r}$. Then, we have the following proposition.

\begin{proposition}[Improved decay for derivatives of velocity averages of massless distribution functions] \label{pro:idmsl}
Let $f$ be a regular distribution function for the massless case solution to $\T_0(f)=0$ on $\mathbb{R}_t  \times \mathbb{R}_x^n \times \left(\mathbb{R}^n_v \setminus \{0 \}\right)$. Then, for all multi-indices $l,k$ and for all $(t,x) \in \mathbb{R}_t \times \mathbb{R}_x^n$, 
\begin{equation}\label{es:idmsl}
\vert \partial^{l} \bar{\partial}^{k}\rho_0(f) (t,x)\vert  \lesssim \frac{1}{(1+|t-|x|\,|)^{1+|l|} (1+|t+|x|\,|)^{n-1+|k|}}\Ab{f}_{\mathbb K,n+k+l}(0).
\end{equation}
\end{proposition}
\begin{proof}
This proof is similar to that of the improved decay estimates for the wave equation, and therefore omitted.
\end{proof}
\begin{remark}Note that the improved decay estimates \eqref{es:idmsl} only apply to velocity averages of $f$, because of the lack of regularity of velocity averages of $|f|$.
\end{remark}

Finally, let us mention that we can obtain decay for other moments of the solutions, provided the corresponding moments for $f$ and the $\widehat{Z}^\alpha(f)$ are finite initially. For instance, in Theorem \ref{th:demsl}, the decay estimate was written for the density of particles, while in Theorem \ref{de:msl}, we considered the energy density. One can move freely from one to the other, by considering $f |v|^q$ instead of $f$ (provided the initial data can support it of course).

\subsection{Klainerman-Sobolev inequalities and decay estimates: massive case} \label{se:ksmsv}
In the massive case $m>0$, we will prove
\begin{theorem}[Klainerman-Sobolev inequalities for velocity averages of massive distribution functions]\label{th:ksmsv} Let $f$ be a regular distribution function for the massive case defined on $\displaystyle \bigcup_{1 \le \rho < P}H_\rho \times \mathbb{R}^n_v$ for some $P \in [1,+\infty].$ Then, for all $(t,x) \in \displaystyle\bigcup_{1 \le \rho < P}H_\rho$, 
\begin{equation}\label{ineq:ksmsv}
\int_{v \in \mathbb{R}^n}|f|(t,x,v)\frac{dv}{v^0} \lesssim \frac{1}{(1+t)^n}\Ab{f}_{\mathbb P,n}(\rho(t,x)),
\end{equation}
where $\rho(t,x)=(t^2-|x|^2)^{1/2}$ and the norm $\Ab{f}_{\mathbb P,n}$ is defined as in Section \ref{se:tn}.
\end{theorem}
\begin{proof}
Recall from Remark \ref{re:cchi}, that 
\begin{equation} 
\chi_m(|f|)(t,x) \ge m^2\frac{t}{2 \rho}\int_v f d\mu_m=  m^2\frac{t}{2 \rho}\int_{v \in \mathbb{R}^n}|f|(t,x,v)\frac{dv}{v^0}
\end{equation}
and thus, note that 
\begin{equation}\label{eq:semsv}
\int_{H_\rho} \chi_m(|f|)(t,x) d\mu_{H_\rho} \ge \int_{H_\rho}  m^2\frac{t}{2 \rho}\int_{v \in \mathbb{R}^n}|f|(t,x,v)\frac{dv}{v^0}d\mu_{H_\rho}.
\end{equation}
Let $(t,x)$ be fixed in $\bigcup_{1 \le \rho \le P}H_\rho$ and define the function $\psi$ in the $(y^\alpha)$ system of coordinates (see end of Section \ref{se:tf}) as follows
$$
\psi(y^0,y^j)\equiv \int_{v \in \mathbb{R}^n} |f|(y^0,x^j+ty^j) d\mu_m.
$$
Similarly to the proof of the massless case, we apply first a one $1-d$ Sobolev inequality in the variable $y^1$
$$
\int_{v \in \mathbb{R}^n} |f|(y^0,x^j) d\mu_m=| \psi(y^0,0) | \lesssim \int_{|y^1| \le 1/(8n)^{1/2}}\left[ \left| \partial_{y^1} \psi \right|(y^0,y^1,0,..,0)+|\psi|(y^0,y^1,0,..,0) \right]dy^1.
$$
Now $\partial_{y^1} \psi= \frac{t}{t(y^0,x^1+t y^1,x^2,..,x^n)}\Omega_{01}$, where the $t$ in the numerator is that of the point $(t,x)$ while $t(y^0,x^j+t y^j)\equiv \left( (y^0)^2+(x^j +t y^j)^2 \right)^{1/2}$ is the time of the point defined in the $y^\alpha$ coordinates by $(y^0, x^j +t y^j)$. 
Now if $|x| \le t/2$, then it follows from the conditions $|y^1|\le \frac{1}{(8n^{1/2})} \le \frac{1}{8}$ that $(y^0)^2 \ge 3/4t^2$ and thus that $|\frac{t}{t(y^0,x^1+t y^1,x^2,..,x^n)}|\le C$ for some uniform $C>0$.
On the other hand if $|x| \ge t/2$, then it follows from the conditions $|y^1|\le \frac{1}{(8n^{1/2})} \le \frac{1}{8}$ that $| x^j+t y^j | \ge 3/8 t$, where $y^j=(y^1,0,..,0)$. Thus, we have for $|y^1| \le 1/(8n^{1/2})$
$$
\left| \partial_{y^1} \psi \right|(y^0,y^1,0,..,0) \lesssim \left|\int_v \Omega_{01}|f|(y^0,x^1+ty^1,x^2,..,x^n,v)  d\mu_m\right|.$$
The remainder of the proof is then similar to the massless case. We have 
$$
\left| \int_v \Omega_{01}(|f|) (y^0,x^1+ty^1,x^2,..,x^n,v)d\mu_m \right| \lesssim \sum_{|\alpha| \le 1}  \int_v |\widehat{Z}^\alpha f|(y^0,x^1+ty^1,x^2,..,x^n,v)d\mu_m.
$$
Inserting in the Sobolev inequality and repeating up to exhaustion of all the variables (the fact that for all $j$, $|y^j| \le \frac{1}{(8n^{1/2})}$, guarantees that $|y|=(\sum_{j=1}^n|y^j|^2)^{1/2} \le 1/8$ so that we still have $\frac{t}{t(y^0, x^j +t y^j)} \sim 1$), we obtain
\begin{eqnarray*}
\int_v |f| (y^0,x^1,x^2,..,x^n,v)d\mu_m  &\lesssim& \sum_{|\alpha| \le n} \int_{|y| \le 1/8} \int_v |\widehat{Z}^\alpha f|(y^0,x^j+ty^j,v)d\mu_m dy.
\end{eqnarray*}
Recall that the volume form on each of the $H_\rho$ is given in spherical coordinates by $\frac{\rho}{t}r^{n-1} dr d\sigma$, or in $y^\alpha$ coordinates by $\frac{y^0}{t}dy$. Thus, we have 
\begin{eqnarray*}
\int_v |f| (y^0,x^1,x^2,..,x^n,v)d\mu_m&\lesssim& \sum_{|\alpha| \le n} \int_{|y| \le 1/8} \int_v |\widehat{Z}^\alpha f|(y^0,x^j+ty^j,v)d\mu_m \frac{t(y^0,x^j+ty^j)}{y^0}d\mu_{H_\rho},\\
&\lesssim& \frac{t(y^0,x^j)}{y^0}\sum_{|\alpha| \le n}  \int_{|y| \le 1/8}\int_v |\widehat{Z}^\alpha f|(y^0,x^j+ty^j,v)d\mu_m d\mu_{H_\rho},
\end{eqnarray*}
where we have used again that $t(y^0,x^j+ty^j) \sim  t(y^0,x^j)$ in the region of integration. 
Applying the change of coordinates $z^j=t y^j$ and noticing that the quantities on the right-hand side are controlled by the estimate \eqref{eq:semsv} applied to $\widehat{Z}^\alpha(f)$ completes the proof.
\end{proof}
Since the norm on the right-hand side of \eqref{ineq:ksmsv} is conserved if $f$ is a solution to the massive transport equation, we obtain, as a corollary, the following pointwise decay estimate.

\begin{theorem}[Pointwise decay estimates for velocity averages of massive distribution functions] Let $f$ be a regular distribution function for the massive case satisfying the massive transport equation $\T_m(f)=0$ on $\displaystyle \bigcup_{1 \le \rho < +\infty}H_\rho \times \mathbb{R}^n_v$. Then, for all $(t,x) \in \bigcup_{1 \le \rho < +\infty}H_\rho$, 
$$\int_{v \in \mathbb{R}^n}|f|(t,x,v)\frac{dv}{v^0} \lesssim \frac{1}{(1+t)^n}\Ab{f}_{\mathbb P,n}(\rho=1).$$
\end{theorem}

Finally, let us mention the following improved decay for derivatives. 

\begin{proposition}[Improved decay estimates for derivatives of velocity averages of massive distribution functions] \label{pro:idmsv}
Let $f$ be a regular distribution function for the massive case satisfying the massive transport equation $\T_m(f)=0$ on $\displaystyle \bigcup_{1 \le \rho < +\infty}H_\rho \times \mathbb{R}^n_v$. Then, for all $i \in \mathbb{N}$, for all multi-indices $l$ and for all $(t,x) \in \bigcup_{1 \le \rho < +\infty}H_\rho$, 
$$\left| \nu_\rho^i \partial_{y}^l \int_{v \in \mathbb{R}^n} f (t,x,v)\frac{dv}{v^0} \right |\lesssim \frac{1}{(1+t)^{n+|l|} \rho^i}\Ab{f}_{\mathbb P,n+i+l}(\rho=1),$$
where $\nu_\rho=\frac{x^\alpha \partial_{x^\alpha}}{\rho}$ is the future unit normal to $H_\rho$ and $\partial_{y}^l$ is a combination of $|l|$ vector fields among the $\partial_{y^k}$, $1 \le k \le n$, which are tangent to the $H_\rho$.
\end{proposition}
\begin{proof}We have $\nu_\rho=\frac{S}{\rho}$ with $S$ the scaling vector field. On the other hand, recall that $S$ essentially commutes with the massive transport operator, so that in particular $\T_m (S(f))=0$ if $\T_m(f)=0$. Thus, $\int_v S(f) \frac{dv}{v^0}=S\left(\int_v f \frac{dv}{v^0}\right)$ satisfies the same decay estimates as $\int_v f \frac{dv}{v^0}$, which shows the improved decay for $\nu_\rho(\int_{v}f \frac{dv}{v^0})$. The higher order derivatives follow similarly. Indeed, using that $S(\rho)=\rho$, we have for instance $S^2 (f)= \rho^2 \nu_\rho^2(f)+ S(f)$. Applying the decay estimates for the velocity averages of $S^2(f)$ and $S(f)$ gives the correct improved decay for velocity averages of $\nu_\rho^2(f)$. Higher normal derivatives can be treated similarly. 
Finally, the improved decay for tangential derivatives of velocity averages is an easy consequence of the fact that $\partial_{y^k}=\frac{1}{t}\Omega_{0k}$.
\end{proof}

\section{Applications to the Vlasov-Nordstr\"om system} \label{se:avn}

\subsection{Generalities on the Vlasov-Nordstr\"om system}
In 1913, Nordstr\"om introduced a gravitation theory based on the replacement of the Poisson equation by a scalar wave equation. The Vlasov-Nordstr\"om system describes the coupling of this gravitational theory with collisionless matter\footnote{See \cite{MR1981446} for an introduction to the system.}.

It can be roughly obtained from the Einstein-Vlasov equations within the class of metrics conformal to the Minkowski metric by neglecting some of the non-linear self-interactions of the conformal factor. In dimension $n=3$, global existence for sufficiently regular massive distribution functions, with compact support in $(x,v)$ has been proven in \cite{MR2238881}.

Following \cite{MR1981446}, it is possible to make a derivation of this system for arbitrary mass, as well as arbitrary dimension. Consider the metric
$$
g= e^{2\phi}\eta
$$
conformal to the Minkowski metric $\eta$, where $\phi$ is a function on $\mathbb{R}^{n+1}$. For this system, the mass shell is defined by the equation
$$
e^{2\phi} \eta_{\al\be} v^\al v^\be = -m^2, \text{ which provides } v^0 = \sqrt{m^2e^{-2\phi} +\eta_{ij}v^i v ^j}.
$$
We can introduce the coordinates
$$
\hat{v}^i = e^{\phi} v^i,
$$
which consistently also provides
$$
\hat{v}^0 = \sqrt{m^2 +\eta_{ij}v^i v ^j} = e^{\phi} v^0.
$$
Considering distributions of particles which are conserved along the geodesic flow of $g$, we can define the associated transport operator as
\begin{eqnarray*}
\T_g &\equiv & v^\alpha \left(\dfrac{\partial}{\partial x^\alpha} -v^\beta\Gamma_{\alpha\beta} ^i\dfrac{\partial}{\partial v^i}\right)\label{eq:vlasovg},
\end{eqnarray*}
where $\Gamma^i_{\alpha \beta}$ are the Christoffel symbols of the metric $g$, which are given by
$$
\Gamma^i_{\alpha \beta} = \delta ^i_\alpha \dfrac{\partial \phi}{\partial x^\be}+\delta ^i_\be \dfrac{\partial \phi}{\partial x^\al} - \eta_{\al \be}\dfrac{\partial \phi}{\partial x^i}, 
$$
so that 
$$
\T_g= v^\alpha\dfrac{\partial}{\partial x^\alpha}- \left(2v^\alpha \nabla_\alpha \phi v^i+ m^2 e^{-2\phi}\nabla^i \phi\right)\dfrac{\partial}{\partial v^i}.
$$

In the $(t,x,\hat{v})$ system of coordinates, we compute
\begin{eqnarray*}
\T_g &=& e^{-\phi}\left(\hat{v}^\alpha\dfrac{\partial}{\partial x^\alpha}- \left(\hat{v}^\alpha \nabla_\alpha \phi \hat{v}^i+ m^2 \nabla ^i \phi\right)\dfrac{\partial}{\partial \hat v^i}\right).
\end{eqnarray*}

To couple the Vlasov field and the scalar function $\phi$, we follow \cite{MR1981446} and require that\footnote{It should not be surprising that the right-hand side of this equation vanishes for massless distribution functions, as, up to an overall factor, it corresponds to the trace of the energy-momentum tensor, the latter being of course proportional to $m$ for Vlasov fields.}% given by 
\begin{eqnarray}
\square \phi &=& m^2 e^{(n+1)\phi} \int_{v} f\frac{dv}{v^0}. \label{eq:nlwavevlasov} 
\end{eqnarray}

Depending on the value of the mass $m$, we are thus faced with the following two systems.\\

\textbf{The massless Vlasov-Nordstr\"om system}
\begin{eqnarray}
\square \phi &=& 0, \label{eq:vmslp}\\
\hat{v}^\alpha\dfrac{\partial   f}{\partial x^\alpha}- \hat{v}^\alpha \nabla_\alpha \phi \hat{v}^i\dfrac{\partial  f }{\partial \hat v^i}&=& 0. \label{eq:vmslf}
\end{eqnarray}
In this case, the equations decouple. We can of course solve the first equation and then think of the second equation as a linear transport equation for $f$.\\

\textbf{The massive Vlasov-Nordstr\"om system}\\
In this case, we can perform yet another change of unknown by considering 
$$
\tilde f (t,x,\hat v) \equiv e^{(n+1)\phi}  f(t,x,\hat v), 
$$
which has the advantage of removing the $\phi$ dependence in the right-hand side of equation \eqref{eq:nlwavevlasov}.

We then obtain the usual expression of the (massive) Vlasov-Nordstr\"om system
\begin{eqnarray}
\square \phi &=& m^2\int_{\hat v} \tilde{f} \dfrac{d \hat {v}}{\hat{v}^0}, \\
\hat{v}^\alpha\dfrac{\partial  \tilde{f}}{\partial x^\alpha}- \left(\hat{v}^\alpha \nabla_\alpha \phi \hat{v}^i+ m^2 \nabla ^i \phi\right)\dfrac{\partial \tilde{f} }{\partial \hat v^i}&=& (n+1) \tilde{f} \hat{v}^\al\dfrac{\partial  \phi }{\partial x^\alpha}.
\end{eqnarray}
From now on, we will drop the $\,\,\hat{\,\,}$ and $\,\,\tilde{\,\,}$ on all the variables to ease the notations.

\subsection{The massless Vlasov-Nordstr\"om system}
We consider in this section, the system \eqref{eq:vmslp}-\eqref{eq:vmslf}. We will denote by $\T_\phi$ the transport operator defined by
$$
\T_\phi \equiv v^\alpha\dfrac{\partial   }{\partial x^\alpha}- v^\alpha \nabla_\alpha \phi\cdot  v^i\dfrac{\partial  }{\partial  v^i},
$$
i.e.~
$$
\T_\phi=\T_0-\T_0(\phi)\cdot v^i \partial_{v^i}.
$$
The massless Vlasov-Nordstr\"om system can then be rewritten as 
\begin{eqnarray}
\square \phi&=&0, \label{eq:wp} \\
\T_\phi (f)&=&0, \label{eq:tfmsl}
\end{eqnarray}
which we complement by the initial conditions
\begin{eqnarray}
&&\hbox{}\phi(t=0)=\phi_0, \quad \partial_t \phi(t=0)=\phi_1, \label{eq:idcp}\\
&&\hbox{}f(t=0)=f_0, \label{eq:idcf}
\end{eqnarray}
where $(\phi_0, \phi_1)$ are sufficiently regular functions defined on $\mathbb{R}^n_x$ and $f_0$ is a sufficiently regular function defined on $\mathbb{R}^n_x \times \left( \mathbb{R}^n_v \setminus \{ 0 \} \right)$.

By sufficiently regular, we mean that all the computations below make sense. We will eventually require that $\mathcal{E}_N[\phi_0, \phi_1] < +\infty $ where $\mathcal{E}_N$ is the energy norm defined by \eqref{def:tnp} and similarly, we will also require below that $|| f_0 ||_{\mathbb{K},N} < +\infty$ (with some additional weights in the case of dimension $3$). Provided $N$ is large enough (depending only on $n$), these two regularity requirements are then enough so that all the computations below are justified. In the remaining, we will therefore omit any further mention of regularity issues. 

\subsubsection{Commutation formula for $\T_\phi$} \label{se:cftp}
Recall the algebra of commutation fields $\widehat{\mathbb K}_0=\widehat{\mathbb{K}} \cup \{ S \}$, where $S$ is the usual scaling vector field, defined in \eqref{def:k0h}. Similar to Lemma \ref{lem:cpcl}, we have 

\begin{lemma} \label{prop:commut1}
For any $\widehat{Z} \in \widehat{\mathbb K}_0$, 

\begin{eqnarray*}
[\T_\phi, \widehat{Z}]&=&c_Z \T_0 +\left[ - c_Z \T_0(\phi)+ \T_0( Z(\phi) ) \right]v^i \partial_{v^i}, \\
&=&c_Z \T_\phi +\T_0( Z(\phi) )v^i \partial_{v^i},
\end{eqnarray*}
where $c_Z=0$ if $\widehat{Z} \in \widehat{ \mathbb{K}}$ and $c_Z=1$ if $\widehat{Z}=S$, and where $Z$ is the non-lifted field corresponding to $\widehat{Z}$ if $\widehat{Z} \in \widehat{ \mathbb{K}}$ and $Z=S$ if $\widehat{Z}=S$.
\end{lemma}
\begin{proof}
Note first that for all $\widehat{Z} \in \widehat{\mathbb{K}}_0$, we have $[\widehat{Z}, v^i \partial_{v^i}]=0$.
We then compute
\begin{eqnarray*}
\left[ \T_\phi, \widehat Z \right] & = & [\T_0, \widehat Z] - \left[ \T_0(\phi) v^i\pv i, \widehat Z \right]\\
&=&[\T_0, \widehat Z]+\widehat Z\left( \T_0 (\phi) \right) v^i\pv i +\T_0(\phi)\left[\widehat Z, v^i\pv i\right]\\
&=&\left[\T_0, \widehat Z\right] +\left(\left[\widehat Z, \T_0\right] \phi +\T_0(\widehat Z \phi)\right) v^i\pv i.
\end{eqnarray*}
To conclude the proof of the lemma,  replace all the instances of $\left[\T_0, \widehat Z\right]$ by $c_Z \T_0$ according to Lemma \ref{lem:cpcl}.
\end{proof}

Iterating the above, one obtains
\begin{lemma}\label{prop:commut} 
Let $f$ be a solution to \eqref{eq:tfmsl}. For any multi-index $\alpha$, we have the commutator estimate
\eq{ \label{es:commut}
	\left\vert\left[ \T_\phi , \widehat{Z}^\alpha \right] f \right\vert\leq C\sum_{\substack{|\beta| \le |\alpha|, |\gamma| \le |\alpha|,\\ |\beta|+|\gamma| \le |\alpha|+1 }} \vert\T_0(Z^{\gamma} \phi)\vert\cdot \vert  \widehat{Z}^\be f\vert,
}
where the $Z^{\gamma} \in \mathbb{K}_0^{|\gamma|}$ and the $\widehat{Z}^\be \in \widehat{\mathbb{K}}_0^{|\beta|}$ and $C>0$ is some constant depending only on $|\alpha|$.
\end{lemma}

\subsubsection{Approximate conservation law}
Similar to Lemma \ref{lem:acl}, we have 
\begin{lemma} \label{lem:aclmsltp}
Let $h$ be a regular distribution function for the massless case in the sense of Section \ref{se:rdf}. 
Let $g$ be a regular solution to $\T_\phi(g)=v^0 h$, with $v^0=|v|$, defined on $[0,T] \times \mathbb{R}^n_x \times \left( \mathbb{R}^n_v\setminus \{0 \} \right)$ for some $T>0$. Then, for all $t \in [0,T]$,
\begin{equation} \label{eq:acltpmsl}
\begin{aligned}
&\int_{\Sigma_t}\rho_0(|g|)(t,x)dx \\
&\quad\le \int_{\Sigma_0}\rho_0(|g|)(0,x)dx+\int_{0}^t\int_{\Sigma_s} \rho_0(|h|)(s,x)dx ds+(n+1)\int_{0}^t \int_{\Sigma_s} \int_{v \in \mathbb{R}^n_v \setminus \{ 0 \}} |\T_0(\phi)f| dv dx ds.
\end{aligned}
\end{equation}
\begin{proof}As for the proof of Lemma \ref{lem:acl}, this follows, after regularization of the absolute value, from integration by parts or an application of Stoke's theorem. The term $\T_0( \phi) v^i \partial_{v^i}|f|$, which appears in the computation, gives rise after integration by parts in $v$ to the last term in \eqref{eq:acltpmsl} since $\partial_{v^i} \left(v^i \T_0( \phi) \right)= (n+1) \T_0 ( \phi) $.
\end{proof}
\end{lemma}

\subsubsection{Massless case in dimension $n \ge 4$} \label{se:mcd4}

In this section, we prove Theorem \ref{th:asmsl4}. The $n=3$ case requires slightly more refined techniques (which of course would work also for $n \ge 4$) but the estimates in the $n \ge 4$ case are slightly stronger and simpler, we therefore provide an independent proof. 

If $\phi$ is a solution to the wave equation, let us consider the energy at time $t=0$ 
\begin{equation}\label{eq:defenergy0}
\mathcal{E}_N[\phi](t=0)\equiv \sum_{|\al | \le N, |\alpha| \in \K^{|\alpha|} } \left | \left | Z^\al ( \partial \phi)(t=0) \right |\right |^2_{L^2(\mathbb{R}^n_x)}.
\end{equation}
Now if $\phi(t=0)=\phi_0$ and $\partial_t \phi (t=0)=\phi_1$, for pairs of sufficiently regular functions $(\phi_0, \phi_1)$ defined on $\mathbb{R}^n_x$, then the above quantity can be computed purely in terms of $\phi_0$, $\phi_1$, so we define\footnote{The alternative to the approach we use here is to assume that $(\phi_0, \phi_1)$ are regular initial data with decay fast enough in $x$, for instance by assuming compact support, so that the resulting $\mathcal{E}_N[\phi(t=0)]$ is finite. What we want to emphasize here is that the quantity $\mathcal{E}_N[\phi(t=0)]$ can in fact be computed purely in terms of the initial data (using the equation to rewrite second and higher time derivatives of $\phi$ in terms of spatial derivatives), and that this is all that is needed in terms of decay in $x$.}

\begin{equation} \label{def:tnp}
\mathcal{E}_N[\phi_0, \phi_1]\equiv \mathcal{E}_N[\phi](t=0).
\end{equation}

Similarly, if $f$ is a solution to \eqref{eq:tfmsl} arising from initial data $f_0$ at $t=0$, then we define
\begin{equation}\label{eq:defenergy01}
E_N[f](t=0) \equiv ||f||_{\mathbb{K},N}(t=0)\left(=\sum_{|\al | \le N, \widehat{Z}^\alpha \in \widehat{\K}^{\alpha}}  \left | \left | \rho_0\left ( \widehat{Z}^\al ( f ) (t=0)\right) \right | \right |_{L^1(\mathbb{R}^n_x)} \right)
\end{equation}
and we remark that this quantity can be computed purely in terms of $f_0$ so we will set $$E_N[f_0]\equiv E_N[f](t=0).$$

We will prove

\begin{theorem} Let $n \ge 4$ and let $N \ge \frac{3n}{2}+1$. Let $(\phi_0, \phi_1, f_0)$ be an initial data set for the massless Vlasov-Nordstr\"om system such that $\mathcal{E}_N[\phi_0, \phi_1] + E_N[f_0] < +\infty$. Then, the unique solution $(f,\phi)$ to \eqref{eq:wp}-\eqref{eq:tfmsl} satisfying the initial conditions \eqref{eq:idcp}-\eqref{eq:idcf} verifies the estimates
\begin{enumerate}
\item Global bounds: for all $t \ge 0$, 
$$
E_N[f](t) \le e^{C\mathcal{E}_N^{1/2}[\phi_0, \phi_1]}E_N[f_0],
$$
where $C>0$ is a constant depending only on $N,n$.
\item Pointwise estimates for velocity averages: for all $(t,x) \in [0,+\infty) \times \mathbb{R}^n_x$ and all multi-indices $\alpha$ satisfying $|\alpha| \le N-n$, 
$$
\rho_0 ( | \widehat{Z}^\alpha f |)(t,x)\lesssim \frac{e^{C\mathcal{E}_N^{1/2}[\phi_0, \phi_1]}E_N[f_0]} {(1+|t-|x|\,|) (1+|t+|x|\,|)^{n-1}}.
$$
\end{enumerate}
\end{theorem}
\begin{proof}
Let $N,n,\phi_0, \phi_1, f_0$ be as in the statement of the theorem. From the conservation of energy and the commutation properties of the $Z^\alpha$ with the wave operator, we have, for all $t$,
$$ 
\mathcal{E}_N[\phi](t) = \mathcal{E}_N[\phi_0, \phi_1] \equiv \mathcal{E}_N.
$$
Applying the standard decay estimates obtained via the vector field method to $\phi$,  we have for all multi-indices $\alpha$ satisfying $|\alpha| \le N-(n+2)/2$ and for all $(t,x) \in \mathbb{R}_t \times \mathbb{R}^n_x$

\begin{eqnarray} \label{eq:sdevfm}
|\partial Z^\alpha \phi (t,x) |^2 \lesssim \frac{\mathcal{E}_N[\phi](t)}{(1+|t-|x|\,|) (1+|t+|x|\,|)^{n-1}}.
\end{eqnarray}
Note that it follows from a standard existence theory for regular data that for all $t$, $E_N[f(t)] < +\infty$.

Applying the Klainerman-Sobolev inequality \eqref{ineq:msl}, we obtain, for all multi-indices $\alpha$ satisfying $|\alpha| \le N-n$ and for all $(t,x) \in \mathbb{R}_t \times \mathbb{R}^n_x$,

$$
|\rho_0 ( \widehat{Z}^\alpha(f))(t,x)| \lesssim \frac{E_N[f](t)} {(1+|t-|x|\,|) (1+|t+|x|\,|)^{n-1}}.
$$

From Lemma \ref{lem:aclmsltp} and the commutator estimate \eqref{es:commut}, we have for all $t \ge 0$ and all multi-indices $\alpha$,

\begin{equation} \label{eq:esza}
  \int_{\Sigma_t}\rho_0(|\widehat{Z}^\alpha f|)(t,x)dx \le \int_{\Sigma_0}\rho_0(|\widehat{Z}^\alpha f|)(0,x)dx+\int_{0}^t\int_{\Sigma_s} \rho_0(|h^\alpha|)(s,x)dx ds, 
\end{equation} 
where\footnote{Note that the last term in the right-hand side of Lemma \ref{lem:aclmsltp} is similar to the error terms arising from the commutator estimate of Lemma \ref{prop:commut} and is therefore accounted for in the $h^\alpha$ error term in equation \eqref{eq:esza}.}

\begin{eqnarray*}
|h^\alpha| &\lesssim& \frac{1}{v^0}\sum_{\substack{|\beta| \le |\alpha|, |\gamma| \le |\alpha|,\\ |\beta|+|\gamma| \le |\alpha|+1 }} \vert\T_0(Z^{\gamma} \phi)\vert\cdot \vert  \widehat{Z}^\be f\vert \\
&\lesssim& \sum_{\substack{|\beta| \le |\alpha|, |\gamma| \le |\alpha|,\\ |\beta|+|\gamma| \le |\alpha|+1 }} \vert
\partial (Z^{\gamma} \phi)\vert\cdot \vert  \widehat{Z}^\be f\vert, \\
\end{eqnarray*}
so that 
$$
\rho_0(|h^\alpha|) \lesssim \sum_{\substack{|\beta| \le |\alpha|, |\gamma| \le |\alpha|,\\ |\beta|+|\gamma| \le |\alpha|+1 }} \vert
\partial ( Z^{\gamma} \phi )\vert \rho_0   \left (\vert  \widehat{Z}^\be f\vert \right), 
$$
since $\phi$ is independent of $v$. Integrating over $x$, we obtain, for all $s \in [0,t]$, 
$$
\int_{\Sigma_s} \rho_0(|h^\alpha|)(s,x) dx \lesssim \sum_{\substack{|\beta| \le |\alpha|, |\gamma| \le |\alpha|,\\ |\beta|+|\gamma| \le |\alpha|+1 }} \int_{\Sigma_s} \vert
\partial (Z^{\gamma} \phi)\vert \rho_0 \left ( \vert  \widehat{Z}^\be f\vert \right)(s,x) dx.
$$

We now estimate each term in the above sum depending on the values of $|\gamma|$ and $|\beta|$.

If $|\beta| \le N-n$, we then apply the pointwise estimates on $\rho_0( \widehat{Z}^\beta(f))$ to obtain

$$
\int_{\Sigma_s} \vert
\partial (Z^{\gamma} \phi)\vert \rho_0 \left ( \vert  \widehat{Z}^\be f\vert \right)(s,x) dx \lesssim \int_{\Sigma_s} \vert
\partial (Z^{\gamma} \phi)\vert \frac{E_N[f](s)} {(1+|s-|x|\,|) (1+|s+|x|\,|)^{n-1}} dx.
$$
Applying the Cauchy-Schwarz inequality and using that\footnote{For the convenience of the reader, we have added in Appendix \ref{app:ie} certain integral estimates which include \eqref{es:l2p}.} 
\begin{equation} \label{es:l2p}
\left|\left| \frac{1} {(1+|s-|x|\,|) (1+|s+|x|\,|)^{n-1}} \right|\right|_{L^2_x} \lesssim \frac{1}{(1+s)^{(n-1)/2}},
\end{equation}
we obtain
\begin{eqnarray}\label{es:er4msl}
\int_{\Sigma_s} \vert
\partial (Z^{\gamma} \phi)\vert \rho_0 \left ( \vert  \widehat{Z}^\be f\vert \right)(s,x) dx \lesssim \mathcal{E}^{1/2}_N[\phi](s) \frac{E_N[f](s)} {(1+s)^{(n-1)/2}}.
\end{eqnarray}
If now $|\beta| > N-n$, then $ |\gamma| \le |\alpha|+1-|\beta|  \le N-(n+2)/2$ and thus, we also have \eqref{es:er4msl}, using this time the pointwise estimates on $\partial (Z^\gamma \phi)$ given by \eqref{eq:sdevfm}. Applying Gr\"onwall's inequality finishes the proof of the theorem. 
\end{proof}

 \subsubsection{Massless case in dimension $n=3$} \label{se:mcd3}
 We now turn to the case of the dimension $3$, where the slower pointwise decay of solutions to the wave equations leads to a slightly harder analysis. First, let us strengthen our norms for the Vlasov field. 

For this, recall the algebra of weights $\mathbb{k}_0$ introduced in Section \ref{se:wpf} and define a rescaled version $\kappa_0$ by
$$
\kappa_0 \equiv (v^0)^{-1} \mathbb{k}_0= \left\{ \frac{\zz}{v^0} \, / \, \zz \in \mathbb{k}_0  \right\},
$$
where we recall that $v^0=|v|$ in the massless case. If $\alpha$ is a multi-index, we will write $\left[\frac{\zz}{v^0}\right]^{\alpha} \in \kappa_0^{|\alpha|}$ to denote a product $|\alpha|$ elements of $\kappa_0$ and $\left[\frac{|\zz|}{v^0}\right]^{\alpha}$ in case we take the product of the absolute values of these elements.

Let us now define, for any regular distribution function $f$, the weighted norm%$g$ defined on $\mathbb{R}^n_x \times\left( \mathbb{R}^n_v \setminus \{0 \} \right)$

\begin{eqnarray} \label{def:n3d}
%\Ab{g}_{\mathbb K,k,q}
E_{N,q}[f] &&\equiv  \sum_{\substack{|\al|\leq N,\\ |\beta| \le q}}\sum_{\widehat{Z}^\alpha \in \widehat{\K}^{|\alpha|}}\int_{\Sigma_t} \rho_0 \left(\ab{\widehat{Z}^\al  f} \left[\frac{|\zz|}{v^0} \right]^\beta \right)(x)  dx\\
&&\left(=\sum_{\substack{|\al|\leq N, \\|\beta| \le q}}\sum_{\widehat{Z}^\alpha \in \widehat{\K}^{|\alpha|}}\int_{\Sigma_t} \int_{v \in \mathbb{R}^n \setminus \{ 0 \}} \left(\ab{\widehat{Z}^\al  f}(x,v) \left[\frac{|\zz|}{v^0} \right]^\beta \right) v^0 dv dx\right),\nonumber
\end{eqnarray}
where the weights $\frac{\zz}{v^0}$ lie in $\kappa_0$. 

\begin{theorem}[Asymptotic behaviour in dimension $n=3$] Consider the dimension $n=3$. Let $N \ge 7$ and $q \ge 1$. Let $(\phi_0, \phi_1, f_0)$ be an initial data set for the massless Vlasov-Nordstr\"om system such that $\mathcal{E}_N[\phi_0, \phi_1] + E_N[f_0]_{N,q} < +\infty$. Then, the unique solution $(f,\phi)$ to \eqref{eq:wp}-\eqref{eq:tfmsl} satisfying the initial conditions \eqref{eq:idcp}-\eqref{eq:idcf} verifies the estimates
\begin{enumerate}
\item Global bounds with growth for the top order norms: for all $t \in \mathbb{R}_t$, 
\begin{equation} \label{eq:gbg}
E_{N,q}[f](t) \le (1+t)^{C\mathcal{E}_N^{1/2}[\phi_0, \phi_1]}E_{N,q}[f_0],
\end{equation}
where $C>0$ is a constant depending only on $N,n$ and $q$.
\item Small data improvement for the low order norms:  there exists an $\varepsilon_0$ (depending only on $n,N,q$) such that if $\mathcal{E}_N[\phi_0,\phi_1] \le \varepsilon_0$, then for all $t \in \mathbb{R}_t$,
\begin{equation} \label{eq:gbwg}
E_{N-(n+4)/2,q-1}[f](t)\le e^{C\mathcal{E}_N^{1/2}[\phi_0, \phi_1]}E_{N,q}[f_0].
\end{equation}
\item Under the above smallness assumption, we also have the optimal pointwise estimates for velocity averages: for all $(t,x) \in \mathbb{R}_t \times \mathbb{R}^n_x$ and all multi-indices $\alpha$ satisfying $|\alpha| \le N-(3n+4)/2$ and all $|\beta| \le q-1$, 
$$
\rho_0 \left( \left| \widehat{Z}^\alpha(f) \left[\frac{\zz}{v^0} \right]^\beta\right| \right)(t,x) \lesssim \frac{e^{C\mathcal{E}_N^{1/2}[\phi_0, \phi_1]}E_{N,q}[f_0]} {(1+|t-|x|\,|) (1+|t+|x|\,|)^{n-1}}.
$$
\end{enumerate}
\end{theorem}

\begin{proof}
First, let us note that for all $\zz \in \mathbb{k}_0$, we have 
$$
v^i \partial_{v^i} \left(\frac{\zz}{v^0} f \right) = \frac{\zz}{v^0} v^i \partial_{v^i}f,
$$
from which it follows that for all regular distribution functions $g$, $\left[\T_\phi, \frac{\zz}{v^0}\right]g=0$.
Thus, we can upgrade Lemma \ref{prop:commut} to

\eq{
	\left\vert\left[ \T_\phi , \left[\frac{\zz}{v^0} \right]^\sigma \widehat{Z}^\alpha \right]  f \right\vert\leq C\sum_{\substack{|\beta| \le |\alpha|, |\gamma| \le |\alpha|,\\ |\beta|+|\gamma| \le |\alpha|+1 }} \vert\T_0(Z^{\gamma} \phi)\vert\cdot \left[\frac{|\zz|}{v^0} \right]^\sigma \vert   \widehat{Z}^\be  f\vert,
}
where the $Z^{\gamma} \in \mathbb{K}_0^{|\gamma|}$, the $\widehat{Z}^\be \in \widehat{\mathbb{K}}_0^{|\beta|}$, $\left[\frac{\zz}{v^0} \right]^\sigma \in \kappa_0^{|\sigma|}$ and $C>0$ is some constant depending only on $|\alpha|$. Applying arguments similar to those used in the $n \ge 4$ case yield
\begin{eqnarray}
E_{N,q}[f](t)  &\le& E_{N,q}[f_0] +C\int_0^t \sum_{\substack{|\beta| \le |\alpha|, |\gamma| \le |\alpha|,\\ |\beta|+|\gamma| \le |\alpha|+1 }} \sum_{|\sigma| \le q} \int_{\Sigma_s} \vert \partial (Z^{\gamma} \phi)\vert \rho_0 \left ( \left[\frac{|\zz|}{v^0} \right]^\sigma \vert  \widehat{Z}^\be f\vert \right)(s,x) dx ds  \nonumber \\
&\le& E_{N,q}[f_0]+C\int_0^t \mathcal{E}_N^{1/2}  \frac{E_{N,q}[f](s)} {(1+s)}ds. \label{es:er3msl}
\end{eqnarray}
Applying Gr\"onwall inequality, we obtain \eqref{eq:gbg}. 

Now assume that $\mathcal{E}_N \le \varepsilon_0$ with $\varepsilon_0$ small enough so that,  
$$
E_{N,q}[f](t)\le (1+t)^{\delta}E_{N,q}[f_0],
$$
with $\delta =C \mathcal{E}_N^{1/2}< 1/2$.

The key to the improved estimates is the following decomposition of the transport operator $\T_0$:
\begin{eqnarray*}
\T_0 &=& v^0 \partial_t+ v^i \partial_{x^i} = v^0 \left( \partial_t + \frac{x^i}{|x|} \partial_{x^i} \right) - v^0 \frac{x^i}{|x|} \partial_{x^i}+v^i \partial_{x^i}\\%\dfrac{v^0}{t}S+\dfrac{v^i t- x^iv^0}{t}\dfrac{\partial}{\partial x^i} \\
&=& v^0\left(\partial_t +\dfrac{x^i}{|x|}\partial_{x^i}\right)
+\dfrac{v^0x^i}{t|x|}(|x|-t)\partial_{x^i}
+\dfrac{v^i t- x^iv^0}{t}\partial_{x^i}\\
&=& v^0\underbrace{\left(\partial_t +\dfrac{x^i}{|x|}\partial_{x^i}\right)}_{\text{outgoing derivatives}}
-\dfrac{v^0}{t}\underbrace{\dfrac{x^i}{|x|}}_{\text{bounded}}\underbrace{\left(\dfrac{-x^j\Omega_{ij}+t\Omega_{0i}-x^iS}{t+r}\right)}_{\leq C\left(|\Omega_{ij}|+|\Omega_{0i}|+|S|\right)}
+v^0\dfrac{\zz}{v^0 t}\partial_{x^i},
\end{eqnarray*}
where the weight $\zz$ in the last term is $v^i t -x^i v^0 \in \mathbb{k}_0$.
Recall\footnote{This can be obtained from the usual Klainerman-Sobolev inequality and the formula for $\partial_s$ in \eqref{id:vfi2} by integration along the constant $t=|x|$ null lines. See for instance \cite{qw:lnlw} for details.} now the following improved decay for outgoing derivatives of solutions to the wave equations: for all multi-indices $\alpha$ such that $|\alpha| \le N-(n+2)/2-1$,
$$
\left| \left( \partial_t + \frac{x^i}{r} \partial_{x^i} \right) Z^\alpha(\phi) \right| \lesssim \frac{\mathcal{E}_N}{(1+t)^{3/2}}.
$$
As a consequence, it follows that for all multi-indices $|\alpha| \le N-(n+2)/2-1$, 

$$|\T_0(Z^\alpha \phi)| \lesssim \mathcal{E}_N v^0 \left( \frac{1}{(1+t)^{3/2}}+ \sum_{\zz \in \mathbb{k}_0} \frac{|\zz|}{v^0}\frac{1}{t(1+t)} \right).
$$
Repeating the previous ingredients then gives \eqref{eq:gbwg}. The pointwise estimates then follow from the Klainerman-Sobolev inequality \eqref{ineq:ksmsv}.
\end{proof}

\subsection{The massive Vlasov-Nordstr\"om system} \label{se:msvcd4}
We now turn to the massive case, that is to say the system
\begin{eqnarray*}
\square \phi &=& m^2\int_{v} f \dfrac{d v}{v^0}  = m^2 \rho_1\left(\dfrac{f}{(v^0)^2}\right)\\
\T_m(f)-\left( \T_m( \phi)v^i +m^2 \nabla^i \phi \right) \frac{ \partial f}{\partial v^i}&=& (n+1) \T_m (\phi) f.
\end{eqnarray*}
As in the massless case, we introduce the notation $\T_\phi \equiv \T_m-\left( \T_m( \phi)v^i +m^2 \nabla^i \phi \right) \frac{ \partial }{\partial v^i}$ for the transport operator that appears on the left-hand side of the last equation. With this notation, we will seek for solutions of the massive Vlasov-Nordstr\"om system\begin{eqnarray}
\square \phi &=& m^2 \int_{v} f \dfrac{d v}{v^0}\label{eq:wpmsv} \\
\T_\phi(f)&=& (n+1) f\, \T_m (\phi) \label{eq:tfmsv}
\end{eqnarray}
completed by the initial conditions
\begin{eqnarray}
\phi_{|H_1}=\phi_0, \quad \partial_t\phi_{|H_1}= \phi_1 \\
f_{| H_1 \times \mathbb{R}^n_v}=f_0. \label{ic:fmsv}
\end{eqnarray}

 As for the massless case, the lower the dimension, the harder it is to close the estimates. We consider here only the dimensions $n \ge 4$. As already explained, to treat the case $n=3$, we need a refinement of our method, for instance, the use of modified vector fields in the spirit of \cite{js:sdsvpsvfm} and we postpone this to future work. The proof that we shall give below will be enough to close the estimates for $n = 4$ with some $\varepsilon$ growth in the norms, and without any growth if $ n > 4$.

In the following, we will set the mass $m=1$.

\subsection{The norms} \label{se:tnvn}

In the context of the massive Vlasov-Nordstr\"om system, we define the following energies, similar to the energies defined in \eqref{eq:defenergy0} and \eqref{eq:defenergy01}:
\begin{itemize}
\item for the field $\phi$, satisfying a wave equation:
\begin{equation}
\mathcal{E}_N[\phi](\rho)\equiv  \sum_{|\al | \le N, Z^\alpha \in \mathbb{P}^{|\alpha|} }\int_{H_\rho} T[Z ^\al \phi](\partial_t, \nu_\rho) d\mu_{H_\rho},
\end{equation}
where, for any scalar function $\psi$ we denote by $T[\psi]= d \psi \otimes d \psi -\frac{1}{2}\eta( \nabla \psi, \nabla \psi) \eta$ its energy-momentum tensor.

\item for the field $f$, satisfying a transport equation:
$$
 E_N[f](\rho)\equiv \sum_{|\al | \le \lfloor N/2\rfloor, Z^\alpha \in \widehat{\mathbb{P}}^{\alpha}}  \left | \left | \chi_1\left (\left(v^0\right)^2\vert \widehat{Z}^\al ( f ) \vert\right) \right | \right |_{L^1(H_\rho)} + \sum_{ \lfloor N/2\rfloor+1 \leq |\al | \le N, Z^\alpha \in \widehat{\mathbb{P}}^{\alpha}}  \left | \left | \chi_1\left (\vert  \widehat{Z}^\al ( f )\vert  \right) \right | \right |_{L^1(H_\rho)},
$$
where for any regular distribution function $g$, $\chi_1(g)$ is defined as in Section \ref{se:pvfset}. 
\end{itemize}
\begin{remark} The weight on the lower order derivatives contained in the norm of $f$ ensures that pointwise estimates can be performed on terms of the form 
$$
\int_v \left |   v^0 \widehat{Z}^\al f(t,x,v) \right|  dv\lesssim \frac{E_N[f](\rho)}{t^n},
$$
accordingly to Theorem \ref{th:ksmsv}, provided that $\vert \al \vert \leq \lfloor \frac{N}{2}\rfloor-n$. It should furthermore be noticed that the "unweighted" standard estimates coming from Theorem \ref{th:ksmsv} are still true for $\vert \al \vert\leq N-n$:
$$
\int_v \left |  \widehat{Z}^\al f(t,x,v) \right | \frac{dv}{v^0}\lesssim \frac{E_N[f](\rho)}{t^n}.
$$
They will nonetheless not be used in the following. 
\end{remark}

\subsubsection{The main result}
Our main result for the massive Vlasov-Nordstr\"om system is contained in the following theorem.

\begin{theorem}\label{th:mtmsv}
Let $n \ge 4$ and $N \ge 3n+4$. Let $(f_0, \phi_0, \phi_1)$ be an initial data set for the system \eqref{eq:wpmsv}-\eqref{ic:fmsv}. Then, there exists an $\varepsilon_0 > 0$ such that, for all $0 \le \varepsilon < \varepsilon_0$,  if
\begin{itemize}
\item  $\mathcal{E}_N[ \phi_0, \phi_1] \le \varepsilon$, (Initial regularity of $\phi$)
\item  $E_{N+n}[f_0] \le \varepsilon$, (Initial regularity of $f$) 
\end{itemize}
then, the unique classical solution $(f, \phi)$ of \eqref{eq:wpmsv}-\eqref{ic:fmsv} exists in the whole of the future unit hyperboloid and verifies the estimates
\begin{enumerate}
\item Energy bounds for $\phi$: for all $\rho \ge 1,$
$$\mathcal{E}_N[\phi](\rho) \le 2 \varepsilon.$$
\item Global bounds for $f$ at order less than $N$: for all $\rho \ge 1$, 
$$E_N[f](\rho) \le \rho^{C\varepsilon^{1/4}} 2 \varepsilon,$$
where the constant $C=1$, when $n=4$, and $0$ when $n>4$.
\item Pointwise decay for $ \partial Z^\alpha\phi$: for all multi-indices $|\alpha|$ such that $|\alpha | \le N-\tfrac{n+2}{2}$ and all $(t,x)$ with $t \ge \sqrt{1+|x|^2}$, we have 
$$
\left| \partial Z^\alpha\phi \right| \lesssim \frac{\varepsilon}{(1+t)^{(n-1)/2}(1+(t-|x|)^{1/2}}.
$$
\item Pointwise decay for $ \rho\left( \left|\partial Z^\alpha f \right| \right)$: for all multi-indices $\alpha$ and $\beta$ such that $|\alpha | \le N-n $, $|\beta|\leq \lfloor N/2\rfloor-n $ and all $(t,x)$ with $t \ge \sqrt{1+|x|^2}$, we have 
$$
\int_v \left|\widehat{ Z}^\alpha f \right| \frac{dv}{v^0} \lesssim \frac{\varepsilon}{(1+t)^{n-\varepsilon^{1/4}}} \text{ and } \int_v v^0\left|\widehat{ Z}^\beta f \right| dv \lesssim \frac{\varepsilon}{(1+t)^{n-\varepsilon^{1/4}}},
$$
where the constant $C=1$, when $n=4$, and $0$ when $n>4$.
\item Finally, the following $L^2$ estimates on $f$ hold: for all multi-indices $\alpha$ with $\lfloor N/2\rfloor-n +1 \le |\al| \le N$, and all $(t,x)$ with $t \ge \sqrt{1+|x|^2}$, we have 
$$
\int_{H_\rho} \frac{t}{\rho}\left(\int_{v}|\widehat{Z}^{\al} f| \frac{dv}{v^0}\right)^{2} d\mu_{H_\rho} \lesssim \varepsilon^2 \rho^{C\varepsilon^{1/4} +1-n},
$$
where the constant $C=1$, when $n=4$, and $0$ when $n>4$.
\end{enumerate}
\end{theorem}

\subsection{Proof of Theorem \ref{th:mtmsv}}

\subsubsection{Structure of the proof and the bootstrap assumptions}
From now on, we consider a solution $(f, \phi)$ to \eqref{eq:wpmsv}-\eqref{ic:fmsv} arising from initial data satisfying the requirements of Theorem \ref{th:mtmsv}.
Let $P$ be the largest (hyperboloidal) time so that the following bootstrap assumptions hold on $[1, P]$: assume that there exist an $\varepsilon$ small enough and $\delta\in[0,1/2)$, such that, for all $(\rho, r, \omega)$ in $[1,P]\times \mathbb{R}^3$,
\begin{itemize}
\item Energy bounds for $\phi$:
\begin{equation}\label{eq:massivebpwave}
\mathcal{E}_N[\phi](\rho) \le 2 \varepsilon;
\end{equation}
\item Global bounds for $f$: 
\begin{equation}\label{eq:massivebptransport}
E_N[f](\rho) \le \rho^{\delta} 2 \varepsilon.
\end{equation}
\end{itemize}

It follows from a continuity argument\footnote{Note that the methods of this paper show in particular that the system is well-posed in the spaces corresponding to the norms $\mathcal{E}_N^{1/2}[\phi]$ and $E_N[f]$ for $N$ sufficiently large. See also \cite{cr:gwnv} for another local existence statement. } that $P > 1$ and the remainder of the proof will be devoted to the improvement of each of the above inequalities, establishing the validity of Theorem \ref{th:mtmsv}. The proof is organized as follows:
\begin{itemize}
\item We first prove the necessary commutation formulae with the transport operator $T_\phi$ in Section \ref{sec:mcommutators}. The fundamental commutator is given in Lemma \ref{lem:massivecommutation}.
\item A second step consists in rewriting the well-known standard Klainerman-Sobolev estimates for scalar fields using the hyperboloidal foliation (Proposition \ref{prop:hypwave}), in Section \ref{sec:kshyp}. These decay estimates for derivatives of scalar fields also provide estimates on the fields themselves after integration along null lines (Lemma \ref{lem:hypwave1}).
 \item In Section \ref{sec:btwave}, the bootstrap assumption \eqref{eq:massivebpwave} is improved, assuming weighted $L^2_x$ decay estimates for the higher order derivatives of the solution to the transport equation (see Lemma \ref{lem:energymassivewave}). The proof is based on energy estimates for which we need the source terms to have sufficient decay. When only low derivatives are involved, our Klainerman-Sobolev inequalities for $f$ are sufficient to close the energy estimates for $\phi$, so that the $L^2_x$ decay estimates are only required to handle the high derivatives case (see Lemma \ref{lem:mwavelow}). 
 \item In Section \ref{sec:bttransport}, the bootstrap assumption \eqref{eq:massivebptransport} is improved. The proof relies on the conservation law for the massive transport equation (Lemma \ref{lem:conslawvn}). Unfortunately, some of the source terms arising from the commutation relations are a priori not space-time integrable. To handle this lack of decay, we use \emph{renormalized} variables by incorporating part of the source term in the original variables (Equation \eqref{eq:defgalpha}). Here we use pointwise estimates for $\partial Z^\alpha \phi$ but also the pointwise estimates on $Z^\alpha \phi$ provided by Lemma \ref{lem:hypwave1}. The improvement of the bootstrap assumption is obtained after returning back to the original variables, provided that the initial data are small enough (Proposition \ref{prop:ml1est}).
 \item One finally proves in Section \ref{sec:l2esttr} the $L^2$-estimates for the transport equation, which are required in Section \ref{sec:btwave} to improve the bootstrap assumption on the solution of the wave equation. To this end, the equations for the renormalized variables introduced in Section \ref{sec:bttransport} (Equation \ref{eq:defgalpha}) are rewritten as a system (Lemma \ref{lem:eqgfull}) of inhomogeneous transport equations. Using the fact that we have control on the initial data for $N+n$ derivatives, it is possible to prove strong pointwise estimates for the homogeneous part of the solution to this system carrying the initial data (Lemma \ref{lem:estbbghom}). The inhomogeneous part of the solution to this system (with no initial data) is then approximated by a series  whose coefficients  satisfy recursive transport equations (Equations \eqref{eq-B-1} and \eqref{eq-B-2}). We prove $L^2_x$ decay estimates for each term in this series, by exploiting the specific structure of the inhomogeneous terms. Note that the smallness of the solution is necessary here for the series to eventually converge.
\end{itemize}

Note that 
\begin{itemize}
  \item Estimates of the power of the smallness of the $\rho$-loss in the energy of the distribution function are obtained in Section \ref{sec:l2esttr} (Lemma \ref{lem:estbbginh}).
 \item Finally, the maximal regularity is required in Lemma \ref{lem:highestgalpha}, when pointwise estimates have to be performed on $f$. 
\end{itemize}

In the sequel, we will heavily use the following pointwise estimates, which hold under the bootstrap assumptions \eqref{eq:massivebpwave} and \eqref{eq:massivebptransport}:
\begin{itemize}
 \item as a consequence of Proposition \ref{prop:hypwave}, if $|\gamma|\leq N-\left\lfloor\frac{n}{2}\right\rfloor-1$, then
 $$
 |\partial Z^\gamma \phi| \lesssim \frac{\sqrt{\varepsilon} }{(t-|x|)^{\tfrac12} (1+t)^{\tfrac{n-1}{2}}} = \frac{\sqrt{\varepsilon}  }{\rho (1+t)^{\tfrac{n}{2}-1}};
 $$
 \item as a consequence of Lemma \ref{lem:hypwave1}, if $|\gamma|\leq N-\left\lfloor\frac{n}{2}\right\rfloor-1$, then
  $$
 |Z^\gamma \phi| \lesssim \frac{\sqrt{\varepsilon} (t-|x|)^{\tfrac12} }{ (1+t)^{\tfrac{n-1}{2}}}= \frac{\sqrt{\varepsilon}  \rho}{(1+t)^{\tfrac{n}{2}}};
 $$
 \item as a consequence of Theorem \ref{th:ksmsv}, if $|\be|\leq N-n$, then
 $$
 \int_v \vert\widehat{Z}^\be  f\vert \frac{dv}{v^0} \lesssim \frac{\varepsilon \rho^\delta }{(1+t)^{n}};
 $$
 \item finally, as a consequence of Theorem \ref{th:ksmsv}, if $|\be|\leq \lfloor N/2\rfloor-n$, then
 $$
 \int_v \vert v^0 \widehat{Z}^\be  f\vert dv \lesssim \frac{\varepsilon \rho^\delta }{(1+t)^{n}}.
 $$
\end{itemize}

\subsubsection{Commutators in the massive case}\label{sec:mcommutators}

Let us start with the following commutation relations.
\begin{lemma}
\begin{eqnarray}
[\partial_t, \partial_{v^i} ]&=&0, \\ [0pt]
[ \partial_{x^i}, \partial_{v^i} ]&=&0,\\ [0pt]
[t\p {x^j}+x^j\p t+v^0\pv j,\pv i]&=&-\frac{v^i}{v^0} \partial_{v^j} \\[0pt]%\frac{1}{v^0}\pv j\\
[x^i\p j-x^j\p i+v^i\pv j-v^j\pv i,\pv k]&=&-\delta^i_k \partial_{v^j}+\delta^j_k \partial_{v^i} \\[0pt]
[t\p {x^j}+x^j\p t+v^0\pv j,v^i\pv i]&=&\frac1{v^0}\pv j \\[0pt]%\frac{1}{v^0}\pv j\\
[x^i\p j-x^j\p i+v^i\pv j-v^j\pv i,v^k\pv k]&=&0
\end{eqnarray}
\end{lemma}

We now evaluate the commutators $[T_\phi,\widehat Z]$ for $\widehat Z\in\widehat{\mathbb P}$. We have	
			
\eq{\alg{
\left[ \T_\phi, \widehat Z \right]f & =  \underbrace{[\T_1, \widehat Z]f}_{=0\text{ if } \widehat Z \in \widehat{\mathbb{P}}} - \left[ \T_1(\phi) v^i\pv i, \widehat Z \right]f-\left[\nabla^i\phi\cdot\pv i, \widehat Z\right]f\\%-(n+1)\left[\T\phi,\widehat  Z\right]f\\
&=\widehat{Z}\left[\T_1 (\phi)\right]v^i\pv i f+\T_1(\phi)\left[\widehat Z, v^i\pv i\right]f-\left[\nabla^i\phi\cdot\pv i,\widehat  Z\right]f\\%-(n+1)\left[\T\phi,\widehat  Z\right]f\\
&=\left(\left[\widehat Z, \T_1\right] \phi +\T_1(\widehat Z \phi)\right) v^i\pv i f-\left[\nabla^i\phi\pv i,\widehat Z\right]f\\
&\quad+\begin{cases}\T_1(\phi)\frac{1}{v^0}\pv jf& \mbox{if }\widehat Z=t\p {x^j}+x^j\p t+v^0\pv j \\
0& \mbox{otherwise} 
\end{cases}
\\
&= \T_1(Z \phi) v^i\pv i f-\left[\nabla^i\phi\cdot \pv i,\widehat Z\right]f\\
&\quad+\begin{cases}\T_1(\phi)\frac{1}{v^0}\pv jf& \mbox{if }\widehat Z=t\p {x^j}+x^j\p t+v^0\pv j \\
0& \mbox{otherwise} 
\end{cases}
}}

We have used that $\widehat{Z}(\phi)=Z(\phi)$ since $\phi$ is independent of $v$. 
To estimate the second term on the right-hand side of the last equation, we need
\begin{lemma}
For any $Z\in\mathbb P$, 
\begin{itemize}
\item if $Z$ is a translation, then
$$
[\nabla^i\phi\cdot\pv i,\widehat Z]=-\nabla^i(Z\phi)\pv i.
$$
\item if $Z=\Omega_{jk}$ is a rotation, so that $\widehat{Z}=\Omega_{jk}+v^j \partial_{v^k}-v^k \partial_{v^j}$, then
$$
[\nabla^i\phi\cdot \pv i,\widehat Z]=-\nabla^i(Z\phi)\pv i. %+ \nabla^j \phi \partial_{v^k}  - \nabla^k \phi \partial_{v^j} .
$$
\item if $Z=\Omega_{0j}$ is a Lorentz boost, so that $\widehat{Z}=\Omega_{0j}+v^0 \partial_{v^j}$, then
$$
[\nabla^i\phi\cdot \pv i,\widehat Z]=-\nabla^i(Z\phi)\pv i+ \nabla_i \phi\frac{v^i}{v^0} \partial_{v^j}+ \partial_t( \phi) \partial_{v^j} f .
$$
\end{itemize}
\end{lemma}

 We summarize these computations in this lemma. 
\begin{lemma}Let $\widehat Z\in\widehat{\mathbb P}$, then 
$$
\left[ \T_\phi, \widehat Z \right]f  =   \T_1(Z \phi) v^i\pv i f+\sum_{\substack{ |\alpha| \le 1, \\1 \le j \le n, \\ 0 \le \beta \le n}} p^{j\beta}\left( \frac{v}{v^0} \right)\partial_{x^\beta} Z^\alpha (\phi)\cdot \partial_{v^j} f,
$$
where the $p^{i\beta}\left( \frac{v}{v^0} \right)$ are polynomial of degree at most $1$ in the variables $\frac{v^k}{v^0}$, $1 \le k \le n$.
\end{lemma}

The terms containing $v$ derivatives in the above formulae are problematic, since the $\partial_{v}$ are not part of the algebra $\widehat{\PP}$. We use the following decomposition, for all $1 \le i \le n$, 
\begin{eqnarray}
\partial_{v^i}&=& \frac{1}{v^0} \left( t \partial_{x^i}+x^i \partial_{t}+v^0 \partial_{v^i} \right) - \frac{1}{v^0} \left( t \partial_{x^i}+x^i \partial_{t} \right) \\
&=& \frac{1}{v^0}\widehat{\Omega}_{0i} - \frac{1}{v^0} \left( t \partial_{x^i}+x^i \partial_{t} \right).
\end{eqnarray}
\begin{remark} Note that 
$$\frac{1}{v^0} \left| t \partial_{x^i}f+x^i \partial_{t} f\right| \le \frac{t}{v^0}  \left(| \partial_t f| +|\partial_{x^i} f | \right),$$
for $(t,x)$ in the future of the unit hyperboloid. 
Now $\partial_t$ and $\partial_{x^i}$ belong to $\widehat{\PP}$, but the price to pay is the extra $t$ factor. It is precisely this extra $t$ growth which forbids us to close the estimate in dimension $3$. A similar obstacle was identified for the Vlasov-Poisson system in dimension $3$ and solved by the mean of \emph{modified vector fields} in \cite{js:sdsvpsvfm}. We hope to treat the $3d$ massive Vlasov-Nordstr\"om system in future work.
\end{remark}
This leads to the commutation formula
\begin{lemma}\label{lem:comfmsv0}
Let $\widehat Z\in\widehat{\mathbb P}$, then 
$$
\left[ \T_\phi, \widehat Z \right]f  =   \T_1(Z \phi) \sum_{ |\alpha| = 1} q_\alpha \left(\frac{v}{v^0},t,x \right) \widehat{Z}^\alpha(f)+\sum_{\substack{ |\alpha| \le 1, |\beta|=1 \\0 \le \gamma \le n}} \frac{p^{\gamma}_{\alpha\beta}\left( \frac{v}{v^0},t,x \right) }{v^0} \partial_{x^\gamma} Z^\alpha (\phi)\cdot \widehat{Z}^\beta(f),
$$
where the $q_\alpha \left(\frac{v}{v^0},t,x \right)$ and $p^{\gamma}_{\alpha\beta} \left( \frac{v}{v^0},t,x \right)$ are polynomial of degree at most $1$ in the variables ,
$$\frac{v^k}{v^0}, \quad \frac{v^k}{v^0}t, \quad \frac{v^k}{v^0}x^i,\quad 1 \le i, k \le n.$$
\end{lemma}
Iterating the above formula, we obtain
\begin{lemma}\label{lem:massivecommutation0} Let $\alpha$ be a multi-index and $\widehat Z^\alpha \in\widehat{\mathbb P}^{|\alpha|}$, then   
\begin{eqnarray*}
\left[ \T_\phi, \widehat Z^\alpha \right]f  &=& \sum_{\substack{|\gamma|+|\beta| \le |\alpha|+1,\\ 1\leq \vert \gamma \vert \leq\vert \al \vert,\\ 1\leq\vert\beta\vert \leq |\al|}}\T_1(Z^\gamma \phi) q_{\beta \gamma} \left(\frac{v}{v^0},t,x \right) \widehat{Z}^\beta(f)\\
&&\hbox{}+\sum_{\substack{ |\gamma|+ |\beta| \le |\alpha|+1, \\0 \le \sigma \le n,\\  1\leq|\be |\leq \vert\al\vert\\ 0\leq\vert\gamma\vert\leq\vert\alpha\vert}}\frac{1}{v^0} p^{\sigma}_{\gamma\beta} \left( \frac{v}{v^0},t,x \right)\partial_{x^\sigma} Z^\gamma (\phi)\cdot \widehat{Z}^\beta f,
\end{eqnarray*}
where 
\begin{itemize}
\item  the $q_{\beta \gamma} \left(\frac{v}{v^0},t,x \right)$ are linear combinations of terms
$$ q\left(\frac{v^k}{v^0}\right),  \quad q'\left(\frac{v^k}{v^0}\right)t, \quad q''\left(\frac{v^k}{v^0}\right) x^i,\quad 1 \le i, k \le n,$$
where $q, q', q''$ are polynomials of degree at most $|\al|$.
\item the $p^{\gamma}_{\gamma\beta} \left( \frac{v}{v^0},t,x \right)$ are  linear combinations with constant coefficient of terms
$$ p\left(\frac{v^k}{v^0}\right),  \quad p' \left(\frac{v^k}{v^0}\right)t, \quad p''\left(\frac{v^k}{v^0}\right) x^i,\quad 1 \le i, k \le n,$$
where $p, p', p''$ are polynomials of degree at most $|\al|$.
\end{itemize}
\end{lemma}
\begin{proof}This follows by an induction argument on the length of the multi-index $\al$ and we therefore only provide some details here. Assume the lemma is true for $|\alpha|$. Recall that, for any $\widehat{Z} \in \widehat{\mathbb{P}}_0$, 
$$
[T_\phi, \widehat{Z}\widehat{Z}^\alpha ](f)=[ T_\phi , \widehat{Z} ] \widehat{Z}^\alpha(f) + \widehat{Z} [T_\phi, \widehat{Z}^\alpha ](f)=I_1+I_2, 
$$
with 
\begin{eqnarray*}
I_1&=& [ T_\phi , \widehat{Z} ] \widehat{Z}^\alpha(f), \\
I_2&=& \widehat{Z} [T_\phi, \widehat{Z}^\alpha ](f).
\end{eqnarray*}
Using Lemma \ref{lem:comfmsv0}, we have for $I_1$, 
\begin{equation} \label{eq:i1}
I_1  =  \T_1(Z \phi) \sum_{ |\gamma| = 1} q_\gamma \left(\frac{v}{v^0},t,x \right) \widehat{Z}^\gamma(\widehat Z^\alpha f)+\sum_{\substack{ |\gamma|\leq 1, |\beta| = 1, \\0 \le \sigma \le n}} \frac{1}{v^0}p^{\sigma}_{\gamma\beta} \left( \frac{v}{v^0},t,x \right)\partial_{x^\sigma} Z^\gamma (\phi)\cdot \widehat{Z}^\beta (\widehat{Z}^\alpha f),
\end{equation}
with $q_\gamma$ and $p^{\sigma}_{\gamma\beta}$ as in the statement of Lemma \ref{lem:comfmsv0}. 
Since all the terms in \eqref{eq:i1} clearly have the desired form, we turn to $I_2$. Applying the induction hypothesis, we have 
\begin{eqnarray*}
I_2&=&\widehat{Z} \left[  \sum_{\substack{|\gamma|+|\beta| \le |\alpha|+1,\\ 1\leq \vert \gamma \vert \leq\vert \al \vert,\\1\leq \vert \be \vert \leq\vert \al \vert}}\T_1(Z^\gamma \phi) q_{\beta \gamma} \left(\frac{v}{v^0},t,x \right) \widehat{Z}^\beta(f)+\sum_{\substack{ |\gamma|+ |\beta| \le |\alpha|+1, \\0 \le \sigma \le n,\\  1\leq|\be |\leq \vert \al \vert}}\frac{1}{v^0} p^{\sigma}_{\gamma\beta} \left( \frac{v}{v^0},t,x \right)\partial_{x^\sigma} Z^\gamma (\phi)\cdot \widehat{Z}^\beta f \right], \\
&=& \sum_{\substack{|\gamma|+|\beta| \le |\alpha|+1,\\ 1\leq \vert \gamma \vert \leq\vert \al \vert,\\ 1\leq \vert \be \vert \leq\vert \al \vert}} \widehat{Z} \left[ \T_1(Z^\gamma \phi) q_{\beta \gamma} \left(\frac{v}{v^0},t,x \right) \widehat{Z}^\beta(f) \right] \\
&&\hbox{}+\sum_{\substack{ |\gamma|+ |\beta| \le |\alpha|+1, \\0 \le \sigma \le n,\\  1\leq|\be |\leq \vert \al \vert}} \widehat{Z} \left[\frac{1}{v^0} p^{\sigma}_{\gamma\beta} \left( \frac{v}{v^0},t,x \right)\partial_{x^\sigma} Z^\gamma (\phi)\cdot \widehat{Z}^\beta f \right]
=J_1+J_2, %+J_3,
\end{eqnarray*}
where 
\begin{eqnarray*}
J_1&=& \sum_{\substack{|\gamma|+|\beta| \le |\alpha|+1,\\ 1\leq \vert \gamma \vert \leq\vert \al \vert,\\ 1\leq \vert \be \vert \leq\vert \al \vert}} \widehat{Z} \left[ \T_1(Z^\gamma \phi) q_{\beta \gamma} \left(\frac{v}{v^0},t,x \right) \widehat{Z}^\beta(f) \right]
\end{eqnarray*} 
and
$$J_2= \sum_{\substack{ |\gamma|+ |\beta| \le |\alpha|+1, \\0 \le \sigma \le n,\\  1\leq|\be |\leq \vert \al \vert}} \widehat{Z} \left[\frac{1}{v^0} p^{\sigma}_{\gamma\beta} \left( \frac{v}{v^0},t,x \right)\partial_{{x^\sigma}} Z^\gamma (\phi)\cdot \widehat{Z}^\beta f \right],
$$
with $q_{\beta \gamma}$ and $p^{\sigma}_{\gamma\beta}$ as in the statement of the Lemma. %One easily sees that $J_1$ has the correct form using the commutation properties of $\widehat{Z}$ and $\T_1$. 
To see that $J_1$ has the correct form, we distribute $\widehat{Z}$ which gives rise to three types of terms. The terms arising when $\widehat{Z}$ hits $\T_1 ( Z^\gamma \phi)$ or $\widehat{Z}^\beta(f)$ are easily seen to have the right form. It remains to look at the case when $\widehat{Z}$ hits $q_{\beta \gamma}$. If $\widehat{Z}$ is a translation, $\widehat{Z}=\partial_{x^\beta}$ and one easily sees that $\widehat{Z}(q_{\beta \gamma})$ has the correct form. If $\widehat{Z}$ is the lift of a rotation or a Lorentz boost, then we write schematically $\widehat{Z}= ^{x}Z+^{v}Z$, where $^{x}Z$ is a homogeneous differential operator of order $1$ in $(t,x)$ and $^{v}Z$ is a homogeneous differential operator of order $1$ in $v$. It is then easy to check that $^{x}Z$  applied to a polynomial in the variables $v^i/v^0$ of degree $\le |\alpha|$, possibly multiplied by the variables $t, x^i$ will produce a polynomial, of the same degree $ \le |\al|$, in the variables $v^i/v^0$, possibly multiplied by the variables $t, x^i$. Similarly, $^{v}Z$ applied to a polynomial in the variable $v^i/v^0$ of degree $\le |\al|$ will produce a polynomial in the variables $v^i/v^0$ of degree $\le |\al|+1$. As a consequence, $\widehat{Z}^\al (q_{\beta\gamma })$ is a linear combination of polynomials of degree $|\al|+1$, possibly multiplied by $t, x^i$. The term $J_2$ can be treated similarly. 
\end{proof}

The full expression for $T_{\phi}\left( \widehat {Z}^\al f \right)$ can now be computed using the transport equation \eqref{eq:tfmsv} satisfied by $f$. 

 \begin{lemma}\label{lem:massivecommutation} Let $Z^\al$ be in $\mathbb{P}^{|\al|}$. Then the following equation holds:
\eq{\alg{
\T_\phi\left(  \widehat{Z}^\al(f) \right) &=\sum_{\substack{|\gamma|+|\beta| \le |\alpha|+1,\\ \vert \gamma \vert  \geq 1, |\be| \geq 1 }}\T_1(Z^\gamma \phi) q_{\beta \gamma} \left(\frac{v^i}{v^0},t,x \right) \widehat{Z}^\beta(f)+\sum_{\substack{ |\gamma|+ |\beta| \le |\alpha|+1, \\0 \le \sigma \le n,\\
1\leq  |\beta|\leq \vert \al \vert }}\frac{1}{v^0} p^{\sigma}_{\gamma\beta} \left( \frac{v}{v^0},t,x \right)\partial_{\sigma} Z^\gamma (\phi) \widehat{Z}^\beta f\\
&\quad +\sum_{\ab{\gamma}+\ab{\beta}= \ab{\al}} r_{\gamma\beta}
\T_1(Z^\gamma\phi)\widehat{ Z}^\beta(f),
}}
where the $q_{\beta\gamma}$ and $p^\sigma_{\gamma \beta}$ are as in Lemma \ref{lem:massivecommutation0} and the $r_{\gamma\beta}$ are constants. 
\end{lemma}

\begin{proof}
  We have 
\eq{
\T_\phi \left( \widehat{ Z}^\al f \right)=[\T_\phi,\widehat{Z}^\al] f+\widehat{Z}^\al \T_\phi f=[\T_\phi,\widehat{ Z}^\al] f+\widehat{ Z}^\al\left((n+1)\T_1(\phi) f\right).
}
The lemma thus follows from Lemma \ref{lem:massivecommutation0} and the fact that 
\eq{
\widehat{ Z}^\al(\T_1(\phi) f)= \sum_{\ab{\gamma}+\ab{\beta}= \ab{\al}} r_{\gamma\beta}
\T_1(Z^\gamma\phi)\widehat{ Z}^\beta(f)
}
where the $r_{\gamma\beta}$ are constants.
\end{proof}

\subsubsection{The $H_\rho$ foliation and the wave equation}\label{sec:kshyp}
The aim of this section is to provide a Klainerman-Sobolev type inequality, applicable to solutions of the inhomogeneous wave equation if the inhomogeneities decay sufficiently fast, using only energies on the $H_\rho$ foliation. This question was addressed by Klainerman in \cite{MR1199196} for the Klein-Gordon operator and we show here how a similar proof can also be applied to the wave operator. We thus consider a function $\psi$ and its energy-momentum tensor 
\begin{equation} \label{def:emtw}
T[\psi]=d\psi\otimes d\psi -1/2  \left( \eta (\nabla \psi, \nabla \psi) \right)\eta.
\end{equation}

If we want to perform energy estimates on $H_\rho$, we need to multiply $T[\psi]$ by $\partial_t$ and the normal to $H_\rho$ $\nu_\rho$ and integrate on $H_\rho$. Let us thus compute the quantity $T(\partial_t, \nu_\rho)$.

We find:
\begin{equation} \label{eq:tptnu}
T[\psi](\partial_t, \nu_\rho)=\frac{t}{2\rho}\left( \psi_{t}^2+\psi_r^2+|\slashed \nabla \psi|^2 \right)+\frac{r}{\rho}\psi_t \psi_r.
\end{equation}

Recall also that the volume form on $H_\rho$ is given by $\frac{\rho}{t} r^{n-1} dr d\sigma$. Since we are looking only at the region $\rho \ge 1$, we have that $t> r$ and $T(\partial_t, \nu_\rho)$ is clearly positive definite, with some degeneration as $r \rightarrow t$. More precisely, fix $(t,x)$ in the future of the unit hyperboloid. Assume first that $r=|x| \le t/2$, then $\frac{\rho}{t}T[\psi](\partial_t, \nu_\rho) \ge |\partial \psi|^2 $. 

Let $(Y^0, Y^i)$ be the coordinates of $(t,x)$ in the $(y^\alpha)$ system of coordinates adapted to the $H_\rho$ foliation as introduced in Section \ref{se:tf}. Let $\Phi(y)=\partial \psi(Y^0,Y^j+ty^j)$. Then, a classical Sobolev inequality yields
\begin{eqnarray}
|\partial \psi(Y^0,Y^j)|^2=| \Phi(0)|^2 &\lesssim&  \sum_{ k \le (n+2)/2} \int_{|y| \le \delta} |\partial_{y}^k \Phi(y) |^2 dy, \\
&\lesssim&  \sum_{ k \le (n+2)/2} \int_{|y| \le \delta} |Z^k(\partial \psi)(Y^0,Y^j+ty^j) |^2 dy,
\end{eqnarray}
using that $\partial_{y^i}=\frac{1}{t} \Omega_{0i}$ and the fact that $\partial_{y^i} \Phi= t \partial_{y^i} \psi$ together with estimates on $\frac{t}{t(Y^0, Y^j+t y^j)}$ similar to those of Section \ref{se:ksmsv}.  Applying the change of coordinates $y^j \rightarrow t y^j$ yields
$$
|\partial \psi(Y^0,Y^j)|^2 \lesssim \frac{1}{t^n} \sum_{ k \le (n+2)/2} \int_{|y| \le t \delta} |Z^k(\partial \psi)(Y^0,Y^j+y^j) |^2 dy.
$$
Finally, note that $|Y^j+y^j|=|x^j + y^j| \lesssim (1/2+\delta)t$ so that if $\delta < 0$, we are still away from the light cone. Thus, the right-hand side of the previous equations can be controlled by the energies of $Z^k( \partial \psi)$ on $H_\rho$. On the other hand, if $t/2 \le r < t$, we first remark that 

$$\frac{\rho}{t}T[\psi](\partial_t, \nu_\rho) \ge ( 1-\frac{r}{t})|\partial \psi|^2.$$

Thus, we may repeat the previous arguments, losing the factor $( 1-\frac{r}{t})$ in the process, as follows
\begin{eqnarray}
|\partial \psi(Y^0,Y^j)|^2 &\lesssim &\frac{1}{t^n} \sum_{ k \le (n+2)/2} \int_{|y| \le t \delta} |Z^k(\partial \psi)(Y^0,Y^j+y^j) |^2 dy, \\
&\lesssim &\frac{1}{t^n} \sum_{ k \le (n+2)/2} \int_{|y| \le t \delta} \left( 1-\frac{r}{t}\right)^{-1} \left( 1-\frac{r}{t}\right)|Z^k(\partial \psi)(Y^0,Y^j+y^j) |^2 dy.
\end{eqnarray}
Since 

$$
\frac{1}{1-\frac{r}{t}}= \frac{t(t+r)}{\rho^2} \lesssim \frac{t^2}{\rho^2}
$$
and since again, we can replace $t(Y^0, Y^j+t y^j)$ by $t$ in all the above computations, we have shown that 
$$
|\partial \psi(Y^0,Y^j)|^2 \lesssim \frac{1}{t^{n-2} \rho^2}.
$$
Since $\rho^2=(t+r)(t-r)$ and since $t \ge r$ in the future of the unit hyperboloid, this is exactly the decay estimate predicted by the usual Klainerman-Sobolev inequality using the $\Sigma_t$ foliation. We summarize this in the following proposition.
\begin{proposition}[Klainerman-Sobolev inequality for the wave equation using the hyperboloidal foliation]\label{prop:hypwave}$\quad$
 For any sufficiently regular function $\psi$ of $(t,x)$ defined on the future of the unit hyperboloid, let $\mathcal{E}_{(n+2)/2}[\psi](\rho)$ denote the energy
$$
\mathcal{E}_{(n+2)/2}[\psi](\rho)=\sum_{ | \alpha | \le \frac{n+2}{2}} \int_{H_\rho} T[Z^\alpha[\psi]] ( \partial_t, \nu_\rho) d\mu_\rho.
$$
Then, for all $(t,x)$ in the future of the unit hyperboloid, 
\begin{equation} \label{eq:kswh}
|\partial \psi |(t,x) \lesssim \frac{1}{t^{(n-1)/2}(t-|x|)^{1/2}} \mathcal{E}_N^{1/2}[\psi](\rho(t,x)).
\end{equation}
\end{proposition}
It is interesting to note that the above proof does not make use of the scaling vector field. 

\begin{remark} We will use in the following the inequality:
\begin{eqnarray*}
\frac{\rho}{2t}T[\psi](\partial_t, \nu_\rho)& =& | \slashed \nabla \psi|^2 +  \left(\psi_{t}^2+\psi_r^2 + \frac{2r}{t}\psi_t \psi_r\right)\\ 
& = & |\slashed \nabla \psi|^2 + \left(1-\frac{r}{t}\right) \left(\psi_{t}^2+\psi_r^2\right) + \frac{r}{t}\left(\psi_{t}+\psi_r\right) ^2
\end{eqnarray*}
that is to say
$$
|\partial \psi|^2\lesssim  \frac{\rho}{t-r}T[\psi](\partial_t, \nu_\rho) = \frac{t+r}{\rho} T[\psi](\partial_t, \nu_\rho).
$$
\end{remark}

The inequality \eqref{eq:kswh} provides decay for $\partial \psi$ but not for $\psi$. By integration along null lines, one can obtain the following decay for $\psi$.

\begin{lemma}\label{lem:hypwave1} 
Let $\psi$ be such that $\mathcal{E}_{(n+2)/2} [\psi](\rho)$ is uniformly bounded on $[1, P]$, for some $P > 1$. Assume moreover that $\psi_{| \rho=1}$ vanishes at $\infty$.

Then $\psi$ satisfies: for all $(t,x)$ in the future of $H_1$,
$$
\vert \psi (t,x) \vert \lesssim \sup_{[1,P]}\left[\mathcal{E}_N^{1/2} \right] \dfrac{u^{\frac12}}{t^{\frac{n-1}{2}}},
$$
where $u=t-|x|$. 
\end{lemma}
\begin{proof} The Klainerman-Sobolev estimates provide
$$
|\partial \psi (t,x)_| \lesssim \dfrac{\sup_{[1,P]}\left[\mathcal{E}_N^{1/2} \right]}{\rho t^{\frac{n}{2}-1}}.
$$
Let $(t,x) = (t,r,\omega \in\mathbb{S}^n)$ be a point in the future of $H_1$, and consider the point on the hyperboloid lying at the intersection of the past light cone from $(t,x)$, the hyperboloid $H_1$, and the direction $\omega$. Writing $u=t-|x|$ and $v=t+|x|$, we have
$$
(t_1, x_1)=\left(t_1 = \sqrt{1+r_1^2}, r_1 = \frac{1}{2}\left(v-\frac{1}{v}\right), \omega\right)\in H_1.
$$
Note that 
$$
\frac{1}{\langle r_1\rangle} = \frac{2}{v+\frac{1}{v}}\lesssim \frac{1}{\langle v\rangle} \text{ since }v\geq t\geq 1.
$$
Integrating along the direction $\omega$ along the past light cone from $(t,x)$ from $(t_1, r_1, \omega)$ to $(t,r,\omega)$, one obtains
$$
\psi(t,x) = \psi(t_1, x_1) + \int_{t_1-r_1}^{t-r} (\partial_u \psi) du,
$$
so that 
\begin{eqnarray*}
|\psi(t,x)| & \lesssim &\dfrac{\mathcal{E}_N^{1/2}(\rho=1)}{\langle r_1\rangle^{\frac{n-1}{2}}} +\int_{t_1-r_1}^{t-r} \frac{\sup_{[1,P]}\left[ \mathcal{E}_N^{1/2} \right] }{u^{\frac12} v^{\frac{n-1}{2}}} du \\
& \lesssim &\dfrac{\mathcal{E}_N^{1/2}(\rho=1)}{\langle v\rangle^{\frac{n-1}{2}}} + \frac{\sup_{[1,P]}\left[ \mathcal{E}_N^{1/2} \right] u^{\frac12}}{v^{\frac{n-1}{2}} }\\
& \lesssim & \sup_{[1,P]}\left[ \mathcal{E}_N^{1/2} \right]\frac{u^{\frac12}}{v^{\frac{n-1}{2}} },
\end{eqnarray*}
which concludes the proof since $\frac{1}{v} \le \frac{2}{t}$. Here we have used that $\left| \psi(t_1, x_1) \right| \lesssim \frac{\mathcal{E}_N^{1/2}(\rho=1)}{\langle r_1\rangle^{\frac{n-1}{2}}}$ which follows from usual weighted Sobolev inequalities on $\mathbb{R}^n$ applied to $\partial \psi$ and the assumption that $\psi$ restricted to $\rho=1$ vanishes at infinity. 
\end{proof}

\subsubsection{Commutation of the wave equation}

The commutation of the wave equation with our set of vector fields is straightforward and leads to 

\begin{lemma}\label{lem:massivecommutwave} %\jj{I have modified this lemma, especially for the $v^2$ thing, I'd like inputs in this. }
For any multi-index $\alpha$, 
\begin{equation} \label{eq:comwave}
\square Z^\alpha(\phi)= \int_{v \in \mathbb{R}^n_v} \widehat{Z}^\alpha (f) \frac{dv}{v^0}.
\end{equation}

\end{lemma}
\begin{proof}First, recall that the vector fields in the algebra $\PP$ commutes exactly with $\square$. The lemma then follows from Lemma \ref{lem:cvfavmsv} and Remark \ref{rem:cvav} in case $Z^\alpha$ contains some combinations of Lorentz boosts. 
%The lemma then follows from Lemma \ref{lem:cvfavmsv} and the remark .
\end{proof}
\begin{remark}\label{re:cchi2} The following inequality will be used later on:
$$ 
\left| \int_{v \in \mathbb{R}^n_v}\widehat Z^\alpha (f) \frac{dv}{v^0} \right|\leq \chi_1\left(\vert\widehat Z^\alpha (f)  \vert \right),
$$
which is a direct consequence of Remark \ref{re:cchi}.
\end{remark}

\subsubsection{Energy estimates for the wave equation on hyperboloids}\label{sec:btwave}
Consider $\psi$ defined in hyperboloidal time for all $\rho \in [1,P]$ and assume that $\psi$ solves $\square \psi = h$.
Recalling the expression for $T[\phi](\partial_t, \nu_\rho)$ given by \eqref{eq:tptnu}, we have
\begin{lemma}\label{lem:inhwavehyp}
Let $\rho \in [1, P]$.
Then,
$$
\int_{H_\rho} T[\phi](\partial_t, \nu_\rho) d\mu_{H_\rho} =   \int_{H_1} T[\phi](\partial_t, \nu_\rho) d\mu_{H_1} + \int_1^\rho \int_{H_{\rho'}} \left(\partial_t \phi \right)(\rho') h(\rho')  d\mu_{H_{\rho'}} d\rho'.
$$
\end{lemma}
\begin{proof} The proof of this fact is only sketched, since classical. The reader can refer to \cite{MR1199196}. Remember that the divergence of the stress-energy tensor $T[\phi]$ is given by
$$
\partial^\al T_{\al \be}[\phi] = h\partial_\beta \phi,
$$
when $\phi$ satisfies the equation $\square  \phi = h$. The lemma then follows by integration between the two hyperboloids $H_1$ and $H_\rho$ and an application of Stoke's theorem.
\end{proof}

To close the energy estimates for $Z^\alpha(\phi)$, we need the right-hand side of \eqref{eq:comwave} to decay. Since for $|\alpha| \le N-n$, the required decay follows from our Klainerman-Sobolev inequality \eqref{ineq:ksmsv} as well as the bootstrap assumption \eqref{eq:massivebptransport}, we have the following lemma

\begin{lemma}\label{lem:energymassivewave}  \label{lem:mwavelow} 
Assume that $\delta <1$.  Assume moreover that for all multi-indices $\alpha$ such that  $N-n +1\le |\al| \le N$, the following $L^2_x$ decay estimate holds 
\begin{equation}\label{ass:l2x} 
 \int_{H_\rho} \frac{t}{\rho}\left(\int_{v}|\widehat{Z}^{\al} f| \frac{dv}{v^0}\right)^{2} d\mu_{H_\rho} \lesssim \varepsilon^{2} \rho^{\delta +1-n}.
\end{equation} 
Then, the following inequality holds, for all $\rho \in [1,P]$:
$$
\mathcal{E}_N[\phi](\rho) \leq \varepsilon (1+C\varepsilon^{1/2}),
$$
where $C$ is a constant depending solely on the dimension $n$ and the regularity $N$. In particular, for $\varepsilon$ small enough, for all $\rho \in [1,P]$:
$$
\mathcal{E}_N[\phi](\rho) \leq \frac{3}{2} \varepsilon.
$$
\end{lemma}
\begin{remark}\begin{itemize}
               \item The weighted $L^2$-estimates \eqref{ass:l2x} will in fact be proven in Section \ref{sec:l2esttr} for the wider range of multi-indices $\al$ with $ \lfloor N/2 \rfloor -n+1\le |\al|\le N$. 
               \item Note that the $L^2$-estimates are needed only for $|\al|>N-n$: for lower order, the pointwise decay estimates for the velocity averages are sufficient to conclude.
              \end{itemize}
\end{remark}

\begin{proof} The proof of this lemma relies on Lemmata \ref{lem:inhwavehyp} . Applying first Lemmata \ref{lem:inhwavehyp} and \ref{lem:massivecommutwave} to $Z^\alpha(\phi)$ for all multi-indices of length $|\alpha| \le N$, one obtains immediately, for $\rho \le P$
\begin{eqnarray*}
\mathcal{E}_N[\phi](\rho) - \mathcal{E}_N[\phi](1) &\lesssim&   \sum_{ \vert \al \vert\leq N }\sum_{Z^\al \in \mathbb{P}^{\vert \al\vert }}\int_1^\rho\int_{H_\rho'}  \vert \partial_t  Z^\al\phi \vert \left(\int_v \frac{| \widehat Z^\al f|}{v^0}dv \right)d\mu_{H_{\rho'}}d \rho'\nonumber\\
 &\lesssim &  \sum_{ \vert \al \vert\leq N }\sum_{Z^\al \in \mathbb{P}^{\vert \al\vert }} \int_1^\rho \left(\int_{H_\rho'}  \left \vert \left(\frac{\rho'}{t}\right)^\frac12 \partial_t  Z^\al\phi \right\vert \cdot  \left(\frac{t}{\rho'}\right)^\frac12 \int_v \frac{| \widehat Z^\al f|}{v^0}dv\right)d\mu_{H_{\rho'}}d \rho'  \nonumber\\
&\lesssim &  \int_1^\rho    \mathcal{E}_N^{1/2}[\phi](\rho')  \left(\sum_{ \vert \al \vert\leq N }\sum_{Z^\al \in \mathbb{P}^{\vert \al\vert }}\left(\int_{H_\rho'}    \left(\frac{t}{\rho'}\right)\left( \int_v \frac{| \widehat Z^\al f|}{v^0}dv\right)^2 d\mu_{H_{\rho'}}\right)^\frac12 \right)d \rho'
\end{eqnarray*}
We now apply for the low derivatives of $f$ Theorem \ref{th:ksmsv} (in conjunction with Remark \ref{re:cchi}): for $|\alpha| \leq N-n $
$$
 \int_{H_{\rho'}}  \left(\frac{t}{\rho'}\right)\left( \int_v \frac{| \widehat Z^\al f|}{v^0}dv\right)^2 d\mu_{H_{\rho'}} \lesssim \int_0^\infty \frac{t}{\rho'} \cdot \frac{\varepsilon^2 \rho'{}^{2\delta}}{t^{2n}} r^{n-1} \frac{\rho'}{t} dr \lesssim \varepsilon^2 \rho'{}^{2\delta -n} \int_{0}^{\infty}\frac{y^{n-1}}{\langle y\rangle^{2n}} dy,
$$
Since the last integral is convergent, we obtain 
$$
 \left[ \int_{H_{\rho'}}  \left(\frac{t}{\rho'}\right)\left( \int_v \frac{| \widehat Z^\al f|}{v^0}dv\right)^2 d\mu_{H_{\rho'}} \right]^{1/2} \lesssim  \varepsilon \rho'^{\delta-n/2},
$$
which, assuming $\delta< 1$ is integrable in $\rho$.
For the higher derivatives of $f$ (\emph{i.e.} for $|\al|> N -n$), one uses the assumption of the lemma to obtain:
$$
\sum_{N-n < |\al|\leq N} \sum_{Z^\al \in \mathbb{P}^{|\al|}} \left(\int_{H_\rho} \frac{t}{\rho}\left(\int_{v}|\widehat{Z}^{\al} f| \frac{dv}{v^0}\right)^{2} d\mu_{H_\rho}\right)^{\frac12} \lesssim \varepsilon \rho^{\frac{\delta}{2} +\frac12-\frac{n}{2 }}
$$
We obtain finally
$$
\mathcal{E}_N[\phi](\rho) \leq \varepsilon + C\varepsilon \left(\int_1^\rho \left( \rho'{}^{\frac{\delta+1-n}{2}}  \right)\mathcal{E}_N^{1/2}[\phi](\rho')   d\rho'\right),
$$
 where $C$ is a constant depending only on the regularity and the dimension. Remark that
 $$
 \frac{\delta+1-n}{2} \leq \frac{\delta}{2} -\frac32<-1,
 $$
 for $n>3$ and $\delta <1$. The result then follows using the bootstrap assumptions \eqref{eq:massivebpwave} and integrating in $\rho$.
\end{proof}

\subsubsection{$L^1$-Estimates for the transport equation}\label{sec:bttransport}

In the remainder of the article, we will use the notation
$$
\mathbb{E}[g](\rho) = \int_{H_\rho}\chi_1(g) d\mu_{\rho},
$$
for any regular distribution function $g$.

\begin{lemma}\label{lem:conslawvn} Let $h$ be a  regular distribution function for the massive case in the sense of Section \ref{se:rdf}. Let $g$ be a regular solution to $\mathbf{T}_\phi g = v^0h $, with $v^0 = (1+|v|^2)^{\tfrac12}$, defined on $\bigcup_{\rho\in[1,P]}H_\rho \times \mathbb{R}^n_v$, for some $P>1$. Then, for all $\rho \in [1, P]$,
\begin{equation}\label{eq:acltpmslm}
\int_{H_\rho} \chi_1 (\vert g\vert ) d \mu_{H_\rho} \lesssim \int_{H_1} \chi_1 (\vert g\vert ) d \mu_{H_1} + \int_1^\rho\int_{H_{\rho'}}\int_v     \left(v^0|h| + \frac{1}{v^0}|\partial_{x^0}\phi g| + |\T_1(\phi) g| \right) dvd \mu_{H_\rho'} d\rho' 
\end{equation}
\end{lemma}

\begin{proof} One proves first that
$$
\int_{H_\rho} \chi_1 ( g ) d \mu_{H_\rho} =\int_{H_1} \chi_1 ( g ) d \mu_{H_1} + \int_1^\rho\int_{H_{\rho'}}\int_v     \left(v^0h + \left(\frac{1}{v^0} \partial_{x^0}\phi-(n+1)\T_1(\phi)\right) g  \right) dvd \mu_{H_\rho'} d\rho',
$$
since by integration by parts
$$
\int_v(\T_1(\phi)v^i + \nabla^i\phi ) \partial_{v^i} g dv = -\int_v g\left((n+1)\T_1(\phi) -\frac{\partial_{x^0}\phi}{v^0} \right) dv.
$$
This establishes the lemma in case $g \ge 0$. As in the proof of Lemma \ref{lem:macl}, the conclusion in the general case follows after regularization of the absolute value.
\end{proof}

For any multi-index $\alpha$, let us now introduce the auxiliary function $g^\al$:
\begin{equation}\label{eq:defgalpha}
g^\al =  \widehat Z^\al (f)-\sum_{\substack{\ab{\gamma}+\ab{\beta}\leq\ab{\al}+1\\ 1 \le \vert\be\vert \\  1\le \vert \gamma \vert \le N-\frac{n+2}{2} }}   q_{\be\gamma}Z^\gamma(\phi)\widehat  Z^\be (f),
\end{equation}
where the $q_{\beta \gamma}$ are as in the statement of Lemma \ref{lem:massivecommutation}).
One can view $g^\al$ as a renormalization of $\widehat Z^\al (f)$. The extra terms in the definition of $g^\alpha$ will allow us to absorb certain source terms in the equation satisfied by $T_{\phi}\left[ \widehat{Z}^\al(f) \right]$ (cf. Lemma \ref{lem:massivecommutation}) which cannot be estimated adequately because they carry too much $v^0$-weight, leading either to a $t$-loss (cf. Remark \ref{re:cchi2}) or to a $v^0$ loss.

To perform $L^1$-estimates on $\widehat{Z}^\alpha(f)$, we therefore proceed as follows:
\begin{itemize}
 \item We derive the equations satisfied by the $g^\al$ and then use them to obtain $L^1$-estimates for the $g^\al$.
 \item We then prove the same $L^1$-estimates for $(v^0)^2 g^\al$ with $|\alpha| \le \lfloor\frac{N}{2}\rfloor$ to take into account that the lower derivatives of $f$ are weighted by $(v^0)^2$ in the $E_N[f]$ norm (see the definition of the norm in Section \ref{se:tnvn}).
 \item Finally, the $L^1$ estimates on $g^\alpha$ are then transformed into $L^1$-estimates on $\widehat{Z}^{\alpha}(f)$ using pointwise estimates on $Z^\gamma(\phi)$ for $\gamma$ sufficiently small.
\end{itemize}
We start by deriving the equations for the $g^\alpha$. 

\begin{lemma}\label{lem:galpha}For any multi-index $\alpha$, $g^\al$ satisfies the equation
\begin{eqnarray*}
\T_\phi g^\al &= &  \sum_{\substack{|\gamma|+|\beta| \le |\alpha|+1,\\ \vert \gamma \vert  \geq 1, \,|\be| \geq 1\\ \vert \gamma \vert>N-\frac{n+2}{2}  }}  q_{\beta \gamma}   \T_1(Z^\gamma \phi) \widehat{Z}^\beta(f)+\sum_{\substack{ |\gamma|+ |\beta| \le |\alpha|+1, \\0 \le \sigma \le n, 1\leq |\beta|\leq |\al| }}\frac{1}{v^0} p^{\sigma}_{\gamma\beta} \partial_{x^\sigma} Z^\gamma (\phi) \widehat{Z}^\beta f\\
&&+ \sum_{\ab{\gamma}+\ab{\beta}= \ab{\al}} r_{\gamma\beta}
\T_1(Z^\gamma\phi)\widehat{ Z}^\beta f\\
&& - \sum_{\substack{\ab{\gamma}+\ab{\beta}\leq\ab{\al}+1\\ 1 \le \vert\be\vert \\  1\le \vert \gamma \vert \le N-\frac{n+2}{2}  } }  \T_\phi (q_{\be\gamma}) Z^\gamma(\phi)\widehat  Z^\be f - \sum_{\substack{\ab{\gamma}+\ab{\beta}\leq\ab{\al}+1\\ 1 \le \vert\be\vert \\  1\le \vert \gamma \vert \le N-\frac{n+2}{2}  }} q_{\be\gamma} Z^\gamma(\phi)  \\
&& \times\Bigg(  \sum_{\substack{|\kappa|+|\sigma| \le |\be|+1,\\   1\leq\vert \kappa \vert ,1\leq  |\sigma| \leq |\be| }}\T_1(Z^\kappa \phi) q_{\sigma \kappa}  \widehat{Z}^\sigma(f)+\sum_{\substack{ |\kappa|+ |\sigma| \le |\be|+1, \\0 \le \omega \le n, 1\leq |\sigma|\leq |\be| }}\frac{1}{v^0} p^{\omega}_{\kappa\sigma} \partial_{x^\omega} Z^\kappa (\phi) \widehat{Z}^\sigma f \\
&&  + \sum_{\ab{\kappa}+\ab{\sigma}= \ab{\be}} r_{\kappa\sigma}
 \T_1(Z^\kappa\phi)\widehat{ Z}^\sigma f \Bigg) .
\end{eqnarray*}
\end{lemma}
\begin{proof} This formula is a direct consequence of the product rule and of a double application of Lemma \ref{lem:massivecommutation}.
\end{proof}

Based on Lemma \ref{lem:galpha}, we now proceed to the estimates on $g$:
\begin{lemma} \label{lem:highestgalpha} Assume that $\delta= \varepsilon^{1/4}$, and let $\alpha$ be a multi-index such that $|\alpha|\geq \lfloor N/2\rfloor +1$. Then, $g^\al$ satisfies, for all $\rho \in [1,P]$:
 $$
 \mathbb{E}[|g^\al|](\rho)  \leq  (1+ C \varepsilon^{1/4}) e^{C\varepsilon^{1/2}} \varepsilon\rho^{ \varepsilon^{1/4}}
 $$
where $C$ is a constant depending only on the regularity and the dimension. If, furthermore $n>4$, then $\delta$ can be vanishing.
\end{lemma}
\begin{proof} Applying Lemmata \ref{lem:conslawvn} and Lemma \ref{lem:galpha}, we obtain $L^1$ estimates for $g^\alpha$ provided we can control the source terms. 
%one can perform estimates on the function $g$. Considering the source term of the equation satisfied by $g^\al$, and the form of the energy estimates satisfied by $g$, 
The worse terms that have to be estimated are integrals in $\rho$ of quantities of the form
\begin{eqnarray}
\int_{H_\rho}\int_v v^0 q_{\beta \gamma}  |\partial Z^\gamma \phi| |\widehat{Z}^\beta f | dv d\mu_{H_\rho}, &\text{for }& |\gamma| >  N - \frac{n+2}{2}  \text { and }|\beta|\leq N +1 - |\gamma|  \label{eq:msvtp1} \\
\int_{H_\rho}\int_v |T_\phi(q_{\beta\gamma})|  | Z^\gamma \phi| |\widehat{Z}^\beta f | dv d\mu_{H_\rho}, &\text{for }& |\gamma| \leq N -\frac{n+2}{2}  \text { and } |\beta|\leq N +1 - |\gamma| \label{eq:msvtp2}\\
 \int_{H_\rho}\int_v v^0 |q_{\beta\gamma}q_{\kappa \sigma}|  | Z^\gamma \phi| |\partial Z^\kappa \phi| |\widehat{Z}^\sigma f | dv d\mu_{H_\rho}, &\text{for }& |\gamma| \leq N -\frac{n+2}{2}  \text{ and }|\kappa| \leq N -\frac{n+2}{2}  \label{eq:msvtp3}\\
 && \text { and } |\sigma |\leq N+1-|\kappa|\nonumber\\
  \int_{H_\rho}\int_v v^0  |q_{\beta\gamma}q_{\kappa \sigma}|   | Z^\gamma \phi| |\partial Z^\kappa \phi| |\widehat{Z}^\sigma f | dv d\mu_{H_\rho}, &\text{for }& |\gamma| \leq N -\frac{n+2}{2} \text{ and }|\kappa| > N -\frac{n+2}{2}  \label{eq:msvtp4}\\
   && \text { and } |\sigma |\leq N+1-|\kappa|\nonumber.
\end{eqnarray}
The other error terms are easier to handle, so that, as an illustration, we will only give details below for the extra error terms 
\begin{eqnarray}
   \int_{H_\rho}\int_v \frac{1}{v^0}   |\partial_t \phi| |g^\al| dv d\mu_{H_\rho}, &&\label{eq:msvtp5}
\end{eqnarray}
and
\begin{eqnarray}
         \int_{H_\rho}\int_v   |\T_1 \phi| |g^\al| dv d\mu_{H_\rho}. &&\label{eq:msvtp6}
   \end{eqnarray}

%We have listed here only the worst terms arising from the equation for $g^\al$. 
We deal first with Equation \eqref{eq:msvtp1}. To this end, we remind that 
$$
|q_{\beta \gamma }|\lesssim t.
$$
Furthermore, since 
$$
|\beta| \leq \frac{n+2}{2}  \leq \left\lfloor \frac{N}{2}\right\rfloor-n, \text{ since } N\geq 3n +4,
$$
the Klainerman-Sobolev estimates of Theorem \ref{th:ksmsv} can be applied, since the bootstrap assumption \eqref{eq:massivebptransport} is satisfied:
\begin{eqnarray*}
 \int_{H_\rho}\int_v v^0 q_{\beta \gamma}  |\partial Z^\gamma \phi| |\widehat{Z}^\beta f | dv d\mu_{H_\rho} & \lesssim& \int_{H_\rho}\int_v v^0  t |\partial Z^\gamma \phi| |\widehat{Z}^\beta f | dv d\mu_{H_\rho} \\
  & \lesssim& \int_{H_\rho} \frac{\varepsilon \rho^\delta  }{t^{n-1}}\left(\frac{t}{\rho}\right)^\frac12 \left(\frac{\rho}{t}\right)^\frac12 |\partial Z^\gamma \phi|  d\mu_{H_\rho} \\
  & \lesssim& \varepsilon^{3/2} \rho^{\delta - \frac12}\left(\int_0^\infty t^{3-2n} r^{n-1} dr\right)^\frac12\\
    & \lesssim& \varepsilon^{3/2} \rho^{\delta +1 -\frac{n}{2}}.
\end{eqnarray*}

We now deal with \eqref{eq:msvtp2}. To this end, we notice that since $|\partial \phi|$ decays faster than $1/t$, we have 
$$
|T_{\phi} (q_{\beta\gamma})|\lesssim v^0.
$$
Using Lemma \ref{lem:hypwave1} and the bootstrap assumption \eqref{eq:massivebpwave}, we obtain
\begin{eqnarray*}
 \int_{H_\rho}\int_v |T_\phi(q_{\beta\gamma})|  | Z^\gamma \phi| |\widehat{Z}^\beta f | dv d\mu_{H_\rho} &\lesssim&  \int_{H_\rho}\int_v  | Z^\gamma \phi| |v^0 \widehat{Z}^\beta f | dv d\mu_{H_\rho}\\
 &\lesssim &  \int_{H_\rho}\frac{\varepsilon^{1/2} \rho}{t^{\frac{n}{2}} } \cdot \frac{t}{\rho} \chi_1(|v^0 \widehat{Z}^\beta f | )  d\mu_{H_\rho}\\
     & \lesssim& \varepsilon^{3/2} \rho^{\delta +1 -\frac{n}{2}}.
\end{eqnarray*}
We have in the course of the estimate used Remark \ref{re:cchi2}. 

Consider now the term \eqref{eq:msvtp3} and recall that 
$$
 |q_{\beta\gamma}q_{\kappa \sigma}| \lesssim t^2.
$$
Assume first that $|\kappa| \le N-\frac{n+2}{2}$. Using Lemma \ref{lem:hypwave1} and the bootstrap assumption \eqref{eq:massivebpwave}, we obtain
\begin{eqnarray*}
  \int_{H_\rho}\int_v v^0 |q_{\beta\gamma}q_{\kappa \sigma}|  | Z^\gamma \phi| |\partial Z^\kappa \phi| |\widehat{Z}^\sigma f | dv d\mu_{H_\rho} &\lesssim & 
   \int_{H_\rho} t^ 2 \cdot \frac{\varepsilon^{1/2} \rho}{t^{\frac{n}{2}}} \cdot \frac{\varepsilon^{1/2} }{\rho t^{\frac{n}{2} -1} } \cdot \frac{t}{\rho} \chi_1(|\widehat{Z}^\sigma f |)d\mu_{H_\rho}\\
   &\lesssim & \frac{\varepsilon}{\rho} \int_{H_{\rho}} t^{4-n} \chi_1(|\widehat{Z}^\sigma f |)d\mu_{H_\rho}\\
    &\lesssim & \varepsilon ^2 \rho^{\delta + 3-n}.
\end{eqnarray*}
We have in the course of the estimate used Remark \ref{re:cchi2}. Now, if $|\kappa| > N-\frac{n+2}{2}$, then $\left| \sigma \right| \le \frac{n+2}{2} \le \left\lfloor\frac{N}{2}\right\rfloor -n$ since $N \ge 3n+4$.

Thus, 
\begin{eqnarray*}
  \int_{H_\rho}\int_v v^0 |q_{\beta\gamma}q_{\kappa \sigma}|  | Z^\gamma \phi| |\partial Z^\kappa \phi| |\widehat{Z}^\sigma f | dv d\mu_{H_\rho} &\lesssim & 
   \int_{H_\rho} t^ 2 \cdot \frac{\varepsilon^{1/2} \rho}{t^{\frac{n}{2}}} \frac{\varepsilon \rho^\delta}{t^n}\left(\frac{t}{\rho}\right)^{1/2} \left(\frac{\rho}{t}\right)^{1/2}\left| \partial Z^\kappa(\phi) \right| d\mu_\rho, \\
&\lesssim& \varepsilon^{3/2} \rho^{\delta+\frac{1}{2}} \left(\int_{H_\rho} \frac{\rho}{t} \left| \partial Z^\kappa(\phi) \right|^2 d\mu_\rho \right)^{1/2} \left( \int_{H_\rho} \frac{t^5}{t^{3n}} d\mu_\rho \right)^{1/2}, \\
&\lesssim & \varepsilon^{2} \rho^{\delta+3-n}.
\end{eqnarray*}

The term \eqref{eq:msvtp4} can be estimated similarly since
$$
|\sigma| \leq\frac{n+2}{2} \leq \left\lfloor\frac{N}{2}\right\rfloor -n, \text{ since } N\geq 3n+4.
$$
%the Klainerman-Sobolev estimates of Theorem \ref{th:ksmsv} can be applied, since the bootstrap assumption \eqref{eq:massivebptransport} is satisfied:

%\begin{eqnarray*}
%   \int_{H_\rho}\int_v v^0  |q_{\beta\gamma}q_{\kappa \sigma}|   | Z^\gamma \phi| |\partial Z^\sigma \phi| |\widehat{Z}^\kappa f | dv d\mu_{H_\rho} &\lesssim &
% \int_{H_\rho }t^2 \cdot \frac{\varepsilon^{1/2}  \rho}{t^{\frac{n}{2}} }\cdot\left(\frac{t}{\rho}\right)^\frac12 \left(\frac{\rho}{t}\right)^\frac12  |\partial Z^\sigma \phi| \cdot \frac{\varepsilon \rho^{\delta} }{t^n}d\mu_{H_\rho} \\
% &\lesssim & \varepsilon^2 \rho^{\delta+\frac12} \left(\int_0^\infty t^{5-3n}r^{n-1}dr \right)^{\frac12}\\
%  &\lesssim & \varepsilon^2\rho^{\delta+3 -n}
%\end{eqnarray*}

Finally, for the error terms \eqref{eq:msvtp5} and \eqref{eq:msvtp6}, we apply Proposition \ref{prop:hypwave} and Remark \ref{re:cchi}:
\begin{eqnarray*}
    \int_{H_\rho}\int_v \frac{1}{v^0}   |\partial_t \phi| |g^\al| dv d\mu_{H_\rho} & \lesssim & \frac{\varepsilon^{1/2} }{\rho{}^{\frac{n}{2}}}\mathbb{E}[|g^\al| ]](\rho)\\
        \int_{H_\rho}\int_v \frac{1}{v^0}   \T_1(\phi) |g^\al| dv d\mu_{H_\rho} & \lesssim & \int_{H_{\rho}} \frac{\varepsilon^{1/2}}{\rho t^{\frac{n}{2} -1}}\frac{t}{\rho}\chi_{1}(|g^\al|)(\rho)d\rho\\
         & \lesssim & \frac{\varepsilon^{1/2}}{\rho{}^{\frac{n}{2}}}\mathbb{E}[|g^\al| ]](\rho) .\\
\end{eqnarray*}
After integration in $\rho$, we then obtain that $g^\al$ satisfies the integral inequality
\begin{eqnarray*}
\mathbb{E}[|g^\al|](\rho)- \mathbb{E}[|g^\al|](1) &\lesssim & \int_1^{\rho }\varepsilon^{3/2} \delta^{-1} \rho'{}^{\delta }d \rho'  +  \int_1^{\rho} \frac{\varepsilon^{1/2}}{\rho'{}^{\frac{n}{2}}}\mathbb{E}[|g^\al| ](\rho') d\rho'\\
 & \leq &\varepsilon + C\left( \frac{\varepsilon^{3/2}}{\delta }\rho^{\delta } + \int_1^{\rho} \frac{\varepsilon^{1/2}}{\rho'{}^{\frac{n}{2}}}\mathbb{E}[|g^\al| ](\rho') d\rho'       \right)\\
\mathbb{E}[|g^\al|](\rho) & \leq & \varepsilon \rho^{\delta }\left(1+C\frac{\varepsilon^{1/2}}{\delta}\right) + C \int_1^{\rho} \frac{\varepsilon^{1/2}}{\rho'{}^{\frac{n}{2}}}\mathbb{E}[|g^\al| ](\rho') d\rho' 
\end{eqnarray*}
for some constant $C$ depending only on the dimension and the regularity. Gr\"onwall's lemma provides:
\begin{equation*}
\mathbb{E}[|g^\al|](\rho)  \leq  \varepsilon \rho^{\delta }\left(1+C\frac{\varepsilon^{1/2}}{\delta}\right) \cdot \exp\left(C\varepsilon^{1/2} \int_1^{\rho} \frac{d\rho' }{\rho'^{\frac{n}{2}}}\right)
 \end{equation*}
 that is, if $\delta = \varepsilon^{1/4}$, there exists a constant $\tilde{C}$,
 $$
 \mathbb{E}[|g^\al|](\rho)  \leq (1+ \tilde{C} \varepsilon^{1/4}) e^{\tilde{C}\varepsilon^{1/2}}\varepsilon\rho^{\varepsilon^{1/4}}.
 $$
\end{proof}

We then consider the lower order derivatives of $f$ particularly since these low derivatives of $f$ are weighted in $v^0$ in the energy:
\begin{lemma} \label{lem:lowestgalpha}  Assume $\delta = \varepsilon^{1/4}$, and let $\al$ be a multi-index such that $|\al| \leq \lfloor N/2\rfloor $. Then, $(v^0)^2g^\al$ satisfies, for all $\rho \in [1,P]$:
 $$
 \mathbb{E}[|(v^0)^2g^\al|](\rho)  \leq   (1+ C\varepsilon^{1/4}) e^{ C\varepsilon^{1/2}} \varepsilon\rho^{\delta}
 $$
where $C$ is a constant depending only on the regularity and the dimension.  If, furthermore $n >4$, then $\delta$ can be vanishing.
\end{lemma}
\begin{proof} Let us compute first $\T_\phi\left((v^0)^2\right)$:
\begin{eqnarray*}
|\T_\phi\left((v^0)^2\right)| &=& |2v^0 T_{\phi}(v^0)|\\
&=& |2v^0 \left(\nabla^i \phi + \T_1(\phi)v^i\right) \pv {i} v^0| \\
&=& |2\left((v^0)^2 \T_1(\phi)- v^0\p {x^0}\phi\right)| \\
|\T_\phi\left((v^0)^2\right) |&\lesssim & |(v^0)^3 \partial \phi|.
\end{eqnarray*}
Using Proposition \ref{prop:hypwave}, one consequently obtains
$$
|\T_\phi((v^0)^2)|\lesssim (v^0)^3 \frac{\varepsilon^{1/2}}{\rho t^{\frac{n}{2} -1}}. 
$$
with  $(v^0)^2 g^\al$ satisfying the equation
$$
\T_\phi((v^0)^2 g^\al) = (v^0)^2T_{\phi} (g^\al)+ \T_\phi\left((v^0)^2\right)g^\al.
$$

Furthermore, since $|\al|\leq N/2$, the source terms the equation satisfied by $g^\al$ (cf. Lemma \ref{lem:galpha}) consequently are only of the form $\partial Z^\gamma$ or $Z^\gamma \phi$ with 
$$
 |\gamma| \leq |\al| \leq \left \lfloor \frac{N}{2}\right \rfloor \leq N-\frac{n+2}{2} \text{ since }N \geq n+2.
$$
As a consequence, for low derivatives acting on $f$, all the terms containing $\phi$ can be estimated pointwise.

The terms to be considered separately in this context are the same as in the proof of Lemma \ref{lem:highestgalpha}. The only difference is the term of Equation \eqref{eq:msvtp1}, which is absent,  since only low derivatives of $\phi$ appear in the expression of $g^\al$. The estimates which one obtains are listed below. The arguments to perform the estimates are the same as above:
\begin{eqnarray*}
\int_{H_\rho}\int_v |p^i_{\gamma \beta}|  |\partial_i Z^\gamma \phi| (v^0)^2|\widehat{Z}^\beta f | dv d\mu_{H_\rho}&\lesssim &\int_{H_\rho} t\cdot  \frac{\varepsilon^{1/2} }{\rho t^{\frac{n}{2}-1}}  \chi_1\left((v^0)^2|\widehat{Z}^\beta f | \right) d\mu_{H_\rho}\\
&\lesssim &\varepsilon^{3/2} \rho^{\delta +1-\frac{n}{2} }\\
\int_{H_\rho}\int_v |T_\phi(q_{\beta\gamma})|  | Z^\gamma \phi| (v^0)^2|\widehat{Z}^\beta f | dv d\mu_{H_\rho}&\lesssim &\int_{H_\rho}  \frac{\varepsilon^{1/2} \rho}{t^{\frac{n}{2}}} \cdot \frac{t}{\rho} \chi_1\left((v^0)^2|\widehat{Z}^\beta f | \right) d\mu_{H_\rho}\\
&\lesssim &\varepsilon^{3/2} \rho^{\delta +1-\frac{n}{2} }\\
 \int_{H_\rho}\int_v v^0 |q_{\beta\gamma}q_{\kappa \sigma}|  | Z^\gamma \phi| |\partial Z^\kappa \phi| (v^0)^2|\widehat{Z}^\sigma f | dv d\mu_{H_\rho}&\lesssim &\int_{H_\rho} t^2 \cdot  \frac{\varepsilon^{1/2} \rho}{t^{\frac{n}{2}}} \cdot \frac{\varepsilon^{1/2} }{\rho t^{\frac{n}{2}-1}} \cdot \frac{t}{\rho} \chi_1\left((v^0)^2|\widehat{Z}^\beta f | \right) d\mu_{H_\rho}\\
 &\lesssim &\varepsilon^2 \rho^{\delta +3-n }\\
   \int_{H_\rho}\int_v \frac{1}{v^0}   |\partial_t \phi| (v^0)^2|g^\al| dv d\mu_{H_\rho}&\lesssim &    \frac{\varepsilon^{1/2} }{\rho^{\frac{n}{2}} } \mathbb{E}[(v^0)^2|g^\al|](\rho)\\
      \int_{H_\rho}\int_v  \T_\phi\left((v^0)^2\right)|g^\al|  dv d\mu_{H_\rho}&\lesssim &\int_{H_\rho} \frac{\varepsilon^{1/2}}{\rho t^{\frac{n}{2}-1 }}\cdot \frac{t}{\rho}\chi_1\left((v^0)^2| g^\al  | \right)   d\mu_{H_\rho}\\
      &\lesssim  &   \frac{\varepsilon^{1/2} }{\rho^{\frac{n}{2}} } \mathbb{E}[(v^0)^2|g^\al|](\rho).
   \end{eqnarray*}
   Altogether $(v^0)^2g^\al$ satisfies the integral inequality:
   $$
   \mathbb{E}[|(v^0)^2 g^\al|] (\rho)-\mathbb{E}[|(v^0)^2g^\al|](1)\lesssim \int_1^{\rho}\varepsilon^{3/2} \rho'^{\delta +1-\frac{n}{2} } d\rho' +  \int_1^\rho \frac{\varepsilon^{1/2} }{\rho'{}^{\frac{n}{2}} } \mathbb{E}[(v^0)^2|g^\al|](\rho') d\rho'.
   $$
The conclusion is obtained in a similar fashion as to the end of the proof of Lemma \ref{lem:highestgalpha}.
\end{proof}

\begin{proposition} \label{prop:ml1est} Assume $\delta = \varepsilon^{1/4}$. The following inequalities holds, for all $\rho \in [1,P]$:
\begin{itemize}
\item if $n>4$,
 $$
E_N[f](\rho) \leq\frac{(1+ C\varepsilon^{1/4}) e^{C\varepsilon^{1/2}} \varepsilon\rho^{ \tilde{C}\varepsilon^{1/4}}   }{1-C\varepsilon^{1/2}}, 
$$
\item and, if  $n=4$,
 $$
E_N[f](\rho) \leq \frac{ (1+C\varepsilon^{1/4}) e^{C\varepsilon^{1/2}} \varepsilon\rho^{\tilde{C}\varepsilon^{1/4}}  }{1-C\varepsilon^{1/2}},
$$
\end{itemize}
where the constant $C$ is a constant depending on the dimension and the regularity, and the constant $\tilde{C}$ is $0$, when $n>4$, and $1$ if $n=4$. In particular, for $\varepsilon$ small enough, then, in dimension $n>4$, for all $\rho \in [1,P]$
$$
E_N[f](\rho)\leq \frac32 \varepsilon,
$$
and, in dimension $4$,
$$
E_N[f](\rho)\leq \frac32 \varepsilon\rho^{\varepsilon^{1/4}}.
$$
\end{proposition}
\begin{proof} We remind here the definition of $g^\al$:
$$
g^\al =  Z^\al f-\sum_{\substack{\ab{\gamma}+\ab{\beta}\leq\ab{\al}+1\\ 1 \le \vert\be\vert \\  1\le \vert \gamma \vert \le N-\frac{n+2}{2} } }  q_{\be\gamma}(v^i/v^0,t,x)Z^\gamma\phi\widehat  Z^\be f. 
$$
 By the second triangular inequality comes immediately
 \begin{eqnarray*}
 \mathbb{E}[|(v^0)^pg^\al|] \geq \big|  \mathbb{E}[|(v^0)^p \widehat{Z}^\al f|]   - \mathbb{E}\big[\big|\sum_{\substack{\ab{\gamma}+\ab{\beta}\leq\ab{\al}+1\\ 1 \le \vert\be\vert \\  1\le \vert \gamma \vert \le N-\frac{n+2}{2}} }  q_{\be\gamma}(v^i/v^0,t,x)Z^\gamma\phi (v^0)^p \widehat  Z^\be f \big| \big]  \big|,
 \end{eqnarray*}
 so that, using Lemma \ref{lem:hypwave1},
  \begin{eqnarray*}
 \mathbb{E}[|(v^0)^p \widehat{Z}^\al f|]   - \sum_{\substack{\ab{\gamma}+\ab{\beta}\leq\ab{\al}+1\\ 1 \le \vert\be\vert \\  1\le \vert \gamma \vert \le N-\frac{n+2}{2} } }  \mathbb{E}\big[\big| t\cdot \frac{C\varepsilon \rho }{t^{\frac{n}{2} } } (v^0)^p \widehat  Z^\be f \big| \big]  \big| \leq   \mathbb{E}[|(v^0)^pg^\al|], \\
  \mathbb{E}[|(v^0)^p \widehat{Z}^\al f|]   -C \varepsilon^{1/2} \rho^{2-\frac{n}{2}} \sum_{\substack{\ab{\gamma}+\ab{\beta}\leq\ab{\al}+1\\ 1 \le \vert\be\vert \\  1\le \vert \gamma \vert \le N-\frac{n+2}{2} } }  \mathbb{E}\big[\big| (v^0)^p \widehat  Z^\be f \big| \big]  \big| \leq   \mathbb{E}[|(v^0)^pg^\al|],
 \end{eqnarray*}
 for some constants $C$.

 We now split the sum above between $|\beta|=|\al| $ and $|\beta|<|\al|$ and sum over the multi-indices $|\al| \leq N$, taking $p=2$ for $|\al|\leq \lfloor \frac{N}{2}\rfloor$ and $p=0$ otherwise to build the energy $E_{N}[f]$.  One gets, using the bootstrap assumptions \eqref{eq:massivebptransport}, as well as Lemmata \ref{lem:highestgalpha} and \ref{lem:lowestgalpha}, for all $\rho$ in $[1,P]$,
 $$
  (1-C \varepsilon^{1/2}) E_N[f](\rho) \leq   (1+ \tilde{C} \varepsilon^{1/4}) e^{\tilde{C}\varepsilon^{1/2}}\varepsilon\rho^{\delta} + C\varepsilon^{3/2}  \rho^{\delta +2 -\frac{n}{2} },
 $$
 where $C$ is a constant (possibly different of the one above) depending only on the dimension and the regularity. Note finally that the $\rho$-loss is present only in dimension $4$.

 As a consequence, for all $\rho$ in $[1,P]$, if $n>4$, 
 $$
E_N[f](\rho) \leq\frac{\leq  (1+ \tilde{C} \varepsilon^{1/4}) e^{\tilde{C}\varepsilon^{1/2}} \varepsilon\rho^{\delta}   }{1-C\varepsilon^{1/2}}, 
$$
and, if  $n=4$,
 $$
E_N[f](\rho) \leq \frac{\leq  (1+ \tilde{C} \varepsilon^{1/4}) e^{\tilde{C}\varepsilon^{1/2}} \varepsilon\rho^{\delta}  }{1-C\varepsilon^{1/2}}. 
$$

 \end{proof}
\subsubsection{$L^2$-estimates for the transport equation}\label{sec:l2esttr}

Consider here the vector $X$ defined by
\eq{
X=(\widehat Z^{\al_1} f,\hdots, \widehat Z^{\al_q}f) \text{ with }|\al_1|\geq  \left\lfloor\frac{N}{2}\right\rfloor-n +1  \text{ and }|\al_q|=  N,
}
where the multi-index $\al$ goes over all the multi-indices of length larger that $\lfloor \frac{N}{2}\rfloor -n+1$. Using the same notation, we introduce the vector $G^h$:
$$
G^h=(g^{\al_1 },\hdots,  g^{\al_q}) \text{ with }|\al_1|\geq \left\lfloor \frac{N}{2}\right \rfloor-n+1  \text{ and }|\al_q|=  N
$$
where $g^\al$ has been defined by Equation \eqref{eq:defgalpha}. Consider finally the vector $H$ defined by
$$
H=(v^0 \widehat Z^{\al_1} f,\hdots,v^0 \widehat Z^{\al_q}f) \text{ with }|\al_1|=  0  \text{ and }|\al_q|\leq  \left\lfloor \frac{N}{2}\right \rfloor -n.
$$
In a similar fashion as above, let us now consider the vector $G^l$ defined by:
$$
G^l=(v^0g^{\al_1} ,\hdots,v^0 g^{\al_q}) \text{ with }|\al_1|=  0  \text{ and }|\al_q|\leq  \left\lfloor \frac{N}{2}\right \rfloor  -n.
$$

We will denote in the following by $|A|_{\infty}$ the supremum over all the components of a vector or matrix $A$. Recall that, if $A$ and $B$ are two matrices, then,
$$
|A B|_{\infty} \lesssim |A|_{\infty}|B|_{\infty}.
$$

All along this section, an inequality of the form 
$$
|A|\lesssim \frac{1}{1-\varepsilon^{1/2}}
$$
appears often, for some quantity $A$. Since we have assumed that 
$$
\varepsilon^{1/2}\leq \frac12,
$$
$|A|$ can be bounded by $2$, and we can ignore the dependency on the upper bound when $1-\varepsilon^{1/2}$ appears in the denominator. We can then write
$$
|A|\lesssim 1.
$$
In the same spirit, $e^{C\sqrt{\varepsilon}}$ is treated as a constant.

The relation between the vectors $X, H, G^j, G^l$ is now stated in the following lemma:
 \begin{lemma} \label{lem:relationgxh}Assume that $\varepsilon$ is sufficiently small. Then, the following relations between $G^h, G^l, H,X$ hold:
\begin{enumerate}
\item There exists a square matrix $A$ such that
\begin{itemize}
 \item $G^l =  H- A^l H$;
 \item $A$ satisfies:
 $$
 |A^l|_{\infty} \lesssim \frac{\varepsilon^{1/2} \rho }{ t^{\frac{n}{2} -1} };
 $$
 \item If $\varepsilon$ is small enough, then $1-A^l$ is invertible, and
 $$
   |(1 -A^l)^{-1}|_{\infty} \lesssim 1
 $$
\end{itemize}
\item There exist a square matrix $A'$, and a rectangular matrix $A''$ such that
\begin{itemize}
\item $G^h =  X -A'X -  A'' H$;
 \item $(A')$ and $(A'')$ satisfy:
 $$
 |A'_{ij}|_{\infty} \lesssim \frac{\varepsilon^{1/2} \rho }{ t^{\frac{n}{2} -1} } \text{ and }  |A''_{ij}|_{\infty} \lesssim \frac{\varepsilon^{1/2} \rho }{v^0 t^{\frac{n}{2} -1} };
 $$
\item If $\varepsilon$ is small enough, then $1-A'$ is invertible, and
 $$
   |(1 -A')^{-1}|_{\infty} \lesssim 1.
 $$
\end{itemize}
\end{enumerate}
\end{lemma}
\begin{proof} The proof of the algebraic relations between $G^l$, $G^h$, $X$, and $H$ is a direct consequence of the definition of $g^\al$, as stated in Equation \eqref{eq:defgalpha}. The components of $A$ are of the form $q_{\beta\gamma} Z^\gamma \phi$ with $\vert \gamma \vert\leq N-\lfloor\frac{n}{2}\rfloor-1$ (cf. Lemma \ref{lem:massivecommutation} for the definition of $q_{\beta \gamma}$). Since $q_{\beta \gamma}\lesssim t$, the decay estimates for $Z^\gamma \phi$ (cf. Lemma \ref{lem:hypwave1}), as well as the bootstrap assumptions \eqref{eq:massivebpwave}, provide the estimates on the components of $A^l$, $A'$, and $A''$. Standard algebraic manipulations ensure the invertibility of the square matrices $1-A^l$ and $1-A'$, and 
$$
|A^l|_{\infty} \lesssim \frac{1}{1-\varepsilon^{1/2}},\quad |A'|_{\infty} \lesssim \frac{1}{1-\varepsilon^{1/2}}. 
$$
\end{proof}

\begin{lemma}\label{lem:eqgh} The commutator relation of Lemma \ref{lem:galpha} can be rewritten as
$$
\T_{\phi} G^h+\overline{\mathbf A} X =  \overline{\mathbf B} H,
$$
where
\begin{itemize}
 \item $\overline{\mathbf A}= (\overline{A}_{ij})$ is square matrix, depending on $(t,x,v)$ whose components can be bounded by
 $$
 |\overline{A}_{ij}|\lesssim \frac{\varepsilon^{1/2} v^0}{\rho t^{\frac{n}{2}-1}} \text{ or }  |\overline{A}_{ij}|\lesssim \frac{\varepsilon^{1/2}}{v^0\rho t^{\frac{n}{2}-2}} \text{ or }  |\overline{A}_{ij}|\lesssim \frac{\varepsilon^{1/2}v^0\rho }{ t^{\frac{n}{2}}}  \text{ or }|\overline{A}_{ij}|\lesssim \frac{\varepsilon v^0}{t^{n-3}};
 $$
 in particular, for any regular distribution function $g$,
 $$
 \int_{H_\rho}\int_v    |\overline{A}_{ij}g| dv d\mu_{H_\rho}\lesssim \frac{\varepsilon^{1/2}}{\rho^{\frac{n}{2}-1}} \int_{H_\rho} \chi_1(|g|)d\mu_{H_\rho}.
 $$
 \item $\overline{\mathbf B}$ is rectangular matrix, depending on $(t,x,v)$ whose components can be bounded by
   $$
 |\overline{B}_{ij}|\lesssim   t |f_{ij}| \text{ with }||f_{ij}||_{L^2(H_\rho)} \lesssim \varepsilon^{1/2}.
 $$
\end{itemize}
\end{lemma}
\begin{proof}
The proof of this lemma consists essentially in rearranging the terms of the commutator formula stated in Lemma \ref{lem:galpha} and on the relation
$$
\T_{\phi}(v^0) = v^0(\T_1(\phi) - \partial_t\phi),
$$
which can be bounded by the mean of Lemma \ref{lem:hypwave1} by
$$
\vert \T_{\phi}(v^0)\vert  \lesssim \dfrac{(v^0)^2\varepsilon^{1/2} }{t^{\frac{n}{2}-1}\rho}.
$$
  We remind furthermore that, if $\widehat{Z}^\be f$ can be estimated pointwise, then $\partial Z^\gamma \phi$ should be estimated in energy, as explained in the proof of Lemma \ref{lem:highestgalpha}. The estimates which then follow from splitting are obtained in a similar as in the same proof as in Lemma \ref{lem:highestgalpha}.  To understand the components of the matrix $\overline{\mathbf{A}}$ and $\overline{\mathbf{B}}$, and the related estimates, we finally provide the following examples of components of these matrices:
  \begin{itemize}
   \item the matrix $\mathbf{A}$ contains the following terms (here in absolute value):
   \begin{eqnarray*}
   |r_{\beta\gamma}\T_{1}(Z^{\gamma}\phi)| \lesssim \frac{\varepsilon^{1/2} v^{0}}{\rho t^{\frac{n}{2} -2}} &\text { for }& |\gamma| \leq N-\frac{n+2}{2}\\
   \left|\frac{1}{v^0}p^{\delta}_{\gamma \beta }\partial_{\delta}(Z^{\gamma}\phi)\right| \lesssim \frac{\varepsilon^{1/2} }{v^0\rho t^{\frac{n}{2} -2}} &\text { for }& |\gamma| \leq N-\frac{n+2}{2}\\
   |\T_{\phi}(q_{\be \gamma}Z^\gamma \phi)|\lesssim \frac{\varepsilon^{1/2} v^0\rho}{t^{\frac{n}{2}}}&\text { for }& |\gamma|, |\kappa| \leq N-\frac{n+2}{2}\\
   |q_{\beta\gamma}q_{\kappa \sigma}Z^{\gamma}\phi\T_1(Z^\kappa \phi)| \lesssim \frac{\varepsilon v^0}{t^{n-3}}&\text { for }& |\gamma|, |\kappa| \leq N-\frac{n+2}{2}\\
   \end{eqnarray*}
   \item    the matrix $\mathbf{B}$ contains the following terms (here in absolute value):
   $$
   \frac{1}{v^0}q_{\beta \gamma}\T_1(Z^\gamma\phi) \text{ with }|\gamma |>N-\frac{n+2}{2}, 
   $$
   where $|q_{\beta \gamma}|\lesssim t$, and $\Vert (v^0)^{-1}\T_1(Z^\gamma\phi) \Vert^2_{L^2(H_\rho)}\lesssim \varepsilon $.
  \end{itemize}

\end{proof}

Consider  now the case when the operator $\T_\phi$ acts on the vector $G^l$:
\begin{lemma}\label{lem:eqgl} There exists a square matrix $\hat{\mathbf{A}}$ such that
$$
\T_{\phi} G^l = \mathbf{\hat{A}} G^l.
$$
The components $(\hat{A}_{ij})$ of the matrix $\hat{\mathbf{A}}$ satisfy
$$
 |\hat{A}_{ij}|\lesssim \frac{\varepsilon^{1/2} v^0}{\rho t^{\frac{n}{2}-1}} \text{ or }  |\hat A_{ij}|\lesssim \frac{\varepsilon^{1/2}}{v^0\rho t^{\frac{n}{2}-2}} \text{ or }  |\hat{A}_{ij}|\lesssim \frac{\varepsilon^{1/2}v^0\rho }{ t^{\frac{n}{2}}} \text{ or } |\hat A_{ij}|\lesssim \frac{\varepsilon^{1/2} v^0}{t^{n-3}};
 $$
 in particular, for any regular distribution function $g$,
 $$
 \int_{H_\rho}\int_v    |\hat A_{ij}g| dv d\mu_{H_\rho}\lesssim \frac{\varepsilon^{1/2}}{\rho^{\frac{n}{2}-1}} \int_{H_\rho} \chi_1(|g|)d\mu_{H_\rho}.
 $$
\end{lemma}
\begin{proof}This equation essentially relies on the commutator formula of Lemma \ref{lem:galpha}. Note that for this formula, since $\alpha$ is very low, the first term of the right hand side of the commutator formula
$$
\sum_{\substack{|\gamma|+|\beta| \le |\alpha|+1,\\ \vert \gamma \vert  \geq 1, \,|\be| \geq 1\\ \vert \gamma \vert\geq N-\lfloor\frac{n}{2}\rfloor   }}  q_{\beta \gamma}   \T_1(Z^\gamma \phi) \widehat{Z}^\beta(f)
$$
does not appear in the formula. Following the arguments of Lemma \ref{lem:lowestgalpha}, in the situation when the number of derivatives is low (smaller that $\lfloor\frac{N}{2}\rfloor-n$), then the derivatives of the wave equation can all be estimated pointwise. Proposition \ref{prop:hypwave} and Lemma \ref{lem:hypwave1} provide a relation of the form
$$
\T_{\phi} G^l =  \hat{\mathbf{A}}' H,
$$
where $\hat{\mathbf{A}}' $ satisfies the same properties as $\overline{\mathbf{A}}$ in Lemma \ref{lem:eqgh}. Finally, we use the relation between $G^l$ and $H$ stated in Lemma \ref{lem:relationgxh}.
\end{proof}

\begin{lemma}\label{lem:eqgfull}  There exists a square matrix $\mathbf A$ and a rectangular matrix $\mathbf{B}$ satisfying:
 \begin{itemize}
 \item $\mathbf A= (A_{ij})$ is square matrix, depending on $(t,x,v)$ whose components can be bounded by
 $$
 |A_{ij}|\lesssim \varepsilon^{1/2}\cdot \frac{ v^0}{\rho t^{\frac{n}{2}-1}} \text{ or }  |A_{ij}|\lesssim \varepsilon^{1/2} \cdot \frac{1}{v^0\rho t^{\frac{n}{2}-2}} \text{ or }  |{A}_{ij}|\lesssim \frac{\varepsilon^{1/2}v^0\rho }{ t^{\frac{n}{2}}}\text{ or } |A_{ij}|\lesssim \varepsilon\cdot \frac{ v^0}{t^{n-3}};
 $$
 in particular, for any regular distribution function $g$,
 $$
 \int_{H_\rho}\int_v    |A_{ij}g| dv d\mu_{H_\rho}\lesssim \frac{1}{\rho^{\frac{n}{2}-1}} \varepsilon^{1/2}\int_{H_\rho} \chi_1(|g|)d\mu_{H_\rho}.
 $$
 \item $\mathbf B$ is a rectangular matrix, depending on $(t,x,v)$ whose components can be bounded by
   $$
 |B_{ij}|\lesssim   t |f_{ij}| \text{ with }||f_{ij}||_{L^2(H_\rho)} \lesssim \varepsilon^{1/2}.
 $$
\end{itemize}
such that the vector $G^h$ satisfies the equation
$$
\T_{\phi} G^h+\mathbf A G^h =  \mathbf B G^l.
$$
\end{lemma}
\begin{proof} The proof relies on the combination of Lemmata \ref{lem:relationgxh}, \ref{lem:eqgh}, and \ref{lem:eqgl}. Assuming $\varepsilon$ small enough so that the matrices $1-A$, and $1-A'$, are invertible, and substituting the expression of $X$, and $H$, in function of $G^h$ and $G^l$, in the equation satisfied by $G^h$ stated in Lemma \ref{lem:eqgh}, one obtains the equation:
$$
\T_{\phi} G^h + \mathbf A\cdot(1-A')^{-1}\cdot G^h =  \mathbf{B} \cdot(1-A)^{-1}\cdot G^l +  \mathbf A\cdot(1-A')^{-1}\cdot A''\cdot G^l
$$
Since the components of the matrices $(1-A')^{-1}$ and $(1-A)^{-1}$ are both bounded by $\frac{1}{1-\varepsilon}$, the components of the matrices $ \mathbf A\cdot(1-A')^{-1}$ and $\mathbf{B} \cdot(1-A)^{-1}$ satisfy the same properties as the $\mathbf A$ and $\mathbf B$, up to constant $\frac{1}{1-\varepsilon}$. We finally consider the components of $C =\mathbf A\cdot(1-A')^{-1}\cdot A''$ whose components can be bounded by the means of Lemma \ref{lem:hypwave1}, as follows (we remind the reader that $A''$ contains a $(v^0)^{-1}$):
 $$
 |C_{ij}|\lesssim \frac{\varepsilon^{1/2}}{1-\varepsilon^{1/2}}\cdot \frac{ \varepsilon^{1/2} }{\rho t^{\frac{n}{2}-1}} \cdot\frac{\varepsilon^{1/2} \rho}{t^{\frac{n}{2} }} \text{ or }  |C_{ij}|\lesssim \frac{\varepsilon^{1/2}}{1-\varepsilon^{1/2}} \cdot \frac{\varepsilon^{1/2}}{\rho t^{\frac{n}{2}-2}} \cdot\frac{\varepsilon^{1/2} \rho}{t^{\frac{n}{2} }}  \text{ or }  |C_{ij}|\lesssim \frac{\varepsilon^{1/2}}{1-\varepsilon^{1/2}}\cdot \frac{ \varepsilon^{1/2} }{t^{n-3}} \cdot\frac{\varepsilon^{1/2} \rho}{t^{\frac{n}{2} }}.
 $$
One now easily notices that they can all be bounded by terms of the form
$$
  |C_{ij}| \lesssim t |f_{ij}| \text{ with } \Vert f_{ij}\Vert_{L^2(H_\rho)}\lesssim \frac{\varepsilon^{1/2}}{1-\varepsilon^{1/2}}.
$$
\end{proof}

\subsubsection*{Decomposition of the solution}

\noindent We consider now the set of equations of the form

\eq{\label{eq:G-full}
\T_{\phi} G^h+\mathbf A G^h=\mathbf B G^{l}, \text{ with initial data }G^h(\rho=1),
}
where all matrices are as above. To obtain estimates on the solution $G^h$, we split it into two parts
\eq{
G^h=G+G_{\mathsf{hom}},
} 
where $G_{\mathsf{hom}}$ and $G$  are respective solutions to the Cauchy problems with initial data on the hyperboloid $H_1$:
\begin{eqnarray}
 \T_{\phi} G_{\mathsf{hom}}+\mathbf A G_{\mathsf{hom}}=0 , &\text{ with }& G_{\mathsf{hom}}(\rho=1)= G^h(\rho=1)\label{eq:hom-part}\\
  \T_{\phi} G+\mathbf A G =\mathbf B G^{l} &\text{ with }&G(\rho=1)=0.
\end{eqnarray}
We proceed by evaluating both components individually.

\subsubsection*{Homogeneous part}
Commuting Equation \eqref{eq:hom-part} $n$-times with vector fields $\widehat{Z}$ yields
\eq{
\T_\phi \left(\widehat Z^\al G_{\mathsf{hom}}\right)=-\widehat{Z}^\al \left(\mathbf A G_{\mathsf{hom}}\right)+[\T_{\phi},\widehat{Z}^{\al}]G_{\mathsf{hom}}
}
for $\ab{\al}\leq n$. Applying estimate \eqref{eq:acltpmslm} directly to this equation would yield problematic terms in the estimate for the same reason as discussed after \eqref{eq:defgalpha}. As before, we introduce an auxiliary function $g_{\mathsf{hom}}^\alpha$ analogous to that defined in \eqref{eq:acltpmslm}. 
Applying estimate \eqref{eq:acltpmslm} to the transport equation of $g_{\mathsf{hom}}^\alpha$ yields

\eq{
\mathbb{E}[\ab{g_{\mathsf{hom}}^\alpha}](\rho)\leq \varepsilon(1+C\varepsilon)e^{C\sqrt{\varepsilon}}.
}
for $\ab{\al}\leq n$.
In turn, for $G_{\mathsf{hom}}$ we obtain for $n>4$:
 $$
\sum_{|\al| \leq n}\mathbb{E}[|\widehat{Z}^\al G_{\mathsf{hom}}|]  \leq \frac{\varepsilon (1+C\varepsilon^{1/2})e^{C\varepsilon^{1/2}} }{1-\varepsilon^{1/2}}, 
$$
 and, if  $n=4$:
 $$
\sum_{|\al| \leq n}\mathbb{E}[| \widehat{Z}^\al G_{\mathsf{hom}}|]  \leq\frac{\varepsilon (1+C\varepsilon^{1/2})e^{C\varepsilon^{1/2}}\rho^\delta }{1-\varepsilon^{1/2}},
$$
where the constant $C$ does depend only on the dimension and the regularity $N$.

This yields
\eq{
\Ab{G_{\mathsf{hom}}}_{\mathbb P,n}(\rho)\leq 
\begin{cases}
\frac{\varepsilon (1+C\varepsilon)e^{C\varepsilon} }{1-\varepsilon^{1/2}}& n>4\\
\frac{\varepsilon (1+C\varepsilon)e^{C\varepsilon}\rho^\delta }{1-\varepsilon^{1/2}}&n=4.
\end{cases}}
In combination with the Klainerman-Sobolev estimate \eqref{th:ksmsv} this implies the following lemma.

\begin{lemma}\label{lem:estbbghom}
The following estimate holds:
\eq{
\int_{v\in \R^{n}}\ab{G_{\mathsf{hom}}}\frac{dv}{v^0}\leq\begin{cases}
\frac{\varepsilon (1+C\varepsilon^{1/2})e^{C\varepsilon^{1/2}}}{(1+t)^n} & n>4\\
\\
\frac{\varepsilon (1+C\varepsilon^{1/2})e^{C\varepsilon^{1/2}}\rho^\delta }{(1+t)^n}&n=4,
\end{cases}}
where the constant $C$ does depend only on the dimension and the regularity $N$.
\end{lemma}

\subsubsection*{Inhomogeneous part}

Before giving the full proof of the actual $L^2$-decay estimates for the inhomogeneous part of $G^h$, let us explain the main ideas on a simple model problem. Assume that $T$ is a transport operator such as the relativistic transport operator or even just the classical one and that $f$ is a function of $(t,x,v)$ satisfying
$$T(f)= h g$$
where $h=h(t,x)$ is such that it is uniformly bounded in $L^2_x$ and such that $g$ is itself a solution to the free transport equation $T(g)=0$ with $g$ regular enough so that $L^1_{x,v}$-bounds hold for $g$ and decay estimates similar to our Klainerman-Sobolev inequality can be applied for the velocity averages of $g$. The aim is to prove $L^2_x$-decay estimates on $\int_{v} \vert f\vert dv$, the difficulty being that $h$ has very little regularity so that we cannot commute the equation. Instead, note that, by uniqueness, that $f=g H$, where $H$ is the solution to the inhomogeneous transport equation $T(H)=h$ with zero data. Indeed, $$T(gH)=T(g)H+gT(H)=gh,$$ since $T(g)=0$. Now, note that 
$$
\left|\left| \int_{v} gH dv \right|\right|_{L^2_x} \lesssim \left|\left| \left(\int_{v} |g| dv\right)^{1/2}\left( \int_{v} |g| H^2 dv \right)^{1/2}    \right|\right|_{L^2_x} \lesssim \left|\left| \left(\int_{v} |g| dv\right)^{1/2}\right|\right|_{L^\infty_x} \left|\left| \int_{v} |g|H^2 dv \right|\right|_{L^1_x}.
$$
Since we have assumed $g$ to solve the free transport equations and to be as regular as needed, we know that we have some decay for $\left|\left| \left(\int_{v} |g| dv\right)^{1/2}\right|\right|_{L^\infty_x}$. Thus, it remains only to prove boundedness for $|\left|\left| \int_{v} |g| H^2 dv \right|\right|_{L^1_x}$. This can be obtained using again the transport equation for $gH$ and the associated approximate conservation laws. Indeed, we have 
$$
T(gH^2)=2g h H 
$$
and thus, we need to estimate an integral of the form $\int_{t,x,v} |g h H| dt dx dv $. This is done as follows. First, 
\begin{eqnarray*} 
\int_{t,x,v} |g h H| dt dx dv &=& \int_{t} \int_{x,v} |g|^{1/2}|h|\,. \,|g|^{1/2} H dx dv dt\\
&\lesssim &\int_{t} \left(\int_{x,v} |g| |h|^2 dx dv\right)^{1/2} \left( \int_{x,v} |g| H^2 dx dv \right)^{1/2}dt.
\end{eqnarray*} 
It follows that, if one can obtain enough decay for $\left(\int_{x,v} |g|(x,v) |h|^2(x) dx dv\right)^{1/2}$, then the estimate can close via a Gr\"onwall type inequality. For the decay estimate, simply note again that 
$$\left|\int_{x,v} |g|(t,x,v) |h|^2(t,x) dx dv\right| \lesssim \left|\left| \int_{v} g dv \right|\right|_{L^\infty} \left|\left| h(t,x) \right|\right|_{L^2_x}.$$ 
This concludes the discussion of the estimates for the model problem. 
To estimate the inhomogeneous part of $G^h$, we will essentially follow this strategy requiring that 
\begin{itemize}
\item we need to work with systems;
\item the operator $T$ needs to be replaced by $T_\phi$ (or rather $T_\phi+A$); 
\item the vector $B$ replacing $h$ is not uniformly bounded in $L^2_x$ (there is a $t$-loss);
\item the vector replacing $g$ does not quite satisfy a homogeneous transport equation;
\item and finally, in all steps, we need to keep track of the exact decay rates in $\rho$ to make sure the time integrals converge.
\end{itemize}

To perform the estimates on the inhomogeneous part of the equation, we introduce the following approximation scheme: let $(G_0, \mathbf{ B}_0) = (G,\mathbf{B})$ initializing the approximation, and consider the sequences $(G_\ell)$ and $(\mathbf B_\ell)$ defined as solutions to the following equations (with vanishing initial data), for $\ell =1$,
\begin{equation}\label{eq-B-1}
\begin{array}{lccl}
 \T_\phi \mathbf B_1+\mathbf A \mathbf B_1=\mathbf B_0, & \text{with} & B_1(\rho=1) = 0\\
 \T_\phi G_1+\mathbf A G_1=-\mathbf B_1\hat{\mathbf A}G^l, & \text{with} & G_1(\rho=1) = 0
 \end{array}
\end{equation} 
and, for  $\ell >1$,
\begin{equation}\label{eq-B-2}
\begin{array}{lccl}
\T_\phi \mathbf B_\ell+\mathbf A \mathbf B_\ell=-\mathbf B_{\ell-1}\hat{\mathbf A},& \text{with} & B_\ell(\rho=1) = 0\\\
\T_\phi G_\ell+\mathbf A G_\ell=-\mathbf B_\ell\hat{\mathbf A}G^l, & \text{with} & G_\ell(\rho=1) = 0.
\end{array}
\end{equation}
Note that the second equation of \eqref{eq-B-1} and \eqref{eq-B-2} are identical.

Equations \eqref{eq-B-1} imply that 
\begin{eqnarray*}
 \T_{\phi}\left(\mathbf{ B}_1G^l + G_1  \right)&=&\T_{\phi}(\mathbf{B}_1)G^l +  \mathbf{B}_1 \T_{\phi} (G^l) +  \T_{\phi}(G_1)\\
 &=& \left( - \mathbf{AB}_1 + \mathbf{B}_0 \right) G^l + \mathbf{B}_1 \hat{\mathbf A} G^l - \mathbf{A}G_1 - \mathbf{B}_1 \hat{\mathbf A} G^l \\
  &=&- \mathbf{A}\left( \mathbf{B}_1 G^l +G_1\right) + \mathbf{B}_0  G^l .
\end{eqnarray*}
As a consequence, $\mathbf{ B}_1G^l + G_1$ satisfies Equation \eqref{eq:G-full} with zero initial data. By uniqueness of solutions to the Cauchy problem,
$$
G_0 = \mathbf{ B}_1G^l + G_1.
$$

In a similar fashion, for $\ell >1$, Equations \eqref{eq-B-2} imply that 
\begin{eqnarray*}
 \T_{\phi}\left(\mathbf{ B}_\ell G^l + G_\ell  \right)&=& (\T_{\phi} \mathbf{ B}_\ell) G^l + \mathbf{ B}_\ell\T_{\phi} (G^l) +  \T_{\phi}(G_\ell)\\
 &=& \left( - \mathbf{AB}_\ell - \mathbf{B}_{\ell-1}  \hat{\mathbf A }\right) G^l + \mathbf{B}_\ell \hat{\mathbf A} G^l - \mathbf{A}G_\ell - \mathbf{B}_\ell1 \hat{\mathbf A} G^l \\
   &=&- \mathbf{A}\left( \mathbf{B}_\ell G^l +G_\ell\right) - \mathbf{B}_{\ell-1}  \hat{\mathbf A } G^l .
\end{eqnarray*}
As a consequence, $\mathbf{ B}_\ell G^l + G_\ell$ satisfies the second equation of \eqref{eq-B-2} at the rank $\ell-1$ with zero initial data. By uniqueness of solutions to the Cauchy problem,
$$
G_{\ell-1} = \mathbf{ B}_\ell G^l + G_\ell.
$$

To sum up, $G$ can be written formally as
$$
G =  \sum_{\ell \geq 1 }\left(G_{\ell-1}-G_{\ell}\right) = \sum_{\ell \geq 1 }B_{\ell} G^l.
$$
The convergence is proven later. Eventually the individual summands will be decomposed schematically into products of the form $b g=\sqrt{b^2g}\sqrt{g}$, where $b$ and $g$ are components of $\mathbf B_\ell$ and $G^l$. In the following we set up an iteration to obtain estimates on norms of the form \eqref{ptw-norm-BBH},
which are necessary to estimate the first factor in these products. These are eventually estimated in Proposition \ref{prop:l2esttr}.

The corresponding equations for $\mathbf B_\ell$ are used in the following to obtain estimates for an energy to be introduced below. We introduce the notations:
\eq{\alg{
&\mathbf B_1=(B^1_{ij}), \quad \mathbf B_\ell =(B^{\ell}_{ij})\\
&\mathbf A=(A_{ij}),\quad \hat{\mathbf A}=(\hat A_{ij})\\
& G^l=(\hdots,g_i,\hdots),\quad \mathbf B_0=(B^0_{ij}).
}}

We introduce a pointwise norm by
\eq{\label{ptw-norm-BBH}
|\mathbf B_\ell\mathbf B_\ell G^l |_{\infty}\equiv\sum_{i,j,k,m,n}|B_{ij}B_{km}g_n|.
}
We define the corresponding energy by
\eq{
F_\ell\equiv  \mathbb{E}(|\mathbf B_\ell\mathbf B_\ell G^l |_{\infty}).
}
We rewrite equations \eqref{eq-B-1} and \eqref{eq-B-2} to obtain scalar equations for the components of $\mathbf B_1$ and $\mathbf B_\ell$. Those are given in the following:
\eq{\alg{
\T_\phi B^1_{ij}+A_i^k B^1_{kj}&=B^0_{ij}\\
\T_\phi B^\ell_{ij}+A_i^k B^\ell_{kj}&=-B_i^{\ell-1,k}\hat A_{kj} \quad\mbox{ for  }\,\ell>1.
}}
As an immediate consequence, we obtain the transport equations for the components of the pointwise norm \eqref{ptw-norm-BBH}. Those read
\eq{\alg{
\T_\phi \left(B^1_{ij}B^1_{km}g_n\right)&=\left(-A_i^aB^1_{aj}+B^0_{ij}\right)B^1_{km}g_n+B^1_{ij}\left(-A_k^aB^1_{am}+B^0_{km}\right)g_n+B^1_{ij}B^1_{km}\hat A_n^bg_b,\\
\T_\phi \left(B^\ell_{ij}B^\ell_{km}g_n\right)&=\left(-A_i^aB^\ell_{aj}-B^{\ell-1,a}_{i}\hat A_{aj}\right)B^\ell_{km}g_n+B^\ell_{ij}\left(-A_k^aB^\ell_{am}-B^{\ell-1,a}_{k}\hat A_{am}\right)g_n\\
&\quad+B^\ell_{ij}B^\ell_{km}\hat A_n^bg_b\quad\mbox{ for }\,\ell>1.
}}
As a consequence of these equations and the general energy estimate \eqref{eq:acltpmslm} we obtain the following energy estimates for the different iteration steps. Note that the initial data vanish.
\eq{\label{eq:recursion0}
\alg{
F_1(\rho)\lesssim \sum_{i,j,k,l,m}\int_1^\rho \int_{H_{\rho'}}\int_v & \Big\{\ab{A_i^a}\ab{B^1_{aj}B^1_{kl}g_m}+\ab{A_k^a}\ab{B^1_{ij}B^1_{al}g_m}+\ab{\hat A_m^b}\ab{B^1_ijB^1_{kl}g_b} \\
&+\sqrt{(B^0_{ij})^2g_m}\sqrt{(B^1_{kl})^2g_m}+\sqrt{(B^0_{kl})^2g_m}\sqrt{(B^1_{ij})^2g_m}\\
&+\left(\frac{|\partial_t\phi|}{v^0} + \T_1(\phi)\right)\ab{B^1_{ij}B^1_{kl}g_m}\Big \}dv d\mu_{\rho'}d\rho'
}}
For $\ell>1$,
\eq{\label{eq:recursioni}
\alg{
F_\ell(\rho)\lesssim \sum_{i,j,k,l,m}\int_1^\rho \int_{H_{\rho'}}\int_v &\Big\{\ab{A_i^a}\ab{B^\ell_{aj}B^\ell_{kl}g_m}+\ab{A_k^a}\ab{B^\ell_{ij}B^\ell_{al}g_m}+\ab{\hat A_m^b}\ab{B^\ell_ijB^\ell_{kl}g_b}\\
&+\ab{\hat A_{aj}}\ \sqrt{(B^{\ell-1,a}_{i})^2g_m}\sqrt{(B^\ell_{kl})^2g_m}+\ab{\hat A_{al}}\sqrt{(B^{\ell-1,a}_{k})^2g_m}\sqrt{(B^\ell_{ij})^2g_m}\\
&+\left(\frac{|\partial_t\phi|}{v^0} + \T_1(\phi)\right)\ab{B^\ell_{ij}B^\ell_{kl}g_m}\Big\}dv d\mu_{\rho'}d\rho'.
}}

We can state the final $L^2$-estimates on the solution of the transport equation, which holds under the bootstrap assumptions \eqref{eq:massivebpwave} and \eqref{eq:massivebptransport}:
\begin{lemma} \label{lem:estbbginh} Assume that $\varepsilon$ is sufficiently small and that $\delta= \varepsilon^\frac14$. Then, the following estimate holds: 
\begin{itemize}
 \item  in dimension $4$, for $\ell\geq 0$,
$$
\sup_{\rho\in[1,P]}\rho^{-\varepsilon^{1/4}} F_\ell (\rho)\lesssim \varepsilon^\ell;
$$
\item in dimension greater than $4$, for $\ell\geq 0$,
$$
\sup_{\rho\in[1,P]} F_\ell (\rho)\lesssim \left(\varepsilon^{\frac32}\right)^{ \ell} .
$$
\end{itemize}
\end{lemma}
\begin{remark} This lemma states in particular that the series $\sum_{\ell>0}\sum_{i,j,k,m,n} B^\ell_{ij}B^\ell_{km}g_n$ (where $i,j,k,m,n$ are the indices of the components of the matrices $\mathbf{B}^\ell$ and $G^l$) converges normally in the appropriate Banach space.
\end{remark}

\begin{proof} We first consider the term $F_0(\rho) = \mathbb{E}(|\mathbf B_0\mathbf B_0 G^l|_{\infty})$. The components of $\mathbf B_0$ are products of functions bounded by $\varepsilon^{1/2}$ in $L^2$ times $t$, according to Lemma \ref{lem:eqgfull}. The components of $G^l$ can be bounded pointwise by Theorem \ref{th:ksmsv} and Remark \ref{re:cchi2}. One consequently obtains:
$$
F_0(\rho) \lesssim \int_{H_{\rho}} \frac{\varepsilon \rho^{\delta}}{t^{n}}|\mathbf B_0|_{\infty}^2 d\mu_{H_\rho}\lesssim \varepsilon^2\rho^{\delta+2-n}.
$$

Equations \eqref{eq:recursion0} and \eqref{eq:recursioni} provide the following integral inequalities, together with the decay of the coefficients of the matrix $\mathbf{A}$ stated in Lemma \ref{lem:eqgfull}:
\begin{eqnarray}
 F_{1}(\rho)&\lesssim &\int_{1}^\rho\frac{\varepsilon^{1/2}}{\rho'{}^{\frac{n}{2}-1} } F_{1}(\rho') +  (F_{1}(\rho'))^{\frac12} (F_{0}(\rho'))^{\frac12} d \rho'\nonumber\\
  &\lesssim &\int_{1}^\rho\frac{\varepsilon^{1/2}}{\rho'{}^{\frac{n}{2}-1} } F_{1}(\rho') +  \varepsilon\rho^{\frac{\delta}{2}+1-\frac{n}{2}} (F_{1}(\rho'))^{\frac12}  d \rho'\nonumber\\
&\lesssim &\int_{1}^\rho\frac{\varepsilon^{1/2}}{\rho'{}^{\frac{n}{2}-1} } F_{1}(\rho') + \varepsilon^{3/4} \rho^{\frac{\delta}{2} +\frac12 - \frac{n}{4}} \cdot\varepsilon^{1/4}\rho^{\frac12 - \frac{n}{4}} (F_{1}(\rho'))^{\frac12} d \rho'\nonumber\\
   &\lesssim &\frac{ \varepsilon^{3/2}}{\frac{n}{2}-2-\delta}\left(1-  \rho^{\delta+2- \frac{n}{2}}\right)+ \int_{1}^\rho\frac{\varepsilon^{1/2}}{\rho'{}^{\frac{n}{2}-1} } F_{1}(\rho')  d\rho'.\label{eq:F1intineq}
    \end{eqnarray}
We continue with the case $n>4$.   
   From the previous estimate, by Gr\"onwall's lemma, one obtains
    $$
    F_{1}(\rho)\lesssim e^{C\varepsilon^{1/2}} \varepsilon^{3/2}\lesssim  \varepsilon^{3/2}.
    $$

    We now consider $F_\ell(\rho)$, for $\ell\geq 2$:
    \begin{eqnarray}
  F_{\ell}(\rho)&\lesssim &\int_{1}^\rho\frac{\varepsilon^{1/2}}{\rho'{}^{\frac{n}{2}-1} } F_{\ell}(\rho') + \frac{\varepsilon}{\rho'{}^{\frac{n}{2}-1} } (F_{\ell-1}(\rho'))^{\frac12} (F_{\ell}(\rho'))^{\frac12} d \rho'\nonumber\\
&\lesssim &\int_{1}^{\rho} \frac{\varepsilon^{3/2}}{\rho'{}^{\frac{n}{2}-1} }  F_{\ell-1}(\rho') d\rho'+  \int_{1}^\rho\frac{\varepsilon^{1/2}}{\rho'{}^{\frac{n}{2}-1} } F_{\ell}(\rho') d\rho'.\label{eq:Flintineq}
\end{eqnarray}

Still for $n>4$, Gr\"onwall's lemma provides
$$
F_{\ell}(\rho) \lesssim \varepsilon^{\frac32} e^{C\sqrt{\varepsilon}} \sup_{\rho \in [1,P]}F_{\ell-1}(\rho)\lesssim \varepsilon^{\frac32} \sup_{\rho \in [1,P]}F_{\ell-1}(\rho).
$$
The conclusion of the proof for $n>4$ is obtained by resorting to standard arguments for geometric sequences. \\

We consider now the case $n=4$. We are proving, by the mean of the bootstrap, the following estimates, for all $\ell\geq 1$:
\eq{\label{eq:intbootstrap}
F_{\ell}(\rho)\leq C \varepsilon^\ell \rho^{\varepsilon^{1/4}},
}
where $C>0$ is a constant depending only on the regularity and the dimension. Assume that the last estimates holds on $[1, P']$, where $P'<P$ is the maximal time where Equation \eqref{eq:intbootstrap} holds. In view of these assumptions, the above estimate for $F_1(\rho)$ takes the form, by inserting the bootstrap assumption \eqref{eq:intbootstrap} into Equation \eqref{eq:F1intineq}:
\begin{equation*}
F_1(\rho)\lesssim \varepsilon^{3/2}\int_1^\rho \rho'^{-1+\varepsilon^{1/4}} d \rho',
\end{equation*}
which yields
\begin{equation*}
F_1(\rho)\lesssim \varepsilon\varepsilon^{1/4}\rho^{\varepsilon^{1/4}},
\end{equation*}
and, in particular, improves the assumption for sufficiently small $\varepsilon$.

For $\ell>1$, we obtain from the estimate for $F_\ell$, by inserting the bootstrap assumption \eqref{eq:intbootstrap} into Equation \eqref{eq:Flintineq}:
\begin{equation*}
F_\ell(\rho)\lesssim \varepsilon^\ell\varepsilon^{1/2}\int_1^\rho \rho'{}^{-1+\varepsilon^{1/4}}d\rho'
\end{equation*}
so that
\begin{equation*}
F_\ell(\rho)\lesssim \varepsilon^{1/4}\varepsilon^\ell \rho{}^{\varepsilon^{1/4}},
\end{equation*}
which improves the bootstrap assumption and proves the lemma for the case $n=4$.

\end{proof}

We can finally summarize the estimates in the following proposition:
\begin{proposition}\label{prop:l2esttr} Assume that $\varepsilon$ is sufficiently small. Under the bootstrap assumptions \eqref{eq:massivebpwave} and \eqref{eq:massivebptransport}, the following estimates hold, for all multi-indices $\al$ such that $ \lfloor N/2 \rfloor -n+1  \leq |\al|\leq N $:
$$
 \int_{H_\rho} \frac{t}{\rho}\left(\int_{v}|\widehat{Z}^{\al} f| \frac{dv}{v^0}\right)^{2} d\mu_{H_\rho}\lesssim  \rho^{2\varepsilon^{1/4}-n}\varepsilon^2.
$$
\end{proposition}
\begin{proof} We first notice that, by Lemma \ref{lem:relationgxh}, for $\al$ such that $\frac{N}{2}+1 \leq |\al|\leq N$,
\begin{eqnarray*}
\left(\int_{H_\rho}\frac{t}{\rho}\left(\int_{v}|\widehat{Z}^{\al} f| \frac{dv}{v^0}\right)^{2}d\mu_{H_\rho}\right)^\frac12  & \lesssim & \frac{1}{1-\varepsilon^{1/2}} \left(\int_{H_\rho}\frac{t}{\rho}\left(\int_{v}|g^{\al} | \frac{dv}{v^0}\right)^{2}d\mu_{H_\rho}\right)^\frac12 \\
&\lesssim&\frac{1}{1-\varepsilon^{1/2}}\left\{\left(\int_{H_\rho} \frac{t}{\rho}\left(\int_v |G^\al_{\mathsf{hom}}| \frac{dv}{v^0}\right)^{2}d\mu_{H_\rho}\right)^{\frac12} \right.\\
&& + \left.\left(\int_{H_\rho}\frac{t}{\rho} \left(\int_v |G^\al| \frac{dv}{v^0}\right)^{2}d\mu_{H_\rho}\right)^{\frac12} \right\},
\end{eqnarray*}
where $G^\al_{\mathsf{hom}}$ and $G^\al$ are the respective components of $G_{\mathsf{hom}}$ and $G$ respectively.
The first term of this sum is estimated by the mean of Lemma \ref{lem:estbbghom}
$$
\left(\int_{H_\rho}\frac{t}{\rho} \left(\int_v |G^\al_{\mathsf{hom}}| \frac{dv}{v^0}\right)^{2}d\mu_{H_\rho}\right)^{\frac12} \lesssim \varepsilon \rho^\frac{\delta -n}{2}\left(\int_0^\infty \frac{y^{n-1}}{<y>^{2n}}dy \right)^{\frac12}.
$$
The second term of the sum is estimated as follows: let us denote by $G^\al$ the components of the vector $G$: we have
\eq{\label{eq:sum-est-4}
\left(\int_{H_\rho}\frac{t}{\rho}\left(\int_v\ab{G^\al}\frac{dv}{v^0}\right)^2d\mu_{H_\rho}\right)^{\frac12}\lesssim
\sum_{\ell\geq 1}\left(\int_{H_\rho}\frac{t}{\rho}\left(\int_v\ab{\mathbf{B}_{\ell,\al}^k G^l_k}\frac{dv}{v^0}\right)^2d\mu_{H_\rho}\right)^{\frac12}
}
by the triangle inequality.
In turn, for every $\ell\geq 1$ we have the following estimate
\eq{
\int_{H_\rho}\frac{t}{\rho}\left(\int_v\ab{\mathbf{B}_{\ell,\al}^k G^l_k}\frac{dv}{v^0}\right)^2d\mu_{H_\rho}\lesssim \sum_k\int_{H_\rho}\frac{t}{\rho}\left(\int_v\ab{G^l_k}\frac{dv}{v^0}\right)\left(\int_v\ab{(\mathbf{B}_{\ell,\al}^k)^2 G^l_k}\frac{dv}{v^0}\right)d\mu_{H_\rho},
}
where the last sum over $k$ is taken of all the components of $\mathbf{B}_\ell$ and $G^l$ and is consequently finite.
In combination with the pointwise decay for $G^l$ and the first estimate in Lemma \ref{lem:estbbginh} this implies
\eq{
\int_{H_\rho}\frac{t}{\rho}\left(\int_v\ab{\mathbf{B}_{\ell,\al}^k G^l_k}\frac{dv}{v^0}\right)^2d\mu_{H_\rho}\lesssim \rho^{2\varepsilon^{1/4}-n}\varepsilon^\ell \varepsilon.
}
This allows to take the sum in Equation \eqref{eq:sum-est-4} and yields
\eq{
\int_{H_\rho}\frac{t}{\rho}\left(\int_v\ab{G^\al}\frac{dv}{v^0}\right)^2d\mu_{H_\rho} \lesssim \rho^{2\varepsilon^{1/4}-n}\varepsilon^2.
}
In combination with the bound on the homogeneous part above this yields the claim.
\end{proof}

\newpage
\appendix 
\section{Distribution functions for massive  particles with compact support in $x$} \label{se:csmsv}
Theorems \ref{th:demsv} and \ref{th:vnm4d} require that the initial data be given on the initial hyperboloid $H_1$ instead of a more traditional $t=const$ hypersurface. In this appendix, we explain how we can go from the $t=0$ hypersurface to $H_1$ provided the initial data on $t=0$ has sufficient decay in $x$. For simplicity, consider the homogeneous massive transport equation with  initial data $f_0$ given at $t=0$. Assume that the support of $f_0$ is contained in the ball of radius $R$. Without loss of generality, we may translate the problem in time, so that we now consider the problem with data at time $t=\sqrt{R^2+1}$. 
\begin{eqnarray}
\T_m(f)&=&0, \\
f(t=\sqrt{R^2+1})&=&f_0.
\end{eqnarray}
Now, by the finite speed of propagation, the solution to this problem vanishes outside of the cone 
$$\begin{aligned}
\mathcal{C}(R)&\equiv\left\{(t,r,\omega)\,/\,t-r= \sqrt{R^2+1}-R,\,\omega \in \mathbb{S}^{n-1},\, t \ge \sqrt{R^2+1} \right\}\\
 &\quad\cup \left\{(t,r,\omega)\,/\, t+r= \sqrt{R^2+1}+R,\, \omega \in \mathbb{S}^{n-1},\, t \le \sqrt{R^2+1} \right\}
\end{aligned}
$$ 
depicted below. 
\begin{figure}[!h]
\center \includegraphics[width=12cm]{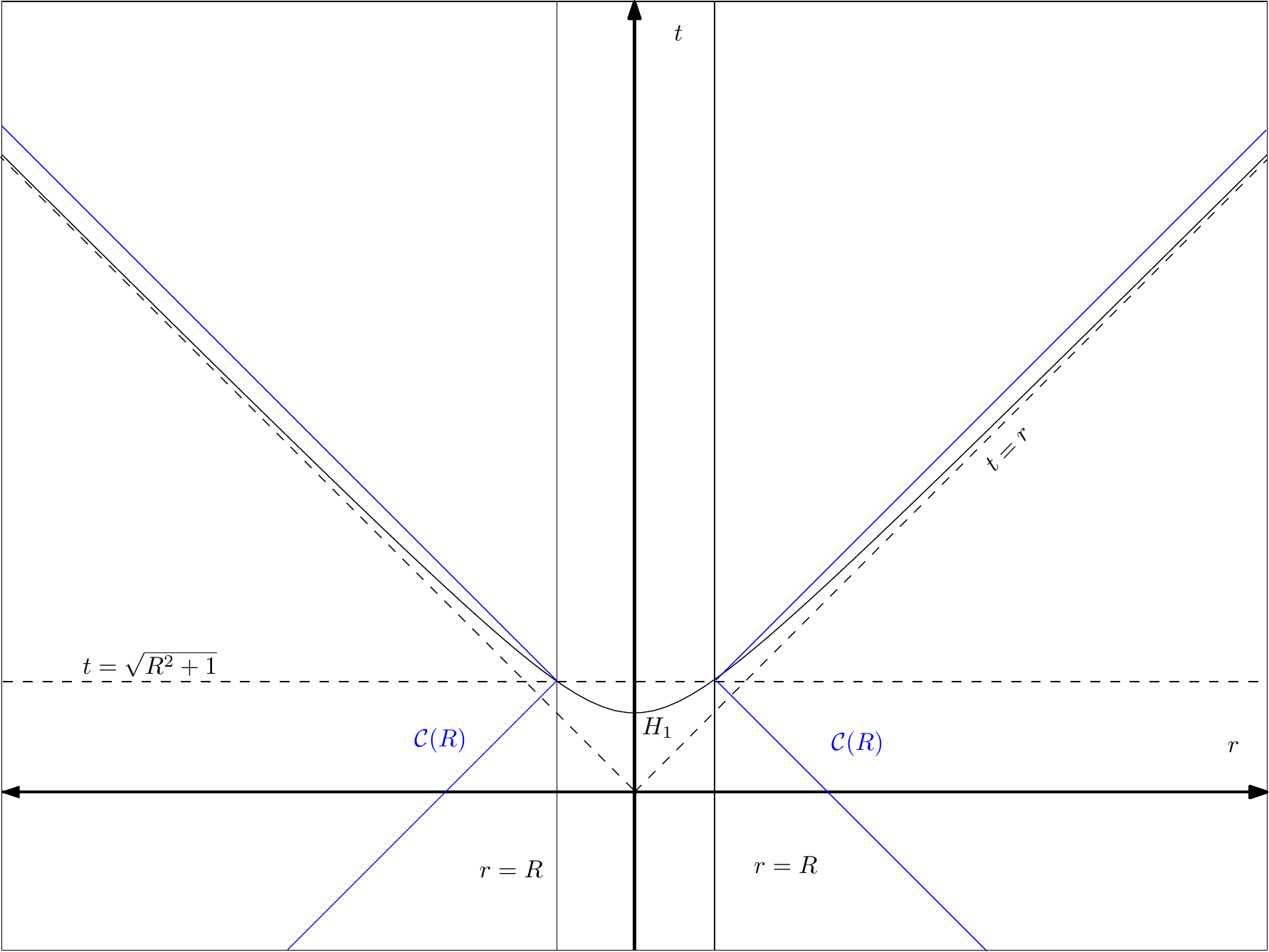}
\caption{The trace of a distribution function with compact support on $H_1$.}
 \end{figure}
Thus, the trace of $f$ on $H_1$ is compactly supported and as a consequence, the norm appearing on the right-hand side of Theorem \ref{th:demsv} is finite. Recall also that Theorem \ref{th:demsv} gives pointwise estimates for $t \ge \sqrt{1+|x|^2}$. On the other hand, the region $t < \sqrt{1+|x|^2}$ lies in the exterior of $\mathcal{C}(R)$ and thus $f(t,x)=0$ for $t < \sqrt{1+|x|^2}$. Thus, for compactly supported initial data given on some $t=const$ hypersurfaces, we can apply Theorem \ref{th:demsv} and obtain a $1/t^n$ decay uniformly in $x$. 

Finally, let us mention that the above arguments can be easily adapted to the non-linear massive Vlasov-Nordstr\"om system for small initial data. Thus, once again, the use of hyperboloids in \ref{th:vnm4d} is merely technical.

\section{Integral estimate} \label{app:ie}
\begin{lemma}\label{lem:sideineq} Let $n$ be a positive integer. Consider $\al, \be$ such that 
	$$
	\al +\be >n.
	$$
	There exists a constant $C_{\al,\be,n}$, such that the following estimates is true: for all $t > 0$, if $\be\neq 1$
	$$
 \int_{0}^\infty\dfrac{r^{n-1} dr }{(1+t+r)^{\al}(1+\vert t-r\vert)^\be}	\leq \dfrac{C_{\al,\be,n}}{t^{\al+\be-n}}\left(1+t^{\be-1}\right).
	$$
	If $\be =1$, then
		$$
		\int_{0}^\infty\dfrac{r^{n-1} dr }{(1+t+r)^{\al}(1+\vert t-r\vert)}	\leq \dfrac{C_{\al,n}}{t^{\al+1-n}}\left(1+\log(t+1)\right).
		$$
	
\end{lemma}

\begin{proof}
Let 
$$
A=\int_{0}^\infty\dfrac{r^{n-1} dr }{(1+t+r)^{\al}(1+\vert t-r\vert)^\be}. 
$$

First, let us make the change of variable
$$ 
r= ty.
$$
This gives
$$
A=\dfrac{1}{t^{\al+\be-n}} \int_{0}^\infty\dfrac{y^{n-1}dy}{\left(\frac{1}{t} +1 +y\right)^{\al}\left(\frac{1}{t} + \vert1-y \vert \right)^\be }.
$$
The first part of the denominator is bounded below by $(1 +y)^{\al}$, so that
$$
A \leq \dfrac{1}{t^{\al+\be-n}}\int_{0}^\infty\dfrac{y^{n-1}dy}{\left(1+y\right)^{\al}\left(\frac{1}{t} + \vert1-y \vert \right)^\be },
$$
We then cut the integral in $2$, at the value $r=2$. Let us thus introduce the constants $K_{\al,\be,n}$ by
$$
K_{\al,\be,n}= \int_2^\infty\dfrac{y^{n-1}dy}{\left(1+y\right)^{\al}\left(y-1\right)^\be }.
$$
$A$ can then be bounded by
$$
A\leq \dfrac{1}{t^{\al+\be-n}}\left(K_{\al,\be,n} +2^n\int_0^1\dfrac{dy}{\left(\frac{1}{t} +1-y\right)^\be} \right).
$$
The remaining integral can be computed: for $\beta\neq 1$,
$$
\int_0^1\dfrac{dy}{\left(\frac{1}{t} +1-y\right)^\be}=\dfrac{1}{\beta-1}\left(t^{\be-1}-\dfrac{1}{\left(\frac{1}{t}+1\right)^{\be-1}}\right)= \dfrac{ t^{\beta-1}}{1-\beta}\left(1 -\dfrac{1}{1+t^{\be-1}} \right)\leq C_\be t^{\be-1} .
$$
If $\be=1$, we get
$$
\int_0^1\dfrac{dy}{\left(\frac{1}{t} +1-y\right)} = \log(1+t).
$$
We finally get the announced result:
if $\be\neq 1$
$$
\int_{0}^\infty\dfrac{r^{n-1} dr }{(1+t+r)^{\al}(1+\vert t-r\vert)^\be}	\leq \dfrac{C_{\al,\be,n}}{t^{\al+\be-n}}\left(1+t^{\be-1}\right).
$$
If $\be =1$, then
$$
\int_{0}^\infty\dfrac{r^{n-1} dr }{(1+t+r)^{\al}(1+\vert t-r\vert)^1}	\leq \dfrac{C_{\al,\be,n}}{t^{\al+1-n}}\left(1+\log(t+1)\right).
$$
\end{proof}

\section{Geometry of Vlasov fields} \label{se:gvf}

In this section we present the necessary elements to understand the underlying geometry of Vlasov fields, on an arbitrary curved manifold. In particular, we will present, with some amount of details,
\begin{itemize}
	\item the geometry of the tangent bundle;
	\item the notion of complete lift, which is an essential
 tool to understand the commutators with the Vlasov field;
	\item how the ambient geometry of the tangent bundle can be reduced to the mass shell. 
\end{itemize} 
Most of the calculations will be left to the reader.

The reader who wishes to know more about the geometry of the tangent bundle can refer to the book by Crampin and Pirani \cite{Crampin:Ra2hzV13}. This section has also been widely inspired by the work of Sarbach \emph{et al.} \cite{Sarbach:2013vy,SZ14}.\\

All along this section, let $M$ be a $n+1$-dimensional smooth, oriented manifold, endowed with a Lorentzian metric $g$, of signature $(-,+,+,+)$. The Levi-Civita connection is denoted by $\nabla$. The tangent bundle of $M$ is denoted by $TM$. We furthermore assume that $M$ is time oriented: there exists a uniformly timelike vector field $T$ chosen, by convention, to be future pointing.

\subsection{Geometry of the tangent bundle}

This section is a reminder of some elementary geometric facts. 
\begin{definition}
	The tangent bundle of $M$ is the disjoint union of the tangent plane to $M$:
	$$
	TM = \amalg_{x\in M}T_xM. 
	$$
	$TM$ is a $GL_n(\mathbb{R})$-principal bundle of dimension $2n+2$ over $M$, with fibre $\mathbb{R}^{n+1}$, and projection given by
	$$
	\pi: V=(x,v)\in TM \longmapsto x\in M.
	$$
\end{definition}
Consider a chart $(U, x^\alpha)$. One defines on the open set $T U$ the following system of coordinates
$$
(x^\alpha, v ^\alpha = d x^\alpha). 
$$
This system of coordinates provides a local trivialization of $TM$, by
$$
V \in TM \longrightarrow (\pi(V), v^\alpha(V)) \in U\times \mathbb{R}^{n+1}. 
$$
In these coordinates, the metric reads
$$
g = g_{\alpha \beta} v^\alpha \otimes v^\beta. 
$$

We now consider the tangent space to the tangent space of $M$, denoted by $TTM$. If $V=(x,v)$ is a point of $TM$, the tangent space to $TM$ at the point $(x,v)$, denoted by $T_{(x,v)}TM$ is generated by the vectors
$$
\left\{\dfrac{\partial}{\partial x^\alpha}, \dfrac{\partial}{\partial v^\alpha}\right\}.
$$

Let $(x,v)$ be in $TM$. If $t\mapsto \sigma(t)$ is a curve on $M$, with $\sigma(0)=x, \dot{\sigma}(0)=v$, the \emph{natural lift} of $\sigma$ is the curve of $TM$ defined by 
$$
\sigma^h=t\mapsto \left(\sigma(t), \dot{\sigma}(t) \right). 
$$
As a consequence, any curve in $M$ can be obtained by projection on $M$ of a curve in $TM$. We consequently define:
\begin{definition}[Vertical space]
	The push forward $\pi_\star$ of the mapping $\pi$ defines, for $(x,v)\in TM$, a surjective mapping $T_{(x,v)}TM$ into $T_xM$. The kernel of $\pi_\star: T_{(x,v)}TM \rightarrow M$ is the horizontal space $V_{(x,v)}M$ at $(x,v)$. This is a subspace of dimension $n+1$ of $T_{(x,v)}TM$. In a system of coordinates $(x^\alpha, v^\alpha)$, it is generated by
	$$
	\left\{\dfrac{\partial }{\partial v^\alpha}\right\}.
	$$
\end{definition}
If $t\mapsto \sigma(t)$ is a curve on $M$, with $\sigma(0)=x, \dot{\sigma}(0)=v$, the \emph{horizontal lift} of $\sigma$ is the curve of $TM$ defined by 
$$
\sigma^h=t\mapsto \left(\sigma(t), V(t) \right), 
$$
where $V(t)$ is the vector field along the curve $\sigma$ obtained by parallely transporting $v$ along the curve $\sigma$. In the coordinates $(x^\alpha, v^\alpha)$, $V$ obeys the differential equation
$$
\nabla^\alpha_{\dot{\sigma} } V =  \dot{V}^\alpha + \Gamma^\alpha_{\beta\gamma}\dot{\sigma}^\beta V^\gamma =0, \text{with }V(0)=v,
$$
where the $\Gamma^\alpha_{\beta\gamma}$ are the Christoffel symbols of the connection. The tangent vector to the curve $\sigma$ is given, in the coordinates $(x^\alpha, v^\alpha)$, at $t=0$, by
$$
 \dot{\sigma}(0)^\alpha\dfrac{\partial }{\partial x^\alpha} + \dot{V}^\alpha(0)\dfrac{\partial }{\partial v^\alpha} = \dot{\sigma}(0)^\alpha\left(\dfrac{\partial }{\partial x^\alpha} -\Gamma^\beta_{\alpha \gamma} v^\gamma \dfrac{\partial }{\partial x^\beta} \right).
$$
The vector 
$$
\dot{\sigma}(0)^\alpha\left(\dfrac{\partial }{\partial x^\alpha} -\Gamma^\beta_{\alpha \gamma} v^\gamma \dfrac{\partial }{\partial x^\beta} \right),
$$
is the \emph{horizontal lift} of the vector $\dot{\sigma}(0)$. This definition depends only on the vector $v$ in $T_xM$.
\begin{definition}
 The horizontal space $H_{(x,v)}M$ at $(x,v)$ is the subspace of $T_{(x,v)}TM$ generated by the horizontal vectors
 $$
 e_\alpha=\dfrac{\partial }{\partial x^\alpha} -\Gamma^\beta_{\alpha \gamma} v^\gamma \dfrac{\partial }{\partial v^\beta}.
 $$
 It is independent of the chosen system of coordinates, and has trivial intersection with the vertical subspace of $T_{(x,v)}TM$.
\end{definition}

Finally, the tangent space is endowed with a metric:
\begin{definition}
	The Sasaki metric on the tangent bundle $TM$ is the metric of signature $(2n,2)$ defined, in coordinates, by
$$
g_s(e_\alpha, e_\beta) = g_s\left(\dfrac{\partial }{\partial v^\alpha},\dfrac{\partial }{\partial v^\beta}\right) =g_{\alpha\beta},
$$
and
$$
g_s\left(e_\alpha, \dfrac{\partial }{\partial v^\beta}\right)=0.
$$
\end{definition}

\subsubsection{Geodesic spray, and its commutators} 

We now turn our attention to the lift of geodesics to the tangent bundle.
\begin{definition} Let $\gamma$ be a geodesic with $\gamma(0)=x, \dot{\gamma}(0)=v$. The vector of $H_{(x,v)}M$ obtained by performing the horizontal lift of $v$, denoted by $T$, is given by
	$$
	T = v^\alpha e_\alpha =  v^\alpha \left(\dfrac{\partial }{\partial x^\alpha} -\Gamma^\beta_{\alpha \gamma} v^\gamma \dfrac{\partial }{\partial x^\beta} \right).
	$$ 
This defines globally a vector field on $TM$, called the geodesic spray.
\end{definition}

Contrary to the geodesic flow, this vector field is defined globally on the manifold. Furthermore, since its integral curves are the natural lift of geodesics, it naturally models the behaviour of freely falling particles in the context of general relativity. 

As we have seen earlier, one key aspect of the this work relies on the commutators with the transport operator $T$ (see section \ref{sec:commutators}). The right tool to understand is the notion of \emph{complete lift} (see section \ref{sec:completelift}). It can be introduced as follows. Consider a vector field $X$ on $M$. Assume that (locally) this vector field arises from a flow $\phi^t$:
$$
\dfrac{d \phi^t }{d t } = X(\phi^t).
$$
The mapping $\phi^t$ can naturally be lifted into a mapping of $TM$ by the formula
$$
\phi^t_\star = (\phi^t,d\phi^t).
$$
This immediately defines a vector field $\hat{X}$ on $TM$ by the formula:
$$
\dfrac{d \phi^t_\star }{d t } = \widehat{X}(\phi^t_\star).
$$
It is also possible to have a definition relying on Lie transport along curves. 
\begin{definition} Let $X$ be a vector field on $M$. The complete lift of the vector field $X$ on $M$ into a vector field $\hat X$ on $TM$ is done as follows: Let $p\in M$ and consider $X(p)$. Let $\gamma$ be an integral curve of $X$ with initial data:
$$
\left\{
\begin{array}{ccc}
\gamma(0) &=& p\\
\dfrac{\ud \gamma}{\ud s}(s)  &=& X(\gamma(s)).
\end{array}
\right.
$$
Let $v\in T_pM$, and consider the vector field $Y$ defined on $\gamma$ by Lie transporting the vector $v$ along $\gamma$. It obeys the equation
$$
\mathcal{L}_XY = [X,Y] = 0 \text{ with }Y(p) =v.
$$
This defines a curve $\Gamma = (\gamma, Y(\gamma))$ on $TM$. The mapping
$$
\begin{array}{ccc}
\widehat X:TM & \longrightarrow & TTM\\
(x,v)&\longmapsto & \dfrac{\ud\Gamma}{\ud s}(0)
\end{array}
$$ 
defines a vector field on $TM$, defined as being the complete lift of $X$ on $TM$.
\end{definition}
An expression of the complete lift of the vector $W = W^\alpha\partial_{x^{\alpha}} $ in adapted coordinates is given in (\cite[p.330]{Crampin:Ra2hzV13} and \cite[p.288]{Crampin:Ra2hzV13} for affine transformations) by
\begin{equation}\label{eq:completelift}
\widehat W = W^\alpha\dfrac{\partial}{\partial x^\alpha} + v^\beta \dfrac{\partial W^\alpha}{\partial x ^\beta} \dfrac{\partial}{\partial v^\alpha}.  
\end{equation}
This expression can also be written as
\begin{equation}\label{eq:completeliftconnexion}
\widehat W = W^\alpha e_\alpha + v^\beta \nabla_\beta W^\alpha \dfrac{\partial}{\partial v^\alpha}. 
\end{equation}
One of the main interests of the complete lift is its relation with the commutators of the geodesic spray. It is possible to give a precise characterization of the commutators with the geodesic spray which arise from vector fields on the base manifold (cf. \cite[Chapter 13, Section 6]{Crampin:Ra2hzV13} ).

\begin{theorem} A complete lift $\widehat X$ of a vector field is a symmetry of the geodesic spray $\Theta$, \emph{i.e.} commutes with the geodesic spray
$$
[\widehat X, T] =0,
$$
if, and only if, the vector field $X$ is an infinitesimal affine transformation of the corresponding affine connection, and satisfies the equation, for every vector fields $V, W$
$$
\mathcal{L}_X \nabla_V W = \nabla_{[X,V]}W +\nabla_V \mathcal{L}_X W. 
$$
\end{theorem}

In the presence of a metric, the commutator of a complete lift $\hat X$, of a vector field $X$, can be written explicitly (cf. \cite[Formula (74)]{Sarbach:2013vy}).
\begin{lemma}\label{lem:commutationgeodesicspray} Let $X$ be a vector field on $M$. The complete lift $\widehat X$ of $X$ commutes with the geodesic spray $\Theta$ if, and only if,
$$
[T, \widehat X ] = v^\alpha v^\beta \left[\nabla{}_\alpha \nabla{}_\beta X^\mu - R^\mu{}_{\beta \alpha\nu}X^\nu\right] \dfrac{\partial}{\partial v^\mu} =0.
$$
\end{lemma}

\begin{remark} 
The equation 
$$
\nabla{}_\alpha \nabla{}_\beta X^\mu - R^\mu{}_{\beta \alpha\nu}X^\nu =0 
$$
is the equation for Jacobi fields (cf. \cite[p.340]{Crampin:Ra2hzV13}).
\end{remark}

\subsection{Geometry of the mass shell}

If one considers a set of freely falling particles of given mass $m$, the $4$-velocity of such a particle satisfies
$$
g(v,v)=-m^2.
$$
It is consequently natural to consider the subset of the tangent bundle $TM$ defined by
$$
P_m=\{(x,v)\in TM \,|\, g_x(v,v) =-m^2,\, v \text{ is future oriented}\},
$$
called the \emph{mass shell}. This set is the phase space of the considered set of particles. When the mass $m$ is positive, $P_m$ is a smooth sub-manifold of $TM$. When the mass $m$ is vanishing, $P_m$ is no longer a smooth sub-manifold because of the singularity at the tip of the vertex. If we ignore this fact, $P_m$ is a principal bundle over $M$, with group structure $SO(n,1)$. The projection over $M$ is obtained by the restriction of the canonical projection of the bundle $TM$ over $M$. The fibre at a point $x$ is the subset of the tangent plane $T_xM$ given by
$$
\{v\in T_xM \,|\, g_x(v,v) =-m^2,\, v \text{ is future oriented}\}.
$$

Consider now a local chart on $M$, $(U, x^\alpha)$. We have seen that this local system of coordinates gives rise to a local chart on $TM$ given by $(TU, x^\alpha, v^\alpha = d x^\alpha)$. This system of coordinates gives rise to a system of coordinates $(\overline{x}^\alpha x^\al, \overline{v}^i=v^i)$ on the mass shell by eliminating $v^0$ in the equation
\eq{\label{eq:v0}
g_{\alpha\beta}v^\alpha v^\beta =-m^2.
}
After one has chosen this system of coordinates, it is necessary to derive the relations between the partial derivatives in the variables $(x^\al, v^\al)$, and the partial derivatives in the variables $(\overline{x}^\alpha= x^\al, \overline{v}^i=v^i)$. This is done by a simple application of the chain rule. Since $v^0$ does depend on the metric, it is first necessary to derive the following relations first: differentiating \eqref{eq:v0} brings
\begin{eqnarray}\label{eq:partialv0}
\dfrac{\partial v ^0}{\partial x ^\alpha } &=&-\dfrac{1}{2v _0}\frac{\partial g_{\beta \gamma}}{\partial x ^\alpha }v^\beta v^\gamma\\
\dfrac{\partial v ^0}{\partial \overline{v}^i} &=&-\dfrac{v_i}{v _0},
\end{eqnarray}
where we have used the notation 
$$
v_\alpha = g_{\al\be}v^\be.
$$
Consider now a smooth function $f$ on the tangent bundle $TM$. Its restriction to the mass shell is denoted by $\overline{f}$. An immediate application of the chain rule brings the following relations:
\begin{eqnarray}\label{eq:partialf1}
\dfrac{\partial \overline{f}}{\partial \overline{x} ^\alpha } &=& \dfrac{\partial f}{\partial x ^\alpha } +\dfrac{\partial v ^0}{\partial x ^\alpha }\dfrac{\partial f}{\partial v ^0} = \dfrac{\partial f}{\partial x ^\alpha }-\dfrac{1}{2v _0}\frac{\partial g_{\beta \gamma}}{\partial x ^\alpha }v^\beta v^\gamma\dfrac{\partial f}{\partial v ^0} \\
\dfrac{\partial \overline{f}}{\partial  \overline{v}^i }&=&  \dfrac{\partial f}{\partial v^i} + \dfrac{\partial v ^0}{\partial \overline{v}^i} \dfrac{\partial f}{\partial v^0}= \dfrac{\partial f}{\partial v^i}  -\dfrac{v_i}{v _0} \dfrac{\partial f}{\partial v^0}.\label{eq:partialf2}
\end{eqnarray}

The relations \eqref{eq:partialf1}, \eqref{eq:partialf2} can now be used to determine which vectors are tangent to the mass shell. We notice first that the vector fields $e_\alpha $, when applied to a function $f$, fulfil
\begin{eqnarray*}
	e_{\alpha}(f)  &=&  \dfrac{\partial f}{\partial x ^\alpha} - v ^\beta \Gamma_{\beta \alpha}^\gamma \dfrac{\partial f}{\partial v^\gamma}\\
	&=&\dfrac{\partial \overline{f}}{\partial \overline{x} ^\alpha } - v ^\beta \Gamma_{\beta \alpha}^i\dfrac{\partial \overline{f}}{\partial \overline{v}^i}\\
	&&-\left(-\dfrac{1}{2v _0}\frac{\partial g_{\beta \gamma}}{\partial x ^\alpha }v^\beta v^\gamma +  v ^\beta \Gamma^0_{\beta \alpha} +\dfrac{v_i}{v _0}v ^\beta \Gamma^i_{\beta \alpha}    \right)\dfrac{\partial f}{\partial v ^0}\\
	&=&\dfrac{\partial \overline{f}}{\partial \overline{x} ^\alpha }  - v ^\beta \Gamma_{\beta \alpha}^i\dfrac{\partial \overline{f}}{\partial \overline{v}^i}\\
	&&+\dfrac{1}{2v _0}\left(\frac{\partial g_{\beta \gamma}}{\partial x ^\alpha }v^\beta v^\gamma -2  v ^\beta v_\gamma \Gamma^\gamma_{\beta \alpha} \right)\dfrac{\partial f}{\partial v ^0}.
\end{eqnarray*}
A quick calculation shows, using the expression of the Christoffel symbols, that
$$
v ^\beta v _\gamma \Gamma^\gamma_{\beta \alpha} = \dfrac12 v ^\beta  v^\gamma\left( \frac{\partial g_{\beta \gamma}}{\partial x ^\alpha }+\frac{\partial g_{ \alpha\gamma}}{\partial x ^ \beta} - \frac{\partial g_{\beta \alpha }}{\partial x ^\gamma} \right) =\dfrac12 \frac{\partial g_{\beta \gamma}}{\partial x ^\alpha }v^\beta v^\gamma. 
$$
The expression of $e_\alpha$ is consequently
$$
e_{\alpha}(\overline{f}) =\dfrac{\partial \overline{f}}{\partial \overline{x} ^\alpha }  - v^\beta \Gamma_{\beta \alpha}^i\dfrac{\partial \overline{f}}{\partial \overline{v}^i}=e_{\alpha}(f). 
$$
This proves in particular that $e_\alpha $ is tangent to the mass shell, as well as the Liouville vector field
$$
T(f) = v ^\alpha e_{\alpha}(f)= v ^\alpha e_{\alpha}( \overline{f}) = v ^\alpha\dfrac{\partial \overline{f}}{\partial \overline{x} ^\alpha} -v ^\alpha v^\beta \Gamma_{\beta \alpha}^i\dfrac{\partial \overline{f}}{\partial \overline{v}^i}. 
$$
In dimension $n$, the mass shell is of dimension $2n+1$. We have as a consequence completely characterized the generators of the tangent plane to the mass shell, which is generated by the vectors
$$
\left \{ e_\alpha, \dfrac{\partial }{\partial \overline{v}^i}\right\}.
$$
From this, it is easy to deduce that the vector
$$
v^\alpha \dfrac{\partial }{\partial v^\alpha}
$$
is normal, for the Sasaki metric, to the mass shell (and also tangent in the massless case). The unit normal is consequently given by, for $m>0$,
\begin{equation}\label{eq:unitnormal}
N = \dfrac{-1}{m^2}v^\alpha\dfrac{\partial }{\partial v^\alpha} =  \dfrac{-1}{v_0} \dfrac{\partial }{\partial v^0} +\dfrac{1}{m^2}v^i\dfrac{\partial }{\partial \overline{v}^i}.
\end{equation}

	We will now discuss the conditions ensuring that a complete lift is tangent to the mass shell.
The same procedure based on Equations \eqref{eq:partialf1}, \eqref{eq:partialf2} can be applied to the complete lift of a vector field $X$:
\begin{eqnarray}
\widehat{ X} &=& X^\alpha e_\alpha + v^\beta \nabla_\beta X^i \dfrac{\partial}{\partial \overline{v}^i} +\dfrac{1}{v_0} v^\beta v^\gamma\nabla_\beta X_\gamma \dfrac{\partial}{\partial v^0}\\
&=& X^\alpha e_\alpha + v^\beta \nabla_\beta X^i \dfrac{\partial}{\partial \overline{v}^i} +\dfrac{1}{v_0} \pi^{(X)}_{\beta \gamma} v^\beta v ^\gamma\dfrac{\partial}{\partial v^0}\label{eq:completeliftcoordinates}.
\end{eqnarray}
We immediately get the following lemma:
\begin{lemma}\label{lem:completelift} 
	\begin{itemize}
		\item If $m>0$,  $X$ is Killing if, and only if, $\widehat{ X}$ is tangent to $P_m$.
		\item  If $m=0$, $X$ is conformal Killing if, and only if, $\widehat{ X}$ is tangent to $P_0$. 
	\end{itemize}
\end{lemma}
\begin{proof} The proof of this fact consists in noticing that 
	$$
	g_S(N, \widehat X ) = v^\mu v^\nu \nabla{}_{(\mu}X{}_{\nu)}. 
	$$
	Then, if $X$ is Killing in the massive case, or conformal Killing in the massless case,  
	$$
	g_S(N, \widehat X ) = 0, 
	$$
	and then $\widehat{X}$ is tangent to $P_m(x)$.
	
	Assume now that 
	$$
	g_S(N, \widehat{ X} ) = v^\mu v^\nu \nabla{}_{(\mu}X{}_{\nu)}= 0.
	$$
	Consider now the symmetric two form $\nabla{}_{(\mu}X{}_{\nu)}$ on the vertical space, which is endowed with the metric $g_{ij}$. Then, in the massless case, the symmetric two form $\nabla{}_{(\mu}X{}_{\nu)}$ vanishes on the light cone of $g_{ij}$ and is, as a consequence, proportional to it:
	$$
	\nabla{}_{(\mu}X{}_{\nu)} = \phi g_{\mu\nu},  
	$$
	\emph{i.e.} $X$ is conformal Killing. The conclusion in the massive case, follows in the same way.  
	\end{proof}

\bibliographystyle{alpha}
\bibliography{refs-2}

\end{document}